\documentclass[12pt,reqno]{amsart}

\usepackage[colorlinks=true,citecolor=blue,linkcolor=blue]{hyperref}
\usepackage{amsmath,amsthm,amssymb,amsfonts,mathrsfs,latexsym}
\usepackage{accents}
\usepackage{a4wide}
\usepackage{color}
\usepackage{amsaddr}
\newtheorem{theorem}{Theorem}[section]
\newtheorem{corollary}[theorem]{Corollary}

\newtheorem{lemma}[theorem]{Lemma}

\theoremstyle{definition}
\newtheorem{definition}[theorem]{Definition}
\newtheorem{remark}[theorem]{Remark}

\newcommand{\bbbr}{\mathbb R}

\newcommand{\bbbz}{\mathbb Z}
\newcommand{\bbbn}{\mathbb N}
\newcommand{\Fspqf}{\|f\|_{\dot{F}^s_{p,q}(\bbbr^n)}}

\newcommand{\Fspq}{\dot{F}^s_{p,q}(\bbbr^n)}
\newcommand{\Fspinf}{\dot{F}^s_{p,\infty}(\bbbr^n)}
\newcommand{\Bspq}{\dot{B}^s_{p,q}(\bbbr^n)}
\newcommand{\Bspinf}{\dot{B}^s_{p,\infty}(\bbbr^n)}
\newcommand{\Bsinfq}{\dot{B}^s_{\infty,q}(\bbbr^n)}
\newcommand{\Bsifif}{\dot{B}^s_{\infty,\infty}(\bbbr^n)}

\newcommand{\FT}{\mathcal{F}}
\newcommand{\iFT}{\mathcal{F}^{-1}}
\newcommand{\Haus}{\mathcal{H}}

\newcommand{\unitsph}{\mathbb{S}^{n-1}}
\newcommand{\HLmax}{\mathcal{M}}
\newcommand{\PFSmax}{\mathcal{P}}

\newcommand{\Sw}{\mathcal{S}}
\def\esssup{\mathop{\rm ess\,sup\,}}
\def\dist{{\rm dist\,}}
\def\mvint_#1{\mathchoice
          {\mathop{\vrule width 6pt height 3 pt depth -2.5pt
                  \kern -8pt \intop}\nolimits_{\kern -3pt #1}}%
          {\mathop{\vrule width 5pt height 3 pt depth -2.6pt
                  \kern -6pt \intop}\nolimits_{#1}}%
          {\mathop{\vrule width 5pt height 3 pt depth -2.6pt
                  \kern -6pt \intop}\nolimits_{#1}}%
          {\mathop{\vrule width 5pt height 3 pt depth -2.6pt
                  \kern -6pt \intop}\nolimits_{#1}}}

\begin{document}

% Place the running head in [], and the full title of the article in {}.
\title[Inequalities in function spaces]
% Running head is the full title or shortened version of the full title. This will appear at the top of odd pages. It should be no more than 40 characters to fit within the width limit.
{Inequalities in homogeneous Triebel-Lizorkin and Besov-Lipschitz spaces}% Only the first word and proper nouns should be capitalized.

\author{Lifeng Wang}
%\thanks{Department of Mathematics, 301 Thackeray Hall, University of Pittsburgh, Pittsburgh, PA 15260, USA}
%\address[Lifeng Wang]{Department of Mathematics, 301 Thackeray Hall, University of Pittsburgh, Pittsburgh, PA 15260, USA}
\address{Department of Mathematics, University of Pittsburgh, Pittsburgh, PA 15260, USA}
\subjclass[2020]{42B35.}
\keywords{Homogeneous Triebel-Lizorkin space, homogeneous Besov-Lipschitz space, iterated difference, Fourier analysis, Hardy-Littlewood maximal function, Peetre-Fefferman-Stein maximal function.}

\begin{abstract}
This paper provides equivalence characterizations of homogeneous Triebel-Lizorkin and Besov-Lipschitz spaces, denoted by \!$\Fspq$ and $\Bspq$ respectively, in terms of maximal functions of the mean values of iterated difference. It also furnishes the reader with inequalities in $\Fspq$ in terms of iterated difference and in terms of iterated difference along coordinate axes. The corresponding inequalities in $\Bspq$ in terms of iterated difference and in terms of iterated difference along coordinate axes are also considered. The techniques used in this paper are of Fourier analytic nature and the Hardy-Littlewood and Peetre-Fefferman-Stein maximal functions.
\end{abstract}

\maketitle
\tableofcontents

\section{Introduction}
We begin by introducing notations, definitions, historical results, and the main theorems of this paper. Let $\mathcal{S}(\bbbr^n)$ denote the space of Schwartz functions on $\bbbr^n$ and $\mathcal{S}'(\bbbr^n)$ be the space of tempered distributions on $\bbbr^n$. For a function $f\in L^1(\bbbr^n)$, we denote its $n$-dimensional Fourier transform by
$$\FT_n f(\xi)=\int_{\bbbr^n}f(x)e^{-2\pi ix\cdot\xi}dx,$$
and the $n$-dimensional inverse Fourier transform is denoted by
$$\iFT_n f(\xi)=\int_{\bbbr^n}f(x)e^{2\pi ix\cdot\xi}dx,$$
where for $x\in\bbbr^n$ and $\xi\in\bbbr^n$, $x\cdot\xi$ is the inner product. If $f\in\Sw'(\bbbr^n)$, we use the same notation to denote the $n$-dimensional distributional Fourier transform and its inverse. Now we introduce the definitions of related function spaces and maximal functions. We fix throughout this paper $\phi,\psi\in\mathcal{S}(\bbbr^n)$ such that
\begin{equation}\label{eq5}
0\leq\FT_n\psi(\xi)\leq 1\qquad\text{and}\qquad spt.\FT_n\psi\subseteq\{\frac{1}{2}\leq|\xi|<2\}
\end{equation}
and also
\begin{equation}
\sum_{j\in\bbbz}\FT_n\psi(2^{-j}\xi)=1 \qquad\text{if}\qquad \xi\neq 0,\label{eq37}
\end{equation}
then the function $\phi$ is defined in a way so that
\begin{equation}\label{eq345}
\FT_n\phi(\xi)=
\begin{cases}
\sum_{j\leq 0}\FT_n\psi(2^{-j}\xi) &\text{ if }\xi\neq 0,\\
1 &\text{ if }\xi=0.
\end{cases}
\end{equation}
then
\begin{equation}\label{eq346}
spt.\FT_n\phi(\xi)\subseteq\{\xi\in\bbbr^n:|\xi|\leq 2\}\text{ and }
\FT_n\phi(\xi)=1\text{ if }|\xi|\leq 1,
\end{equation}
furthermore we have the equality
\begin{equation}\label{eq347}
\FT_n\phi(\xi)+\sum_{j=1}^{\infty}\FT_n\psi(2^{-j}\xi)=1\text{ for all }\xi\in\bbbr^n.
\end{equation}
Define for $f\in\Sw'(\bbbr^n)$, $j\in\bbbz$ and $x,y\in\bbbr^n$, the function $f_j(x):=\psi_{2^{-j}}*f(x)$ where $\psi_{2^{-j}}(y)=2^{jn}\psi(2^jy)$ and thus we have the following decompositions:
\begin{equation}\label{eq38}
f=\sum_{j\in\bbbz}f_j,
\end{equation}
where the sum in (\ref{eq38}) converges in $\Sw'(\bbbr^n)/\mathscr{P}(\bbbr^n)$ and $\Sw'(\bbbr^n)/\mathscr{P}(\bbbr^n)$ is the space of tempered distributions modulo polynomials (cf. \cite[section 1.1.1]{14modern}), and
\begin{equation}\label{eq512}
f=f*\phi+\sum_{j=1}^{\infty}f_j,
\end{equation}
where the sum in (\ref{eq512}) converges in $\Sw'(\bbbr^n)$. Also due to the support condition of $\FT_n\psi$, we have the following
\begin{equation}
f_j(x)=\sum_{l=j-1}^{j+1}(f_j)_l(x)\qquad\text{for almost every }x\in\bbbr^n,\label{eq39}
\end{equation}
where $(f_j)_l=\psi_{2^{-l}}*f_j=\psi_{2^{-l}}*\psi_{2^{-j}}*f$. Given a sequence $\{f_k(x)\}_{k\in\bbbz}$ of functions defined on $\bbbr^n$ and $0<p,q\leq\infty$, we use the notation $\|\{f_k\}_{k\in\bbbz}\|_{L^p(l^q)}$ to denote the $\|\cdot\|_{L^p(\bbbr^n)}$-quasinorm of the $\|\cdot\|_{l^q}$-quasinorm of the sequence $\{f_k(x)\}_{k\in\bbbz}$ and the notation $\|\{f_k\}_{k\in\bbbz}\|_{l^q(L^p)}$ represents the $\|\cdot\|_{l^q}$-quasinorm of the sequence
$\{\|f_k\|_{L^p(\bbbr^n)}\}_{k\in\bbbz}$. Furthermore, $\esssup_{x\in\bbbr^n}|f_k(x)|$ denotes the essential supremum of the function $|f_k(x)|$ over $\bbbr^n$, that is, the least upper bound of $|f_k(x)|$ over $\bbbr^n$ except on a subset of $\bbbr^n$ of Lebesgue measure zero. Moreover, $\esssup_{k\in\bbbz}|f_k(x)|$ denotes the essential supremum of the sequence $\{|f_k(x)|\}_{k\in\bbbz}$ at $x\in\bbbr^n$, that is, the least upper bound of $\{|f_k(x)|\}_{k\in\bbbz}$ except on a subset of $\bbbz$ of counting measure zero, and in this sense $\esssup_{k\in\bbbz}|f_k(x)|=\sup_{k\in\bbbz}|f_k(x)|$.
\begin{definition}\label{definition1}
For $0<p<\infty$, $0<q\leq\infty$ and $s\in\bbbr$, the homogeneous Triebel-Lizorkin space $\Fspq$ as a subspace of the space $\Sw'(\bbbr^n)/\mathscr{P}(\bbbr^n)$ is
\begin{equation}\label{eq41}
\Fspq=\{f\in\Sw'(\bbbr^n)/\mathscr{P}(\bbbr^n):
\|f\|_{\Fspq}:=\|\{2^{ks}f_k\}_{k\in\bbbz}\|_{L^p(l^q)}<\infty\}.
\end{equation}
For $0<p\leq\infty$, $0<q\leq\infty$ and $s\in\bbbr$, the homogeneous Besov-Lipschitz space $\Bspq$ as a subspace of the space $\Sw'(\bbbr^n)/\mathscr{P}(\bbbr^n)$ is
\begin{equation}\label{eq42}
\Bspq\!=\!\{f\!\in\!\Sw'(\bbbr^n)/\mathscr{P}(\bbbr^n):
\|f\|_{\Bspq}\!:=\!\|\{2^{ks}f_k\}_{k\in\bbbz}\|_{l^q(L^p)}\!<\!\infty\}.
\end{equation}
\end{definition}
It is a well-known fact that the space $\Sw_0(\bbbr^n)$ of Schwartz functions that satisfy the condition
$$\int_{\bbbr^n}x^{\alpha}\varphi(x)dx=0\text{ for all multi-indices }\alpha$$
is dense in $\Fspq$ and $\Bspq$ when $0<p,q<\infty$ and $s\in\bbbr$, and this condition is also equivalent to the condition that all the derivatives of the Fourier transform $\FT_n\varphi$ equal to $0$ at the origin.
\begin{definition}\label{definition3}
For $0<p<\infty$, $0<q\leq\infty$ and $s\in\bbbr$, the inhomogeneous Triebel-Lizorkin space $F^s_{p,q}(\bbbr^n)$ as a subspace of the space $\Sw'(\bbbr^n)$ is
\begin{equation}\label{eq43}
F^s_{p,q}(\bbbr^n)\!=\!\{f\!\in\!\Sw'(\bbbr^n)\!:\!
\|f\|_{F^s_{p,q}(\bbbr^n)}\!:=\!\|\phi*\!f\|_{L^p(\bbbr^n)}\!+\!\|\{2^{ks}\!f_k\}_{k>0}\|_{L^p(l^q)}\!\!<\!\!\infty\}.
\end{equation}
For $0<p\leq\infty$, $0<q\leq\infty$ and $s\in\bbbr$, the inhomogeneous Besov-Lipschitz space $B^s_{p,q}(\bbbr^n)$ as a subspace of the space $\Sw'(\bbbr^n)$ is
\begin{equation}\label{eq44}
B^s_{p,q}(\bbbr^n)\!=\!\{f\!\in\!\Sw'(\bbbr^n)\!:\!
\|f\|_{B^s_{p,q}(\bbbr^n)}\!\!:=\!\!\|\phi*\!f\|_{L^p(\bbbr^n)}\!+\!\|\{2^{ks}\!f_k\}_{k>0}\|_{l^q(L^p)}\!\!<\!\!\infty\}.
\end{equation}
\end{definition}
Here we give the definition of iterated difference. Let $\bbbn$ denote the set of positive integers and $L\in\bbbn$. For a function $f$ defined on $\bbbr^n$ and $x,h\in\bbbr^n$ we define
\begin{equation}\label{eq259}
\varDelta^1_h f(x)=f(x+h)-f(x),\qquad(\varDelta^L_h f)(x)=\varDelta^1_h(\varDelta^{L-1}_h f)(x).
\end{equation}
It is not hard to prove by induction on $L$ that
\begin{equation}\label{eq260}
(-1)^{L+1}(\varDelta^L_h f)(x)=\sum^L_{j=1}d_j f(x+jh)-f(x),
\end{equation}
where
\begin{equation}\label{eq261}
\sum^L_{j=1}d_j=1\text{ and }d_j\in\bbbz\text{ for }1\leq j\leq L.
\end{equation}
When $f$ is in $\Sw(\bbbr^n)$, we have
\begin{equation}\label{eq341}
(\varDelta^1_h\FT_n f)(x)=\FT_n((e^{-2\pi ih\cdot\xi}-1)f(\xi))(x),
\end{equation}
and by iteration, we can obtain
\begin{equation}\label{eq342}
(\varDelta^L_h\FT_n f)(x)=\FT_n((e^{-2\pi ih\cdot\xi}-1)^L f(\xi))(x).
\end{equation}
In a similar way, we also have
\begin{equation}\label{eq343}
(\varDelta^L_h\iFT_n f)(x)=\iFT_n((e^{2\pi ih\cdot\xi}-1)^L f(\xi))(x),
\end{equation}
and therefore the following is true
\begin{equation}\label{eq344}
(\varDelta^L_h f)(x)=\iFT_n((e^{2\pi ih\cdot\xi}-1)^L\FT_n f(\xi))(x).
\end{equation}
Given $L\in\bbbn$, $h\in\bbbr^n$, $f\in\Sw'(\bbbr^n)$, $\varphi\in\Sw(\bbbr^n)$, observe the facts that $<\varDelta^L_h f,\varphi>=<f,\varDelta^L_{-h}\varphi>$ and both spaces $\Sw(\bbbr^n)$ and $\Sw_0(\bbbr^n)$ are closed under the operation $\varDelta^L_{-h}$, then (\ref{eq38}) and (\ref{eq512}) also suggest that
\begin{equation}\label{eq513}
\varDelta^L_h f=\sum_{j\in\bbbz}\varDelta^L_h f_j\qquad\text{in the sense of }\Sw'(\bbbr^n)/\mathscr{P}(\bbbr^n),
\end{equation}
and
\begin{equation}\label{eq514}
\varDelta^L_h f=\varDelta^L_h(f*\phi)+\sum_{j=1}^{\infty}\varDelta^L_h f_j\qquad\text{in the sense of }\Sw'(\bbbr^n).
\end{equation}
We introduce some maximal functions below. Let \!$\Haus^{n-1}$\! denote the \!$(n\!-\!1)$-dimensional Hausdorff measure on the unit sphere $\unitsph$ in $\bbbr^n$ and $A$ denote the annulus $A:=\{z\in\bbbr^n:1\leq|z|<2\}$. In this paper we denote the Lebesgue measures and integrals with respect to Lebesgue measures in the usual way, then $|\unitsph|=\Haus^{n-1}(\unitsph)$ and $|A|$ represent the corresponding surface measure and volume respectively. Also ``$\mvint_{}$'' is the mean value integral.
\begin{definition}\label{definition5}
For a function $f$ defined on $\bbbr^n$, let $t>0$, $r>0$, $x\in\bbbr^n$, $0\neq h\in\bbbr^n$ and $L\in\bbbn$, then the following maximal functions are defined
\begin{align}
(S^L_t f)(x)&=\esssup_{y\in\bbbr^n}|\mvint_{\unitsph}(\varDelta^L_{tz} f)(x-y)d\Haus^{n-1}(z)|\cdot(1+t^{-1}|y|)^{\frac{-n}{r}},
\label{eq348}\\
(V^L_t f)(x)&=\esssup_{y\in\bbbr^n}|\mvint_{A}(\varDelta^L_{tz} f)(x-y)dz|\cdot(1+t^{-1}|y|)^{\frac{-n}{r}},\label{eq349}\\
(D^L_h f)(x)&=\esssup_{y\in\bbbr^n}|(\varDelta^L_h f)(x-y)|\cdot(1+\frac{|y|}{|h|})^{\frac{-n}{r}}.\label{eq350}
\end{align}
\end{definition}
Now we propose the first pair of equivalence characterizations for homogeneous $\Fspq$ and $\Bspq$ spaces. We recall in the theory of tempered distributions, the integral of an equivalence class is justified when the equivalence class has a function representative and integrals of distinct representatives in the same equivalence class are identified. The same principle also applies to the supremum of an equivalence class. Integrals and supremums included in the main theorems and their corollaries below are all considered this way.
\begin{theorem}\label{theorem4}
Let $n\geq 2$, $0<p<\infty$, $0<q\leq\infty$, $f\in\Fspq$ has a function representative and assume $L\in\bbbn$, $s\in\bbbr$ satisfy $\frac{n}{\min\{p,q\}}<s<L$, then for every $r$ as in Definition \ref{definition5} satisfying $\frac{n}{s}<r<\min\{p,q\}$, the following five quasinorms are equivalent quasinorms in $\Fspq$,
\begin{align}
&\|\{2^{ks}S^L_{2^{-k}}f\}_{k\in\bbbz}\|_{L^p(l^q)},\label{eq351}\\
&\|\{2^{ks}\esssup_{0<\tau<2} S^L_{\tau 2^{-k}}f\}_{k\in\bbbz}\|_{L^p(l^q)},\label{eq352}\\
&\|\{2^{ks}V^L_{2^{-k}}f\}_{k\in\bbbz}\|_{L^p(l^q)},\label{eq353}\\
&\|\{2^{ks}\esssup_{0<\tau<2} V^L_{\tau 2^{-k}}f\}_{k\in\bbbz}\|_{L^p(l^q)},\label{eq354}\\
&\|\{2^{ks}\esssup_{0<|h|<2} D^L_{2^{-k}h}f\}_{k\in\bbbz}\|_{L^p(l^q)}.\label{eq355}
\end{align}
\end{theorem}
\begin{theorem}\label{theorem5}
Let $n\geq 2$, $0<p\leq\infty$, $0<q\leq\infty$, $f\in\Bspq$ has a function representative and assume $L\in\bbbn$, $s\in\bbbr$ satisfy $\frac{n}{p}<s<L$, then for every $r$ as in Definition \ref{definition5} satisfying $\frac{n}{s}<r<p$, the following five quasinorms are equivalent quasinorms in $\Bspq$,
\begin{align}
&\|\{2^{ks}S^L_{2^{-k}}f\}_{k\in\bbbz}\|_{l^q(L^p)},\label{eq356}\\
&\|\{2^{ks}\esssup_{0<\tau<2} S^L_{\tau 2^{-k}}f\}_{k\in\bbbz}\|_{l^q(L^p)},\label{eq357}\\
&\|\{2^{ks}V^L_{2^{-k}}f\}_{k\in\bbbz}\|_{l^q(L^p)},\label{eq358}\\
&\|\{2^{ks}\esssup_{0<\tau<2} V^L_{\tau 2^{-k}}f\}_{k\in\bbbz}\|_{l^q(L^p)},\label{eq359}\\
&\|\{2^{ks}\esssup_{0<|h|<2} D^L_{2^{-k}h}f\}_{k\in\bbbz}\|_{l^q(L^p)}.\label{eq360}
\end{align}
\end{theorem}
\noindent Equivalence characterizations of the same type for inhomogeneous Triebel-Lizorkin and Besov-Lipschitz spaces were obtained by H. Triebel in \cite[section 2.5.9]{1983functionspaces} with unnatural restrictions $s\geq G_p$, $M>2G_p+s$, $s\geq G_{pq}$, $M>2G_{pq}+s$ where $G_p=n+3+\frac{3n}{p}$ and $G_{pq}=n+3+\frac{3n}{\min\{p,q\}}$. Theorems \ref{theorem4} and \ref{theorem5} are both new and present optimal conditions on $s$ and $L$ and extend the ranges of $\tau$ and $|h|$ under the supremums in H. Triebel's result. The proof of Theorem \ref{theorem4} can be found in section \ref{proof.of.theorem4} and the proof of Theorem \ref{theorem5} is given in section \ref{proof.of.theorem5}.

This paper also furnishes the reader with a theorem of inequalities in the homogeneous Triebel-Lizorkin space $\Fspq$ in terms of iterated difference, which includes the case $0<q<1$. The theorem was proven by using Fourier analytic techniques and we now state this result below. Let
\begin{equation}\label{eq361}
\sigma_{pq}=\max\{0,n(\frac{1}{\min\{p,q\}}-1)\},\qquad
\tilde{\sigma}_{pq}=\max\{0,n(\frac{1}{p}-\frac{1}{q})\}.
\end{equation}
\begin{theorem}\label{theorem2}
Let $L\in\bbbn$, $0<p<\infty$, $0<q\leq\infty$, $s\in\bbbr$ and $f\in\Fspq$.\\
(i)If $0<p,q<\infty$, $\tilde{\sigma}_{pq}<s<L$, then
\begin{equation}\label{eq266}
\|(\int_{\bbbr^n}|h|^{-sq}|\varDelta^L_{h}f|^q\frac{dh}{|h|^n})^{\frac{1}{q}}\|_{L^p(\bbbr^n)}\lesssim\|f\|_{\Fspq}.
\end{equation}
(ii)Suppose $f$ has a function representative. If $0<p<\infty$, $0<q<1$ and $\sigma_{pq}+\tilde{\sigma}_{pq}<s<\infty$, or if $0<p<\infty$, $1\leq q<\infty$ and $-n<s<\infty$, then
\begin{equation}\label{eq415}
\|f\|_{\Fspq}\lesssim\|(\int_{\bbbr^n}|h|^{-sq}|(\varDelta^L_{h}f)(\cdot)|^q\frac{dh}{|h|^n})^{\frac{1}{q}}\|_{L^p(\bbbr^n)}.
\end{equation}
(iii)If $0<p<\infty$, $q=\infty$ and $\frac{n}{p}<s<L$, then
\begin{equation}\label{eq309}
\|\esssup_{h\in\bbbr^n}\frac{|\varDelta^L_{h}f|}{|h|^s}\|_{L^p(\bbbr^n)}\lesssim\|f\|_{\Fspinf}.
\end{equation}
(iv)Suppose $f$ has a function representative. If $0<p<\infty$, $q=\infty$ and $-n<s<\infty$, then
\begin{equation}\label{eq318}
\|f\|_{\Fspinf}\lesssim\|\esssup_{h\in\bbbr^n}\frac{|(\varDelta^L_{h}f)(\cdot)|}{|h|^s}\|_{L^p(\bbbr^n)}.
\end{equation}
\end{theorem}
The proof of Theorem \ref{theorem2} can be found in section \ref{proof.of.theorem2}. We believe Theorem \ref{theorem2} (ii) and (iv) are new and better than existing results. Theorem \ref{theorem2} (i) shows the term
\begin{equation}\label{eq3}
\|(\int_{\bbbr^n}|h|^{-sq}|\varDelta^L_{h}f|^q\frac{dh}{|h|^n})^{\frac{1}{q}}\|_{L^p(\bbbr^n)}
\end{equation}
may not be independently defined for tempered distributions, since the iterated difference $\varDelta^L_{h}f$ may not have a function representative if $f$ is a member of $\Sw'(\bbbr^n)$. Another example is that if $P(x)=x^{\alpha}$ is a polynomial function and we put it into (\ref{eq3}) then the resulting term may not have a finite value. However if we consider $P$ as a tempered distribution in $\Fspq$ and the conditions of Theorem \ref{theorem2} (i) are met, then (\ref{eq448}) designates the function representative of $\varDelta^L_{h}P$ is given by $\sum_{j\in\bbbz}\varDelta^L_{h}(\psi_{2^{-j}}*P)(x)=0$ for all $x\in\bbbr^n$. This is because $(\psi_{2^{-j}}*P)(x)$ can be expressed as a linear combination, with coefficients depending on $x$, of derivatives of the Fourier transform $\FT_n\psi$ evaluated at $0$ and these evaluations are identically zero due to the support condition of $\FT_n\psi$. We believe Theorem \ref{theorem2} (i) extends the definition of the term (\ref{eq3}). The same discussion is also true for Theorem \ref{theorem2} (iii).

In \cite[section 2.5.10]{1983functionspaces}, H. Triebel gave an equivalence characterization theorem of the inhomogeneous Triebel-Lizorkin space $F^s_{p,q}(\bbbr^n)$ in terms of iterated difference, where the conditions for the parameters are $0<p<\infty$, $0<q\leq\infty$ and $\frac{n}{\min\{p,q\}}<s<M$. In \cite[Theorem 1 pp. 102]{Stein1961}, E. M. Stein gave the equivalence characterization
$$[f]_{W^{\alpha}_{p,2}(\bbbr^n)}+\|f\|_{L^p(\bbbr^n)}\sim\|f\|_{L^p_{\alpha}(\bbbr^n)}$$
where the restrictions $0<\alpha<1$, $1<p<\infty$ and $\frac{2n}{n+2\alpha}<p<\infty$ were considered essentially sharp. And for $1<p<\infty$, $\alpha\in\bbbr$ and $f\in\Sw'(\bbbr^n)$, the inhomogeneous Sobolev norm of $f$ (cf. \cite[section 1.3.1]{14modern}) is defined to be
\begin{equation}\label{eq45}
\|f\|_{L_{\alpha}^p(\bbbr^n)}:=\|\iFT_n((1+|\xi|^2)^{\alpha/2}\FT_n f)\|_{L^p(\bbbr^n)},
\end{equation}
where $\|\cdot\|_{L^p(\bbbr^n)}$ is considered via duality. Since the inhomogeneous Triebel-Lizorkin space satisfies $L^p_{\alpha}(\bbbr^n)\sim F^{\alpha}_{p,2}(\bbbr^n)$ if $1<p<\infty$ and since
$$[f]_{W^{\alpha}_{p,2}(\bbbr^n)}=
\|(\int_{\bbbr^n}|h|^{-2\alpha}|(\varDelta^1_{h}f)(\cdot)|^2\frac{dh}{|h|^n})^{\frac{1}{2}}\|_{L^p(\bbbr^n)},$$
we consider E. M. Stein's result is an improvement of H. Triebel's. Furthermore in \cite[Theorem 1 pp. 393]{seeger1989note}, A. Seeger provided an equivalence characterization for the homogeneous anisotropic space
$$\|f\|_{\dot{F}^{\alpha}_{p,q}(\bbbr^n)}\sim\|S^{\alpha}_{q,r,m}f\|_{L^p(\bbbr^n)}$$
where $0<p<\infty$, $0<q\leq\infty$, $m>\alpha/a_0$, $r\geq 1$ with
$$\alpha>\max\{0,\nu(\frac{1}{p}-\frac{1}{r}),\nu(\frac{1}{q}-\frac{1}{r})\},$$
and
$$S^{\alpha}_{q,r,m}f(x)=(\int_0^{\infty}[\mvint_{\varrho(h)\leq t}
|(\varDelta^m_{h}f)(x)|^r dh]^{q/r}\frac{dt}{t^{1+\alpha q}})^{1/q}.$$
If we consider the isotropic case of A. Seeger's result, in which $\varrho(h)$ above can be deemed as $|h|$ and $a_0$ can be deemed as $1$, then by letting $r=q$ and changing the order of integration we can obtain
\begin{align*}
\|f\|_{\dot{F}^{\alpha}_{p,q}(\bbbr^n)}
&\sim\|(\int_0^{\infty}\int_{|h|\leq t}t^{-1-n-q\alpha}\cdot|(\varDelta^m_h f)(\cdot)|^q dhdt)^{\frac{1}{q}}\|_{L^p(\bbbr^n)}\\
&\sim\|(\int_{\bbbr^n}|h|^{-q\alpha}\cdot|(\varDelta^m_h f)(\cdot)|^q\frac{dh}{|h|^n})^{\frac{1}{q}}\|_{L^p(\bbbr^n)},
\end{align*}
for $0<p<\infty$, $1\leq q\leq\infty$ and $\max\{0,\nu(\frac{1}{p}-\frac{1}{q})\}<\alpha<m$. Recently in \cite[Theorem 1.2 pp. 693]{Prats19} M. Prats proves an equivalence characterization theorem of the inhomogeneous norm $\|f\|_{F^s_{p,q}(\Omega)}$ in terms of the sum of $\|f\|_{W^{k,p}(\Omega)}$ and
\begin{equation}\label{eq1}
\sum_{|\alpha|=k}\big(\int_{\Omega}\big(\int_{\Omega}
\frac{|D^{\alpha}f(x)-D^{\alpha}f(y)|^q}{|x-y|^{\sigma q+d}}
dy\big)^{\frac{p}{q}}dx\big)^{\frac{1}{p}}
\end{equation}
when parameters satisfy $1\leq p<\infty,1\leq q\leq\infty,s=k+\sigma$, $\max\{0,d(\frac{1}{p}-\frac{1}{q})\}<\sigma<1$ and $\Omega$ is a uniform domain in $\bbbr^d$. Furthermore, M. Prats also shows under the same conditions on parameters, the equivalence relation stands if (\ref{eq1}) is replaced by
\begin{equation}\label{eq2}
\sum_{|\alpha|=k}\big(\int_{\Omega}\big(\int_{\mathbf{Sh}(x)}
\frac{|D^{\alpha}f(x)-D^{\alpha}f(y)|^q}{|x-y|^{\sigma q+d}}
dy\big)^{\frac{p}{q}}dx\big)^{\frac{1}{p}}
\end{equation}
where $\mathbf{Sh}(x):=\{y\in\Omega:|y-x|\leq c_{\Omega}\delta(x)\}$ is the Carleson box centered at $x$, $\delta(x)=\dist(x,\partial\Omega)$ and $c_{\Omega}>1$ is a constant. Moreover when $1\leq q\leq p<\infty$, the set $\mathbf{Sh}(x)$ in (\ref{eq2}) can be improved and replaced by the Whitney ball $B(x,\rho\delta(x))$ for $0<\rho<1$.

Comparing Theorem \ref{theorem2} with H. Triebel's result in \cite[section 2.5.10]{1983functionspaces}, we find that if $0<q<1$ and $\frac{q}{q+1}\leq p<\infty$, then the restriction $\sigma_{pq}+\tilde{\sigma}_{pq}<s<L$ is better than the restriction of $s$ in H. Triebel's result. However if $0<p<\frac{q}{q+1}<q<1$ then we have
$$\frac{n}{\min\{p,q\}}<\sigma_{pq}+\tilde{\sigma}_{pq}$$
and the restrictions of H. Triebel's result remain better. If in addition the number $s$ also satisfies the condition
\begin{equation}\label{eq322}
s\leq n+\frac{n}{q},
\end{equation}
then $\sigma_{pq}+\tilde{\sigma}_{pq}<s$ implies $\frac{q}{q+1}<p$ and hence the restrictions in Theorem \ref{theorem2} are better than the restrictions in H. Triebel's result. This happens for sure when we pick $L=1$ or $L=2$ since $n+\frac{n}{q}>2$ for $0<q<1$. Therefore we formulate these two cases as corollaries below. When $0<p,q<\infty$, $s\in\bbbr$ and $f$ is a function, we define the generalized Gagliardo seminorm of $f$ (cf. \cite[section 2 pp. 524]{hitchhiker2012} for the usual Gagliardo seminorm) is
\begin{equation}\label{eq40}
[f]_{W^s_{p,q}(\bbbr^n)}:=\left(\int_{\bbbr^n}\left(\int_{\bbbr^n}\frac{|f(x)-f(y)|^q}{|x-y|^{n+sq}}dy \right)^{\frac{p}{q}}dx\right)^{\frac{1}{p}}.
\end{equation}
\begin{corollary}\label{corollary2}
Let $0<p<\infty$, $0<q\leq\infty$, $s\in\bbbr$ and $f\in\Fspq$ has a function representative.\\
(i)If $0<p,q<\infty$, $\tilde{\sigma}_{pq}<s<1$, then
\begin{equation}\label{eq464}
(\int_{\bbbr^n}(\int_{\bbbr^n}\frac{|f(x)-f(y)|^q}{|x-y|^{n+sq}}dy)^{\frac{p}{q}}dx)^{\frac{1}{p}}\lesssim\|f\|_{\Fspq}.
\end{equation}
(ii)If $0<p<\infty$, $0<q<1$ and $\sigma_{pq}+\tilde{\sigma}_{pq}<s<\infty$, or if $0<p<\infty$, $1\leq q<\infty$ and $-n<s<\infty$, then
\begin{equation}\label{eq466}
\|f\|_{\Fspq}\lesssim(\int_{\bbbr^n}(\int_{\bbbr^n}\frac{|f(x)-f(y)|^q}{|x-y|^{n+sq}}dy)^{\frac{p}{q}}dx)^{\frac{1}{p}}.
\end{equation}
(iii)If $0<p<\infty$, $q=\infty$ and $\frac{n}{p}<s<1$, then
\begin{equation}\label{eq465}
(\int_{\bbbr^n}\esssup_{y\in\bbbr^n}\frac{|f(x)-f(y)|^p}{|x-y|^{sp}}dx)^{\frac{1}{p}}\lesssim\|f\|_{\Fspinf}.
\end{equation}
(iv)If $0<p<\infty$, $q=\infty$ and $-n<s<\infty$, then
\begin{equation}\label{eq467}
\|f\|_{\Fspinf}\lesssim(\int_{\bbbr^n}\esssup_{y\in\bbbr^n}\frac{|f(x)-f(y)|^p}{|x-y|^{sp}}dx)^{\frac{1}{p}}.
\end{equation}
\end{corollary}
\begin{proof}[Proof of Corollary \ref{corollary2}]
Apply inequalities (\ref{eq266}), (\ref{eq415}), (\ref{eq309}) and (\ref{eq318}) with $L=1$ and use appropriate change of variable. In the case $1\leq q\leq\infty$, inequalities (\ref{eq466}) and (\ref{eq467}) are still true for $s=0$ or $s=1$. In particular, if we let $0<p\leq 1,q=2,s=0$ and apply the equivalence relation $\|\cdot\|_{\dot{F}^0_{p,2}(\bbbr^n)}\sim\|\cdot\|_{H^p(\bbbr^n)}$ in (\ref{eq466}), then we have
\begin{equation}\label{eq515}
\|f\|_{H^p(\bbbr^n)}\lesssim(\int_{\bbbr^n}(\int_{\bbbr^n}\frac{|f(x)-f(y)|^2}{|x-y|^{n}}dy)^{\frac{p}{2}}dx)^{\frac{1}{p}},
\end{equation}
where $\|\cdot\|_{H^p(\bbbr^n)}$ represents the Hardy quasinorm.
\end{proof}
\begin{corollary}\label{corollary1}
Let $0<p<\infty$, $0<q\leq\infty$, $s\in\bbbr$ and $f\in\Fspq$ has a function representative.\\
(i)If $0<p,q<\infty$, $\tilde{\sigma}_{pq}<s<2$, then
\begin{equation}\label{eq468}
(\int_{\bbbr^n}(\int_{\bbbr^n}\frac{|f(x)+f(y)-2f(\frac{x+y}{2})|^q}{|x-y|^{n+sq}}dy)^{\frac{p}{q}}dx)^{\frac{1}{p}}
\lesssim\|f\|_{\Fspq}.
\end{equation}
(ii)If $0<p<\infty$, $0<q<1$ and $\sigma_{pq}+\tilde{\sigma}_{pq}<s<\infty$, or if $0<p<\infty$, $1\leq q<\infty$ and $-n<s<\infty$, then
\begin{equation}\label{eq470}
\|f\|_{\Fspq}\lesssim
(\int_{\bbbr^n}(\int_{\bbbr^n}\frac{|f(x)+f(y)-2f(\frac{x+y}{2})|^q}{|x-y|^{n+sq}}dy)^{\frac{p}{q}}dx)^{\frac{1}{p}}.
\end{equation}
(iii)If $0<p<\infty$, $q=\infty$ and $\frac{n}{p}<s<2$, then
\begin{equation}\label{eq469}
(\int_{\bbbr^n}\esssup_{y\in\bbbr^n}\frac{|f(x)+f(y)-2f(\frac{x+y}{2})|^p}{|x-y|^{sp}}dx)^{\frac{1}{p}}\lesssim\|f\|_{\Fspinf}.
\end{equation}
(iv)If $0<p<\infty$, $q=\infty$ and $-n<s<\infty$, then
\begin{equation}\label{eq471}
\|f\|_{\Fspinf}\lesssim(\int_{\bbbr^n}\esssup_{y\in\bbbr^n}\frac{|f(x)+f(y)-2f(\frac{x+y}{2})|^p}{|x-y|^{sp}}dx)^{\frac{1}{p}}.
\end{equation}
\end{corollary}
\begin{proof}[Proof of Corollary \ref{corollary1}]
Apply inequalities (\ref{eq266}), (\ref{eq415}), (\ref{eq309}) and (\ref{eq318}) with $L=2$ and use appropriate change of variable. In the case $1\leq q\leq\infty$, inequalities (\ref{eq470}) and (\ref{eq471}) are still true for $s=0,1,2$. In particular, if we let $0<p\leq 1,q=2,s=0$ and apply the equivalence relation $\|\cdot\|_{\dot{F}^0_{p,2}(\bbbr^n)}\sim\|\cdot\|_{H^p(\bbbr^n)}$ in (\ref{eq470}), then we have
\begin{equation}\label{eq516}
\|f\|_{H^p(\bbbr^n)}\lesssim
(\int_{\bbbr^n}(\int_{\bbbr^n}\frac{|f(x)+f(y)-2f(\frac{x+y}{2})|^2}{|x-y|^{n}}dy)^{\frac{p}{2}}dx)^{\frac{1}{p}},
\end{equation}
where $\|\cdot\|_{H^p(\bbbr^n)}$ represents the Hardy quasinorm.
\end{proof}

In Theorem \ref{theorem2}, it was shown that the quasinorm $\|f\|_{\Fspq}$ is equivalent to
$$\|(\int_{\bbbr^n}|h|^{-sq}|\varDelta^L_{h}f|^q\frac{dh}{|h|^n})^{\frac{1}{q}}\|_{L^p(\bbbr^n)}=
\|(\int_{\unitsph}\int_0^{\infty}t^{-sq}|\varDelta^L_{t\theta}f|^q\frac{dt}{t}d\Haus^{n-1}(\theta))^{\frac{1}{q}}
\|_{L^p(\bbbr^n)},$$
and the expression inside the parenthesis on the right side is an integral of the term $\int_0^{\infty}t^{-sq}|\varDelta^L_{t\theta}f|^q\frac{dt}{t}$ over the set of all the unit directions $\theta\in\unitsph$. Therefore it is natural to ask: is there a similar equivalence relation if we replace the integral over $\unitsph$ by a finite sum of unit vectors? It seems this question can be answered when $\theta$ takes value in the set of elementary unit vectors $\{e_j\}_{j=1}^n$, where each $e_j\in\bbbr^n$ has its $j$-th coordinate equal to $1$ and all the other $(n-1)$ coordinates equal to $0$. We use the notation $\varDelta^L_{t,j}f=\varDelta^L_{te_j}f$ for $j\in\{1,\cdots,n\}$, then for example, when $f$ is a function defined on $\bbbr^n$, $\varDelta^1_{t,1}f(x)=f(x_1+t,x_2,\cdots,x_n)-f(x_1,x_2,\cdots,x_n)$. Denote
\begin{equation}\label{eq449}
\tilde{\sigma}^1_{pq}=\max\{0,\frac{1}{p}-\frac{1}{q}\},\qquad
\sigma_p=\max\{0,n(\frac{1}{p}-1)\}.
\end{equation}
\begin{theorem}\label{theorem6}
Let $L\in\bbbn$, $0<p<\infty$, $0<q\leq\infty$, $s\in\bbbr$ and $f\in\Fspq$.\\
(i)If $0<p,q<\infty$, $\tilde{\sigma}^1_{pq}<s<L$, then for each $j\in\{1,\cdots,n\}$
\begin{equation}\label{eq450}
\|(\int_0^{\infty}t^{-sq}|\varDelta^L_{t,j}f|^q\frac{dt}{t})^{\frac{1}{q}}\|_{L^p(\bbbr^n)}
\lesssim\|f\|_{\Fspq}.
\end{equation}
(ii)Suppose $f$ has a function representative. If $1<\min\{p,q\}$, $q<\infty$ and $s\in\bbbr$, or if $\min\{p,q\}\leq 1$, $q<\infty$ and $\sigma_{pq}+\tilde{\sigma}^1_{pq}<s<\infty$, then
\begin{equation}\label{eq451}
\|f\|_{\Fspq}\lesssim\sum_{j=1}^n
\|(\int_0^{\infty}t^{-sq}|\varDelta^L_{t,j}f(\cdot)|^q\frac{dt}{t})^{\frac{1}{q}}\|_{L^p(\bbbr^n)}.
\end{equation}
(iii)If $0<p<\infty$, $q=\infty$ and $\frac{1}{p}<s<L$, then for each $j\in\{1,\cdots,n\}$
\begin{equation}\label{eq562}
\|\esssup_{t>0}\frac{|\varDelta^L_{t,j}f|}{t^s}\|_{L^p(\bbbr^n)}\lesssim\|f\|_{\Fspinf}.
\end{equation}
(iv)Suppose $f$ has a function representative. If $1<p<\infty$, $q=\infty$ and $s\in\bbbr$, or if $0<p\leq 1$, $q=\infty$ and $\sigma_p+\frac{1}{p}<s<\infty$, then
\begin{equation}\label{eq563}
\|f\|_{\Fspinf}\lesssim\sum_{j=1}^n\|\esssup_{t>0}\frac{|\varDelta^L_{t,j}f(\cdot)|}{t^s}\|_{L^p(\bbbr^n)}.
\end{equation}
\end{theorem}
Theorem \ref{theorem6} is new and it is worth noting that the condition in Theorem \ref{theorem6} (i) and the condition in Theorem \ref{theorem6} (ii) when $1<\min\{p,q\}$ are completely independent of the dimension of the ambient space $\bbbr^n$ while the restriction of $s$ in Theorem \ref{theorem6} (ii) when $\min\{p,q\}\leq 1$ is only partially dependent on the dimension $n$. H. Triebel formulated the counterpart of Theorem \ref{theorem6} for the inhomogeneous $F^s_{p,q}(\bbbr^n)$ space in \cite[section 2.6.2]{1992functionspaces} with rough conditions $0<p<\infty$, $0<q\leq\infty$ and $\frac{n}{\min\{p,q\}}<s<M$. The proof of Theorem \ref{theorem6} can be found in section \ref{proof.of.theorem6}. The corollaries of Theorem \ref{theorem6} are given below.
\begin{corollary}\label{corollary3}
Let $0<p<\infty$, $0<q\leq\infty$, $s\in\bbbr$ and $f\in\Fspq$ has a function representative.\\
(i)If $0<p,q<\infty$, $\tilde{\sigma}^1_{pq}<s<1$, then for each $j\in\{1,\cdots,n\}$
\begin{equation}\label{eq592}
(\int_{\bbbr^n}(\int_0^{\infty}|f(x+te_j)-f(x)|^q\frac{dt}{t^{1+sq}})^{\frac{p}{q}}dx)^{\frac{1}{p}}
\lesssim\|f\|_{\Fspq}.
\end{equation}
(ii)If $1<\min\{p,q\}$, $q<\infty$ and $s\in\bbbr$, or if $\min\{p,q\}\leq 1$, $q<\infty$ and $\sigma_{pq}+\tilde{\sigma}^1_{pq}<s<\infty$, then
\begin{equation}\label{eq593}
\|f\|_{\Fspq}\lesssim\sum_{j=1}^n
(\int_{\bbbr^n}(\int_0^{\infty}|f(x+te_j)-f(x)|^q\frac{dt}{t^{1+sq}})^{\frac{p}{q}}dx)^{\frac{1}{p}}.
\end{equation}
(iii)If $0<p<\infty$, $q=\infty$ and $\frac{1}{p}<s<1$, then for each $j\in\{1,\cdots,n\}$
\begin{equation}\label{eq594}
(\int_{\bbbr^n}\esssup_{t>0}\frac{|f(x+te_j)-f(x)|^p}{t^{sp}}dx)^{\frac{1}{p}}\lesssim\|f\|_{\Fspinf}.
\end{equation}
(iv)If $1<p<\infty$, $q=\infty$ and $s\in\bbbr$, or if $0<p\leq 1$, $q=\infty$ and $\sigma_p+\frac{1}{p}<s<\infty$, then
\begin{equation}\label{eq595}
\|f\|_{\Fspinf}\lesssim\sum_{j=1}^n(\int_{\bbbr^n}\esssup_{t>0}\frac{|f(x+te_j)-f(x)|^p}{t^{sp}}dx)^{\frac{1}{p}}.
\end{equation}
\end{corollary}
\begin{proof}[Proof of Corollary \ref{corollary3}]
Apply inequalities (\ref{eq450}), (\ref{eq451}), (\ref{eq562}) and (\ref{eq563}) with $L=1$.
\end{proof}
\begin{corollary}\label{corollary4}
Let $0<p<\infty$, $0<q\leq\infty$, $s\in\bbbr$ and $f\in\Fspq$ has a function representative.\\
(i)If $0<p,q<\infty$, $\tilde{\sigma}^1_{pq}<s<2$, then for each $j\in\{1,\cdots,n\}$
\begin{equation}\label{eq596}
(\int_{\bbbr^n}(\int_0^{\infty}|f(x+2te_j)-2f(x+te_j)+f(x)|^q\frac{dt}{t^{1+sq}})^{\frac{p}{q}}dx)^{\frac{1}{p}}
\lesssim\|f\|_{\Fspq}.
\end{equation}
(ii)If $1<\min\{p,q\}$, $q<\infty$ and $s\in\bbbr$, or if $\min\{p,q\}\leq 1$, $q<\infty$ and $\sigma_{pq}+\tilde{\sigma}^1_{pq}<s<\infty$, then
\begin{equation}\label{eq597}
\|f\|_{\Fspq}\lesssim\sum_{j=1}^n
(\int_{\bbbr^n}(\int_0^{\infty}|f(x+2te_j)-2f(x+te_j)+f(x)|^q\frac{dt}{t^{1+sq}})^{\frac{p}{q}}dx)^{\frac{1}{p}}.
\end{equation}
(iii)If $0<p<\infty$, $q=\infty$ and $\frac{1}{p}<s<2$, then for each $j\in\{1,\cdots,n\}$
\begin{equation}\label{eq598}
(\int_{\bbbr^n}\esssup_{t>0}\frac{|f(x+2te_j)-2f(x+te_j)+f(x)|^p}{t^{sp}}dx)^{\frac{1}{p}}\lesssim\|f\|_{\Fspinf}.
\end{equation}
(iv)If $1<p<\infty$, $q=\infty$ and $s\in\bbbr$, or if $0<p\leq 1$, $q=\infty$ and $\sigma_p+\frac{1}{p}<s<\infty$, then
\begin{equation}\label{eq599}
\|f\|_{\Fspinf}\lesssim\sum_{j=1}^n(\int_{\bbbr^n}\esssup_{t>0}\frac{|f(x+2te_j)-2f(x+te_j)+f(x)|^p}{t^{sp}}dx)^{\frac{1}{p}}.
\end{equation}
\end{corollary}
\begin{proof}[Proof of Corollary \ref{corollary4}]
Apply inequalities (\ref{eq450}), (\ref{eq451}), (\ref{eq562}) and (\ref{eq563}) with $L=2$.
\end{proof}
Finally, as part of a systematic study, we also state and prove the counterparts of Theorem \ref{theorem6} and Theorem \ref{theorem2} and their corollaries for $\Bspq$ spaces.
\begin{theorem}\label{theorem7}
Let $L\in\bbbn$, $0<p,q\leq\infty$, $s\in\bbbr$ and $f\in\Bspq$.\\
(i)If $0<p\leq\infty$, $0<q<\infty$ and $0<s<L$, then for each $j\in\{1,\cdots,n\}$
\begin{equation}\label{eq600}
(\int_0^{\infty}t^{-sq}\|\varDelta^L_{t,j}f\|_{L^p(\bbbr^n)}^q\frac{dt}{t})^{\frac{1}{q}}
\lesssim\|f\|_{\Bspq}.
\end{equation}
(ii)Suppose $f$ has a function representative. If $1<p\leq\infty$, $1\leq q<\infty$ and $s\in\bbbr$, or if $1<p\leq\infty$, $0<q<1$ and $0<s<L$, or if $0<p\leq 1$, $0<q<\infty$ and $\sigma_p<s<L$, then
\begin{equation}\label{eq601}
\|f\|_{\Bspq}\lesssim\sum_{j=1}^n(\int_0^{\infty}t^{-sq}\|\varDelta^L_{t,j}f\|_{L^p(\bbbr^n)}^q\frac{dt}{t})^{\frac{1}{q}}.
\end{equation}
(iii)If $0<p\leq\infty$, $q=\infty$ and $0<s<L$, then for each $j\in\{1,\cdots,n\}$
\begin{equation}\label{eq602}
\esssup_{t>0}t^{-s}\|\varDelta^L_{t,j}f\|_{L^p(\bbbr^n)}\lesssim\|f\|_{\Bspinf}.
\end{equation}
(iv)Suppose $f$ has a function representative. If $1<p\leq\infty$, $q=\infty$ and $s\in\bbbr$, or if $0<p\leq 1$, $q=\infty$ and $\sigma_p<s<\infty$, then
\begin{equation}\label{eq603}
\|f\|_{\Bspinf}\lesssim\sum_{j=1}^n\esssup_{t>0}t^{-s}\|\varDelta^L_{t,j}f\|_{L^p(\bbbr^n)}.
\end{equation}
\end{theorem}
Theorem \ref{theorem7} is new. The proof of Theorem \ref{theorem7} is given in section \ref{proof.of.theorem7}. The counterpart of Theorem \ref{theorem7} for the inhomogeneous $B^s_{p,q}(\bbbr^n)$ space was obtained by H. Triebel in \cite[section 2.6.1]{1992functionspaces} with rough conditions $0<p,q\leq\infty$ and $\sigma_p<s<M$. The corollaries of Theorem \ref{theorem7} are formulated below.
\begin{corollary}\label{corollary5}
Let $0<p,q<\infty$, $s\in\bbbr$ and $f\in\Bspq$ has a function representative.\\
(i)If $0<p<\infty$, $0<q<\infty$ and $0<s<1$, then for each $j\in\{1,\cdots,n\}$
\begin{equation}\label{eq672}
(\int_0^{\infty}(\int_{\bbbr^n}|f(x+te_j)-f(x)|^p dx)^{\frac{q}{p}}\frac{dt}{t^{1+sq}})^{\frac{1}{q}}\lesssim\|f\|_{\Bspq}.
\end{equation}
(ii)If $1<p<\infty$, $1\leq q<\infty$ and $s\in\bbbr$, or if $1<p<\infty$, $0<q<1$ and $0<s<1$, or if $0<p\leq 1$, $0<q<\infty$ and $\sigma_p<s<1$, then
\begin{equation}\label{eq673}
\|f\|_{\Bspq}\lesssim\sum_{j=1}^n
(\int_0^{\infty}(\int_{\bbbr^n}|f(x+te_j)-f(x)|^p dx)^{\frac{q}{p}}\frac{dt}{t^{1+sq}})^{\frac{1}{q}}.
\end{equation}
\end{corollary}
\begin{proof}[Proof of Corollary \ref{corollary5}]
Apply Theorem \ref{theorem7} (i) and (ii) with $L=1$. And (\ref{eq673}) also indicates the following inequality
\begin{equation}\label{eq674}
2^{ks}\|\psi_{2^{-k}}*f\|_{L^p(\bbbr^n)}\lesssim\sum_{j=1}^n
(\int_0^{\infty}(\int_{\bbbr^n}|f(x+te_j)-f(x)|^p dx)^{\frac{q}{p}}\frac{dt}{t^{1+sq}})^{\frac{1}{q}}
\end{equation}
for every $k\in\bbbz$, and hence $\lim_{k\rightarrow+\infty}\|\psi_{2^{-k}}*f\|_{L^p(\bbbr^n)}=0$ when $s>0$ and the right side of (\ref{eq673}) is finite.
\end{proof}
\begin{corollary}\label{corollary6}
Let $0<p,q\leq\infty$, $s\in\bbbr$ and at least one of $p$ and $q$ is infinity. Assume $f\in\Bspq$ has a function representative.\\
(i)If $p=\infty$, $0<q<\infty$ and $0<s<1$, then for each $j\in\{1,\cdots,n\}$
\begin{equation}\label{eq675}
(\int_0^{\infty}\esssup_{x\in\bbbr^n}|f(x+te_j)-f(x)|^q\frac{dt}{t^{1+sq}})^{\frac{1}{q}}\lesssim\|f\|_{\Bsinfq}.
\end{equation}
(ii)If $p=\infty$, $1\leq q<\infty$ and $s\in\bbbr$, or if $p=\infty$, $0<q<1$ and $0<s<1$, then
\begin{equation}\label{eq676}
\|f\|_{\Bsinfq}\lesssim\sum_{j=1}^n(\int_0^{\infty}\esssup_{x\in\bbbr^n}|f(x+te_j)-f(x)|^q\frac{dt}{t^{1+sq}})^{\frac{1}{q}}.
\end{equation}
(iii)If $0<p<\infty$, $q=\infty$ and $0<s<1$, then for each $j\in\{1,\cdots,n\}$
\begin{equation}\label{eq677}
\esssup_{t>0}t^{-s}(\int_{\bbbr^n}|f(x+te_j)-f(x)|^p dx)^{\frac{1}{p}}\lesssim\|f\|_{\Bspinf}.
\end{equation}
(iv)If $1<p<\infty$, $q=\infty$ and $s\in\bbbr$, or if $0<p\leq 1$, $q=\infty$ and $\sigma_p<s<\infty$, then
\begin{equation}\label{eq678}
\|f\|_{\Bspinf}\lesssim\sum_{j=1}^n\esssup_{t>0}t^{-s}(\int_{\bbbr^n}|f(x+te_j)-f(x)|^p dx)^{\frac{1}{p}}.
\end{equation}
(v)If $p=q=\infty$ and $0<s<1$, then for each $j\in\{1,\cdots,n\}$
\begin{equation}\label{eq679}
\esssup_{t>0}\esssup_{x\in\bbbr^n}\frac{|f(x+te_j)-f(x)|}{t^s}\lesssim\|f\|_{\Bsifif}.
\end{equation}
(vi)If $p=q=\infty$ and $s\in\bbbr$, then
\begin{equation}\label{eq680}
\|f\|_{\Bsifif}\lesssim\sum_{j=1}^n\esssup_{t>0}\esssup_{x\in\bbbr^n}\frac{|f(x+te_j)-f(x)|}{t^s}.
\end{equation}
\end{corollary}
\begin{proof}[Proof of Corollary \ref{corollary6}]
Apply Theorem \ref{theorem7} with $L=1$. We also deduce from (\ref{eq675}) that for almost every $x\in\bbbr^n$,
\begin{equation}\label{eq681}
(\int_0^{\infty}|f(x+te_j)-f(x)|^q\frac{dt}{t^{1+sq}})^{\frac{1}{q}}\lesssim\|f\|_{\Bsinfq}
\end{equation}
when conditions of Corollary \ref{corollary6} (i) are satisfied. Furthermore if $0<p\leq\infty$, $q=\infty$ and $0<s<1$, then we can infer from (\ref{eq677}) and (\ref{eq679}) the following inequality
\begin{equation}\label{eq682}
\|f(\cdot+te_j)-f(\cdot)\|_{L^p(\bbbr^n)}\lesssim t^s\|f\|_{\Bspinf}\qquad\text{ for }t>0.
\end{equation}
Moreover when conditions of Corollary \ref{corollary6} (iv) are satisfied, the inequality (\ref{eq678}) indicates
\begin{equation}\label{eq683}
\|f\|_{\Bspinf}\lesssim\sum_{j=1}^n(\int_{\bbbr^n}\esssup_{t>0}\frac{|f(x+te_j)-f(x)|^p}{t^{sp}}dx)^{\frac{1}{p}}.
\end{equation}
From (\ref{eq676}) and (\ref{eq680}) we see that for every $k\in\bbbz$ and almost every $x\in\bbbr^n$, $2^{ks}|\psi_{2^{-k}}*f(x)|$ can be estimated from above by the right sides of (\ref{eq676}) and (\ref{eq680}) respectively, and hence we deduce $\lim_{k\rightarrow+\infty}\psi_{2^{-k}}*f(x)=0$ when $s>0$ and the right sides of (\ref{eq676}) and (\ref{eq680}) are finite. From (\ref{eq678}) we see that for every $k\in\bbbz$
\begin{align}\label{eq697}
2^{ks}\|\psi_{2^{-k}}*f\|_{L^p(\bbbr^n)}
&\lesssim\sum_{j=1}^n\esssup_{t>0}t^{-s}(\int_{\bbbr^n}|f(x+te_j)-f(x)|^p dx)^{\frac{1}{p}} \nonumber\\
&\lesssim\sum_{j=1}^n(\int_{\bbbr^n}\esssup_{t>0}\frac{|f(x+te_j)-f(x)|^p}{t^{sp}}dx)^{\frac{1}{p}},
\end{align}
and hence we deduce $\lim_{k\rightarrow+\infty}\|\psi_{2^{-k}}*f\|_{L^p(\bbbr^n)}=0$ when $s>0$ and the right side of (\ref{eq697}) is finite.
\end{proof}
\begin{corollary}\label{corollary7}
Let $0<p,q<\infty$, $s\in\bbbr$ and $f\in\Bspq$ has a function representative.\\
(i)If $0<p<\infty$, $0<q<\infty$ and $0<s<2$, then for each $j\in\{1,\cdots,n\}$
\begin{equation}\label{eq684}
(\int_0^{\infty}(\int_{\bbbr^n}|f(x+2te_j)-2f(x+te_j)+f(x)|^p dx)^{\frac{q}{p}}\frac{dt}{t^{1+sq}})^{\frac{1}{q}}
\lesssim\|f\|_{\Bspq}.
\end{equation}
(ii)If $1<p<\infty$, $1\leq q<\infty$ and $s\in\bbbr$, or if $1<p<\infty$, $0<q<1$ and $0<s<2$, or if $0<p\leq 1$, $0<q<\infty$ and $\sigma_p<s<2$, then
\begin{equation}\label{eq685}
\|f\|_{\Bspq}\lesssim\sum_{j=1}^n
(\int_0^{\infty}(\int_{\bbbr^n}|f(x+2te_j)-2f(x+te_j)+f(x)|^p dx)^{\frac{q}{p}}\frac{dt}{t^{1+sq}})^{\frac{1}{q}}.
\end{equation}
\end{corollary}
\begin{proof}[Proof of Corollary \ref{corollary7}]
Apply Theorem \ref{theorem7} (i) and (ii) with $L=2$. And (\ref{eq685}) also indicates the following inequality
\begin{align}\label{eq686}
&2^{ks}\|\psi_{2^{-k}}*f\|_{L^p(\bbbr^n)} \nonumber\\
&\lesssim\sum_{j=1}^n
(\int_0^{\infty}(\int_{\bbbr^n}|f(x+2te_j)-2f(x+te_j)+f(x)|^p dx)^{\frac{q}{p}}\frac{dt}{t^{1+sq}})^{\frac{1}{q}}
\end{align}
for every $k\in\bbbz$, and hence $\lim_{k\rightarrow+\infty}\|\psi_{2^{-k}}*f\|_{L^p(\bbbr^n)}=0$ when $s>0$ and the right side of (\ref{eq685}) is finite.
\end{proof}
\begin{corollary}\label{corollary8}
Let $0<p,q\leq\infty$, $s\in\bbbr$ and at least one of $p$ and $q$ is infinity. Assume $f\in\Bspq$ has a function representative.\\
(i)If $p=\infty$, $0<q<\infty$ and $0<s<2$, then for each $j\in\{1,\cdots,n\}$
\begin{equation}\label{eq687}
(\int_0^{\infty}\esssup_{x\in\bbbr^n}|f(x+2te_j)-2f(x+te_j)+f(x)|^q\frac{dt}{t^{1+sq}})^{\frac{1}{q}}\lesssim\|f\|_{\Bsinfq}.
\end{equation}
(ii)If $p=\infty$, $1\leq q<\infty$ and $s\in\bbbr$, or if $p=\infty$, $0<q<1$ and $0<s<2$, then
\begin{equation}\label{eq688}
\|f\|_{\Bsinfq}\lesssim\sum_{j=1}^n
(\int_0^{\infty}\esssup_{x\in\bbbr^n}|f(x+2te_j)-2f(x+te_j)+f(x)|^q\frac{dt}{t^{1+sq}})^{\frac{1}{q}}.
\end{equation}
(iii)If $0<p<\infty$, $q=\infty$ and $0<s<2$, then for each $j\in\{1,\cdots,n\}$
\begin{equation}\label{eq689}
\esssup_{t>0}t^{-s}(\int_{\bbbr^n}|f(x+2te_j)-2f(x+te_j)+f(x)|^p dx)^{\frac{1}{p}}\lesssim\|f\|_{\Bspinf}.
\end{equation}
(iv)If $1<p<\infty$, $q=\infty$ and $s\in\bbbr$, or if $0<p\leq 1$, $q=\infty$ and $\sigma_p<s<\infty$, then
\begin{equation}\label{eq690}
\|f\|_{\Bspinf}\lesssim\sum_{j=1}^n\esssup_{t>0}t^{-s}(\int_{\bbbr^n}|f(x+2te_j)-2f(x+te_j)+f(x)|^p dx)^{\frac{1}{p}}.
\end{equation}
(v)If $p=q=\infty$ and $0<s<2$, then for each $j\in\{1,\cdots,n\}$
\begin{equation}\label{eq691}
\esssup_{t>0}\esssup_{x\in\bbbr^n}t^{-s}|f(x+2te_j)-2f(x+te_j)+f(x)|\lesssim\|f\|_{\Bsifif}.
\end{equation}
(vi)If $p=q=\infty$ and $s\in\bbbr$, then
\begin{equation}\label{eq692}
\|f\|_{\Bsifif}\lesssim\sum_{j=1}^n\esssup_{t>0}\esssup_{x\in\bbbr^n}t^{-s}|f(x+2te_j)-2f(x+te_j)+f(x)|.
\end{equation}
\end{corollary}
\begin{proof}[Proof of Corollary \ref{corollary8}]
Apply Theorem \ref{theorem7} with $L=2$. We also deduce from (\ref{eq687}) that for almost every $x\in\bbbr^n$,
\begin{equation}\label{eq693}
(\int_0^{\infty}|f(x+2te_j)-2f(x+te_j)+f(x)|^q\frac{dt}{t^{1+sq}})^{\frac{1}{q}}\lesssim\|f\|_{\Bsinfq}
\end{equation}
when conditions of Corollary \ref{corollary8} (i) are satisfied. Furthermore if $0<p\leq\infty$, $q=\infty$ and $0<s<2$, then we can infer from (\ref{eq689}) and (\ref{eq691}) the following inequality
\begin{equation}\label{eq694}
\|f(\cdot+2te_j)-2f(\cdot+te_j)+f(\cdot)\|_{L^p(\bbbr^n)}\lesssim t^s\|f\|_{\Bspinf}\qquad\text{ for }t>0.
\end{equation}
Moreover when conditions of Corollary \ref{corollary8} (iv) are satisfied, the inequality (\ref{eq690}) indicates
\begin{equation}\label{eq695}
\|f\|_{\Bspinf}\lesssim\sum_{j=1}^n(\int_{\bbbr^n}\esssup_{t>0}\frac{|f(x+2te_j)-2f(x+te_j)+f(x)|^p}{t^{sp}}dx)^{\frac{1}{p}}.
\end{equation}
From (\ref{eq688}) and (\ref{eq692}) we see that for every $k\in\bbbz$ and almost every $x\in\bbbr^n$, $2^{ks}|\psi_{2^{-k}}*f(x)|$ can be estimated from above by the right sides of (\ref{eq688}) and (\ref{eq692}) respectively, and hence we deduce $\lim_{k\rightarrow+\infty}\psi_{2^{-k}}*f(x)=0$ when $s>0$ and the right sides of (\ref{eq688}) and (\ref{eq692}) are finite. From (\ref{eq690}) we see that for every $k\in\bbbz$
\begin{align}\label{eq696}
&2^{ks}\|\psi_{2^{-k}}*f\|_{L^p(\bbbr^n)} \nonumber\\
&\lesssim\sum_{j=1}^n\esssup_{t>0}t^{-s}(\int_{\bbbr^n}|f(x+2te_j)-2f(x+te_j)+f(x)|^p dx)^{\frac{1}{p}} \nonumber\\
&\lesssim\sum_{j=1}^n(\int_{\bbbr^n}\esssup_{t>0}\frac{|f(x+2te_j)-2f(x+te_j)+f(x)|^p}{t^{sp}}dx)^{\frac{1}{p}},
\end{align}
and hence we deduce $\lim_{k\rightarrow+\infty}\|\psi_{2^{-k}}*f\|_{L^p(\bbbr^n)}=0$ when $s>0$ and the right side of (\ref{eq696}) is finite.
\end{proof}
The following theorem is the counterpart of Theorem \ref{theorem2} for $\Bspq$ spaces.
\begin{theorem}\label{theorem8}
Let $L\in\bbbn$, $0<p,q\leq\infty$, $s\in\bbbr$ and $f\in\Bspq$.\\
(i)If $0<p\leq\infty$, $0<q<\infty$ and $0<s<L$, then
\begin{equation}\label{eq698}
(\int_{\bbbr^n}|h|^{-sq}\|\varDelta^L_h f\|_{L^p(\bbbr^n)}^q\frac{dh}{|h|^n})^{\frac{1}{q}}\lesssim\|f\|_{\Bspq}.
\end{equation}
(ii)Suppose $f$ has a function representative. If $1<p\leq\infty$, $0<q<\infty$ and $s\in\bbbr$, or if $0<p<1$, $0<q<\infty$ and $\sigma_p<s<\infty$, then
\begin{equation}\label{eq699}
\|f\|_{\Bspq}\lesssim(\int_{\bbbr^n}|h|^{-sq}\|\varDelta^L_h f\|_{L^p(\bbbr^n)}^q\frac{dh}{|h|^n})^{\frac{1}{q}}.
\end{equation}
(iii)If $0<p\leq\infty$, $q=\infty$ and $0<s<L$, then
\begin{equation}\label{eq700}
\esssup_{h\in\bbbr^n}|h|^{-s}\cdot\|\varDelta^L_h f\|_{L^p(\bbbr^n)}\lesssim\|f\|_{\Bspinf}.
\end{equation}
(iv)Suppose $f$ has a function representative. If $1<p\leq\infty$, $q=\infty$ and $s\in\bbbr$, or if $0<p<1$, $q=\infty$ and $\sigma_p<s<\infty$, then
\begin{equation}\label{eq701}
\|f\|_{\Bspinf}\lesssim\esssup_{h\in\bbbr^n}|h|^{-s}\cdot\|\varDelta^L_h f\|_{L^p(\bbbr^n)}.
\end{equation}
(v)Suppose $f$ has a function representative. If $p=1$, $1\leq q<\infty$ and $-n<s<\infty$, or if $p=1$, $0<q<1$ and $0<s<\infty$, then
\begin{equation}\label{eq702}
\|f\|_{\dot{B}^s_{1,q}(\bbbr^n)}\lesssim
(\int_{\bbbr^n}|h|^{-sq}\|\varDelta^L_h f\|_{L^1(\bbbr^n)}^q\frac{dh}{|h|^n})^{\frac{1}{q}}.
\end{equation}
If $p=1$, $q=\infty$ and $-n<s<\infty$, then
\begin{equation}\label{eq703}
\|f\|_{\dot{B}^s_{1,\infty}(\bbbr^n)}\lesssim\esssup_{h\in\bbbr^n}|h|^{-s}\cdot\|\varDelta^L_h f\|_{L^1(\bbbr^n)}.
\end{equation}
\end{theorem}
Theorem \ref{theorem8} (v) and the case for $1<p\leq\infty$ in Theorem \ref{theorem8} (ii) and (iv) are new. The proof of Theorem \ref{theorem8} can be found in section \ref{proof.of.theorem8}. Initial considerations of the counterpart of Theorem \ref{theorem8} for the inhomogeneous $B^s_{p,q}(\bbbr^n)$ space were given by H. Triebel in \cite[section 2.5.12]{1983functionspaces} and \cite[section 2.6.1]{1992functionspaces}. The corollaries of Theorem \ref{theorem8} are formulated below.
\begin{corollary}\label{corollary9}
Let $0<p,q<\infty$, $s\in\bbbr$ and $f\in\Bspq$ has a function representative.\\
(i)If $0<p<\infty$, $0<q<\infty$ and $0<s<1$, then
\begin{equation}\label{eq746}
(\int_{\bbbr^n}(\int_{\bbbr^n}|f(x+h)-f(x)|^p dx)^{\frac{q}{p}}\frac{dh}{|h|^{n+sq}})^{\frac{1}{q}}\lesssim\|f\|_{\Bspq}.
\end{equation}
(ii)If $1<p<\infty$, $0<q<\infty$ and $s\in\bbbr$, or if $0<p<1$, $0<q<\infty$ and $\sigma_p<s<\infty$, then
\begin{equation}\label{eq747}
\|f\|_{\Bspq}\lesssim(\int_{\bbbr^n}(\int_{\bbbr^n}|f(x+h)-f(x)|^p dx)^{\frac{q}{p}}\frac{dh}{|h|^{n+sq}})^{\frac{1}{q}}.
\end{equation}
(iii)If $p=1$, $1\leq q<\infty$ and $-n<s<\infty$, or if $p=1$, $0<q<1$ and $0<s<\infty$, then
\begin{equation}\label{eq748}
\|f\|_{\dot{B}^s_{1,q}(\bbbr^n)}\lesssim
(\int_{\bbbr^n}(\int_{\bbbr^n}|f(x+h)-f(x)|dx)^q\frac{dh}{|h|^{n+sq}})^{\frac{1}{q}}.
\end{equation}
\end{corollary}
\begin{proof}[Proof of Corollary \ref{corollary9}]
Apply inequalities (\ref{eq698}), (\ref{eq699}) and (\ref{eq702}) with $L=1$. Inequalities (\ref{eq747}) and (\ref{eq748}) also indicate
\begin{equation}\label{eq749}
2^{ks}\|\psi_{2^{-k}}*f\|_{L^p(\bbbr^n)}\lesssim\|f\|_{\Bspq}\lesssim
(\int_{\bbbr^n}(\int_{\bbbr^n}|f(x+h)-f(x)|^p dx)^{\frac{q}{p}}\frac{dh}{|h|^{n+sq}})^{\frac{1}{q}},
\end{equation}
for every $k\in\bbbz$, and hence $\lim_{k\rightarrow+\infty}\|\psi_{2^{-k}}*f\|_{L^p(\bbbr^n)}=0$ when $s>0$ and the right sides of (\ref{eq747}) and (\ref{eq748}) are finite.
\end{proof}
We can also pick some special values for $p$, $q$, $s$ in the above inequalities and then deduce some other interesting inequalities. For example, let $1\leq q\leq p<\frac{n}{s}$, then $0<s<1$ and $s<\frac{n}{p}$. By Lemma \ref{lemma7}, we have $\|f_k\|_{L^p(\bbbr^n)}\lesssim 2^{kn(\frac{1}{q}-\frac{1}{p})}\|f_k\|_{L^q(\bbbr^n)}$ and hence
\begin{eqnarray}
\|f\|_{\Bspq}
&\lesssim&\big(\sum_{k\leq 0}2^{kq(s+\frac{n}{q}-\frac{n}{p})}
\|f_k\|_{L^q(\bbbr^n)}^q+\sum_{k>0}2^{kq(s+\frac{n}{q}-\frac{n}{p})}\|f_k\|_{L^q(\bbbr^n)}^q\big)^{\frac{1}{q}}\nonumber\\
&\lesssim&\big(\sum_{k\leq 0}\|f_k\|_{L^q(\bbbr^n)}^q \big)^{\frac{1}{q}}+\big(\sum_{k>0}2^{kn} \|f_k\|_{L^q(\bbbr^n)}^q\big)^{\frac{1}{q}}\nonumber\\
&\lesssim&\|f\|_{\dot{B}^{0}_{q,q}(\bbbr^n)}+
\|f\|_{\dot{B}^{\frac{n}{q}}_{q,q}(\bbbr^n)}.\label{eq778}
\end{eqnarray}
By the inequalities given in Corollary \ref{corollary9} and Fubini's theorem, we can further deduce the following inequality
\begin{eqnarray}
&&(\int_{\bbbr^n}(\int_{\bbbr^n}|f(x+h)-f(x)|^p dx)^{\frac{q}{p}}\frac{dh}{|h|^{n+sq}})^{\frac{1}{q}}\lesssim\|f\|_{\Bspq}
\nonumber\\
&\lesssim&\big(\int_{\bbbr^n}\int_{\bbbr^n}\frac{|f(x)-f(y)|^q}{|x-y|^n}dxdy\big)^{\frac{1}{q}}+\big(\int_{\bbbr^n}\int_{\bbbr^n}\frac{|f(x)-f(y)|^q}{|x-y|^{2n}}dxdy\big)^{\frac{1}{q}},\label{eq779}  
\end{eqnarray}
when the corresponding conditions on the parameters are satisfied.
\begin{corollary}\label{corollary10}
Let $0<p,q\leq\infty$, $s\in\bbbr$ and at least one of $p$ and $q$ is infinity. Assume $f\in\Bspq$ has a function representative.\\
(i)If $p=\infty$, $0<q<\infty$ and $0<s<1$, then
\begin{equation}\label{eq750}
(\int_{\bbbr^n}\esssup_{x\in\bbbr^n}\frac{|f(x+h)-f(x)|^q}{|h|^{n+sq}}dh)^{\frac{1}{q}}\lesssim\|f\|_{\Bsinfq}.
\end{equation}
(ii)If $p=\infty$, $0<q<\infty$ and $s\in\bbbr$, then
\begin{equation}\label{eq751}
\|f\|_{\Bsinfq}\lesssim(\int_{\bbbr^n}\esssup_{x\in\bbbr^n}\frac{|f(x+h)-f(x)|^q}{|h|^{n+sq}}dh)^{\frac{1}{q}}.
\end{equation}
(iii)If $0<p<\infty$, $q=\infty$ and $0<s<1$, then
\begin{equation}\label{eq752}
\esssup_{h\in\bbbr^n}|h|^{-s}\cdot(\int_{\bbbr^n}|f(x+h)-f(x)|^p dx)^{\frac{1}{p}}\lesssim\|f\|_{\Bspinf}.
\end{equation}
(iv)If $1<p<\infty$, $q=\infty$ and $s\in\bbbr$, or if $p=1$, $q=\infty$ and $-n<s<\infty$, or if $0<p<1$, $q=\infty$ and $\sigma_p<s<\infty$, then
\begin{equation}\label{eq753}
\|f\|_{\Bspinf}\lesssim\esssup_{h\in\bbbr^n}|h|^{-s}\cdot(\int_{\bbbr^n}|f(x+h)-f(x)|^p dx)^{\frac{1}{p}}.
\end{equation}
(v)If $p=q=\infty$ and $0<s<1$, then
\begin{equation}\label{eq754}
\esssup_{x,y\in\bbbr^n}\frac{|f(x)-f(y)|}{|x-y|^s}\lesssim\|f\|_{\Bsifif}.
\end{equation}
(vi)If $p=q=\infty$ and $s\in\bbbr$, then
\begin{equation}\label{eq755}
\|f\|_{\Bsifif}\lesssim\esssup_{x,y\in\bbbr^n}\frac{|f(x)-f(y)|}{|x-y|^s}.
\end{equation}
\end{corollary}
\begin{proof}[Proof of Corollary \ref{corollary10}]
Apply Theorem \ref{theorem8} with $L=1$. From (\ref{eq750}) we can see that
\begin{equation}\label{eq756}
(\int_{\bbbr^n}\frac{|f(x+h)-f(x)|^q}{|h|^{n+sq}}dh)^{\frac{1}{q}}\lesssim\|f\|_{\Bsinfq},
\end{equation}
for almost every $x\in\bbbr^n$ when conditions of Corollary \ref{corollary10} (i) are satisfied. From (\ref{eq752}) and (\ref{eq754}), we also deduce the following inequality
\begin{equation}\label{eq757}
\|f(\cdot+h)-f(\cdot)\|_{L^p(\bbbr^n)}\lesssim|h|^s\cdot\|f\|_{\Bspinf}
\end{equation}
for almost every $h\in\bbbr^n$ when $0<p\leq\infty$, $q=\infty$ and $0<s<1$. From (\ref{eq753}) with a proper change of variable, we can obtain
\begin{equation}\label{eq758}
\|f\|_{\Bspinf}\lesssim(\int_{\bbbr^n}\esssup_{y\in\bbbr^n}\frac{|f(x)-f(y)|^p}{|x-y|^{sp}}dx)^{\frac{1}{p}},
\end{equation}
when conditions of Corollary \ref{corollary10} (iv) are satisfied. Furthermore from (\ref{eq751}) we have
\begin{equation}\label{eq759}
2^{ks}|\psi_{2^{-k}}*f(x)|\lesssim(\int_{\bbbr^n}\esssup_{x\in\bbbr^n}\frac{|f(x+h)-f(x)|^q}{|h|^{n+sq}}dh)^{\frac{1}{q}}
\end{equation}
for every $k\in\bbbz$ and almost every $x\in\bbbr^n$, and from (\ref{eq755}) we have
\begin{equation}\label{eq760}
2^{ks}|\psi_{2^{-k}}*f(x)|\lesssim\esssup_{x,y\in\bbbr^n}\frac{|f(x)-f(y)|}{|x-y|^s}
\end{equation}
for every $k\in\bbbz$ and almost every $x\in\bbbr^n$, therefore $\lim_{k\rightarrow+\infty}|\psi_{2^{-k}}*f(x)|=0$ when $s>0$ and the right sides of (\ref{eq751}) and (\ref{eq755}) are finite. Moreover from (\ref{eq753}) we have
\begin{equation}\label{eq761}
2^{ks}\|\psi_{2^{-k}}*f\|_{L^p(\bbbr^n)}\lesssim\esssup_{h\in\bbbr^n}|h|^{-s}\cdot(\int_{\bbbr^n}|f(x+h)-f(x)|^p dx)^{\frac{1}{p}}
\end{equation}
for every $k\in\bbbz$, hence $\lim_{k\rightarrow+\infty}\|\psi_{2^{-k}}*f\|_{L^p(\bbbr^n)}=0$ when $s>0$ and the right side of (\ref{eq753}) is finite.
\end{proof}
If $0<\alpha\leq p<\infty$, $0<s<1$ and $0<\beta<\infty$, then by Lemma \ref{lemma7}, $f_k=\psi_{2^{-k}}*f$ satisfies $\|f_k\|_{L^p(\bbbr^n)}\lesssim 2^{kn(\frac{1}{\alpha}-\frac{1}{p})}\|f_k\|_{L^{\alpha}(\bbbr^n)}$, and we have
\begin{align}
\|f\|_{\Bspinf}
&\lesssim\esssup_{k\in\bbbz}2^{k(s+\frac{n}{\alpha}-\frac{n}{p})}
\|f_k\|_{L^{\alpha}(\bbbr^n)}\nonumber\\
&\lesssim\big(\int_{\bbbr^n}(\esssup_{k\in\bbbz}
2^{k(s+\frac{n}{\alpha}-\frac{n}{p})}|f_k(x)|
)^{\alpha}dx\big)^{\frac{1}{\alpha}}\nonumber\\
&\lesssim\big(\int_{\bbbr^n}(\sum_{k\in\bbbz}
2^{k\beta(s+\frac{n}{\alpha}-\frac{n}{p})}|f_k(x)|^{\beta}
)^{\frac{\alpha}{\beta}}dx\big)^{\frac{1}{\alpha}}=
\|f\|_{\dot{F}^{s+\frac{n}{\alpha}-\frac{n}{p}}_{\alpha,\beta}
(\bbbr^n)}.\label{eq780}
\end{align}
By combining (\ref{eq752}), (\ref{eq780}), and (\ref{eq466}) altogether, we can obtain
\begin{equation}\label{eq781}
\esssup_{h\in\bbbr^n}\frac{\|f(\cdot+h)-f(\cdot)\|_{L^p(\bbbr^n)}}{|h|^s}\lesssim
(\int_{\bbbr^n}(\int_{\bbbr^n}\frac{|f(x)-f(y)|^{\beta}}{|x-y|^{n+\beta(s+\frac{n}{\alpha}-\frac{n}{p})}}dy)^{\frac{\alpha}{\beta}}dx)^{\frac{1}{\alpha}},
\end{equation}
where the parameters satisfy
$$\max\{0,n(\frac{1}{\min\{\alpha,\beta\}}-1)\}+
\max\{0,n(\frac{1}{\alpha}-\frac{1}{\beta})\}<
s+\frac{n}{\alpha}-\frac{n}{p}$$
if $0<\beta<1$, and there are no extra conditions for parameters if $1\leq\beta<\infty$ since $-n<s+\frac{n}{\alpha}-\frac{n}{p}<\infty$ is always true for $0<\alpha\leq p<\infty$ and $0<s<1$. In particular, letting $\alpha=p$ in (\ref{eq781}) yields
\begin{equation}\label{eq782}
\esssup_{h\in\bbbr^n}\frac{\|f(\cdot+h)-f(\cdot)\|_{L^p(\bbbr^n)}}{|h|^s}\lesssim
(\int_{\bbbr^n}(\int_{\bbbr^n}\frac{|f(x)-f(y)|^{\beta}}{|x-y|^{n+s\beta}}dy)^{\frac{p}{\beta}}dx)^{\frac{1}{p}}=
[f]_{W^s_{p,\beta}(\bbbr^n)},   
\end{equation}
when the above conditions are met. If $0<\alpha<p=\infty$, $0<s<1$, and $0<\beta<\infty$, then by Lemma \ref{lemma7} we have $\|f_k\|_{L^{\infty}(\bbbr^n)}\lesssim 2^{\frac{kn}{\alpha}}\|f_k\|_{L^{\alpha}(\bbbr^n)}$ and
\begin{equation}\label{eq783}
\|f\|_{\Bsifif}\lesssim\esssup_{k\in\bbbz}2^{k(s+\frac{n}{\alpha})}\|f_k\|_{L^{\alpha}(\bbbr^n)}\lesssim\|f\|_{\dot{F}^{s+\frac{n}{\alpha}}_{\alpha,\infty}(\bbbr^n)}.
\end{equation}
By (\ref{eq754}), (\ref{eq783}), and (\ref{eq467}), we have
\begin{equation}\label{eq784}
\esssup_{x,y\in\bbbr^n}\frac{|f(x)-f(y)|}{|x-y|^s}\lesssim
(\int_{\bbbr^n}\esssup_{y\in\bbbr^n}\frac{|f(x)-f(y)|^{\alpha}}{|x-y|^{n+s\alpha}}dx)^{\frac{1}{\alpha}}
\end{equation}
for all $0<\alpha<\infty$ and $0<s<1$. Because $\|f\|_{\dot{F}^{s+\frac{n}{\alpha}}_{\alpha,\infty}(\bbbr^n)}\lesssim\|f\|_{\dot{F}^{s+\frac{n}{\alpha}}_{\alpha,\beta}(\bbbr^n)}$ for all $0<\beta<\infty$, then (\ref{eq754}), (\ref{eq783}) and (\ref{eq466}) combined together give us the following inequality
\begin{equation}\label{eq785}
\esssup_{x,y\in\bbbr^n}\frac{|f(x)-f(y)|}{|x-y|^s}\lesssim
(\int_{\bbbr^n}(\int_{\bbbr^n}\frac{|f(x)-f(y)|^{\beta}}{|x-y|^{n+(s+\frac{n}{\alpha})\beta}}dy)^{\frac{\alpha}{\beta}}dx)^{\frac{1}{\alpha}},
\end{equation}
where the parameters satisfy
$$\max\{0,n(\frac{1}{\min\{\alpha,\beta\}}-1)\}+
\max\{0,n(\frac{1}{\alpha}-\frac{1}{\beta})\}<
s+\frac{n}{\alpha}$$
if $0<\beta<1$, and there are no extra conditions for parameters if $1\leq\beta<\infty$. In particular, when $\alpha$ and $\beta$ are related by the equation $\beta=\alpha\cdot\gamma$ for some $\gamma>0$, then (\ref{eq785}) becomes
\begin{equation}\label{eq786}
\esssup_{x,y\in\bbbr^n}\frac{|f(x)-f(y)|}{|x-y|^s}\lesssim
(\int_{\bbbr^n}(\int_{\bbbr^n}\frac{|f(x)-f(y)|^{\alpha\cdot\gamma}}{|x-y|^{n+n\gamma+s\alpha\cdot\gamma}}dy)^{\frac{1}{\gamma}}dx)^{\frac{\gamma}{\beta}},
\end{equation}
when the corresponding conditions are satisfied.
\begin{corollary}\label{corollary11}
Let $0<p,q<\infty$, $s\in\bbbr$ and $f\in\Bspq$ has a function representative.\\
(i)If $0<p<\infty$, $0<q<\infty$ and $0<s<2$, then
\begin{equation}\label{eq762}
(\int_{\bbbr^n}(\int_{\bbbr^n}|f(x+2h)-2f(x+h)+f(x)|^p dx)^{\frac{q}{p}}\frac{dh}{|h|^{n+sq}})^{\frac{1}{q}}
\lesssim\|f\|_{\Bspq}.
\end{equation}
(ii)If $1<p<\infty$, $0<q<\infty$ and $s\in\bbbr$, or if $0<p<1$, $0<q<\infty$ and $\sigma_p<s<\infty$, then
\begin{equation}\label{eq763}
\|f\|_{\Bspq}\lesssim
(\int_{\bbbr^n}(\int_{\bbbr^n}|f(x+2h)-2f(x+h)+f(x)|^p dx)^{\frac{q}{p}}\frac{dh}{|h|^{n+sq}})^{\frac{1}{q}}.
\end{equation}
(iii)If $p=1$, $1\leq q<\infty$ and $-n<s<\infty$, or if $p=1$, $0<q<1$ and $0<s<\infty$, then
\begin{equation}\label{eq764}
\|f\|_{\dot{B}^s_{1,q}(\bbbr^n)}\lesssim
(\int_{\bbbr^n}(\int_{\bbbr^n}|f(x+2h)-2f(x+h)+f(x)|dx)^q\frac{dh}{|h|^{n+sq}})^{\frac{1}{q}}.
\end{equation}
\end{corollary}
\begin{proof}[Proof of Corollary \ref{corollary11}]
Apply inequalities (\ref{eq698}), (\ref{eq699}) and (\ref{eq702}) with $L=2$. Inequalities (\ref{eq763}) and (\ref{eq764}) also indicate
\begin{equation}\label{eq765}
2^{ks}\|\psi_{2^{-k}}*f\|_{L^p(\bbbr^n)}\lesssim
(\int_{\bbbr^n}(\int_{\bbbr^n}|f(x+2h)-2f(x+h)+f(x)|^p dx)^{\frac{q}{p}}\frac{dh}{|h|^{n+sq}})^{\frac{1}{q}},
\end{equation}
for every $k\in\bbbz$, and hence $\lim_{k\rightarrow+\infty}\|\psi_{2^{-k}}*f\|_{L^p(\bbbr^n)}=0$ when $s>0$ and the right sides of (\ref{eq763}) and (\ref{eq764}) are finite.
\end{proof}
\begin{corollary}\label{corollary12}
Let $0<p,q\leq\infty$, $s\in\bbbr$ and at least one of $p$ and $q$ is infinity. Assume $f\in\Bspq$ has a function representative.\\
(i)If $p=\infty$, $0<q<\infty$ and $0<s<2$, then
\begin{equation}\label{eq766}
(\int_{\bbbr^n}\esssup_{x\in\bbbr^n}|f(x+2h)-2f(x+h)+f(x)|^q\frac{dh}{|h|^{n+sq}})^{\frac{1}{q}}\lesssim\|f\|_{\Bsinfq}.
\end{equation}
(ii)If $p=\infty$, $0<q<\infty$ and $s\in\bbbr$, then
\begin{equation}\label{eq767}
\|f\|_{\Bsinfq}\lesssim(\int_{\bbbr^n}\esssup_{x\in\bbbr^n}|f(x+2h)-2f(x+h)+f(x)|^q\frac{dh}{|h|^{n+sq}})^{\frac{1}{q}}.
\end{equation}
(iii)If $0<p<\infty$, $q=\infty$ and $0<s<2$, then
\begin{equation}\label{eq768}
\esssup_{h\in\bbbr^n}|h|^{-s}\cdot(\int_{\bbbr^n}|f(x+2h)-2f(x+h)+f(x)|^p dx)^{\frac{1}{p}}\lesssim\|f\|_{\Bspinf}.
\end{equation}
(iv)If $1<p<\infty$, $q=\infty$ and $s\in\bbbr$, or if $p=1$, $q=\infty$ and $-n<s<\infty$, or if $0<p<1$, $q=\infty$ and $\sigma_p<s<\infty$, then
\begin{equation}\label{eq769}
\|f\|_{\Bspinf}\lesssim\esssup_{h\in\bbbr^n}|h|^{-s}\cdot(\int_{\bbbr^n}|f(x+2h)-2f(x+h)+f(x)|^p dx)^{\frac{1}{p}}.
\end{equation}
(v)If $p=q=\infty$ and $0<s<2$, then
\begin{equation}\label{eq770}
\esssup_{x,y\in\bbbr^n}\frac{|f(x)+f(y)-2f(\frac{x+y}{2})|}{|x-y|^s}\lesssim\|f\|_{\Bsifif}.
\end{equation}
(vi)If $p=q=\infty$ and $s\in\bbbr$, then
\begin{equation}\label{eq771}
\|f\|_{\Bsifif}\lesssim\esssup_{x,y\in\bbbr^n}\frac{|f(x)+f(y)-2f(\frac{x+y}{2})|}{|x-y|^s}.
\end{equation}
\end{corollary}
\begin{proof}[Proof of Corollary \ref{corollary12}]
Apply Theorem \ref{theorem8} with $L=2$. From (\ref{eq766}) we can see that
\begin{equation}\label{eq772}
(\int_{\bbbr^n}|f(x+2h)-2f(x+h)+f(x)|^q\frac{dh}{|h|^{n+sq}})^{\frac{1}{q}}\lesssim\|f\|_{\Bsinfq},
\end{equation}
for almost every $x\in\bbbr^n$ when conditions of Corollary \ref{corollary12} (i) are satisfied. From (\ref{eq768}) and (\ref{eq770}), we also deduce the following inequality
\begin{equation}\label{eq773}
\|f(\cdot+2h)-2f(\cdot+h)+f(\cdot)\|_{L^p(\bbbr^n)}\lesssim|h|^s\cdot\|f\|_{\Bspinf}
\end{equation}
for almost every $h\in\bbbr^n$ when $0<p\leq\infty$, $q=\infty$ and $0<s<2$. From (\ref{eq769}) with a proper change of variable, we can obtain
\begin{equation}\label{eq774}
\|f\|_{\Bspinf}\lesssim
(\int_{\bbbr^n}\esssup_{y\in\bbbr^n}\frac{|f(x)+f(y)-2f(\frac{x+y}{2})|^p}{|x-y|^{sp}}dx)^{\frac{1}{p}},
\end{equation}
when conditions of Corollary \ref{corollary12} (iv) are satisfied. Furthermore from (\ref{eq767}) we have
\begin{equation}\label{eq775}
2^{ks}|\psi_{2^{-k}}*f(x)|\lesssim
(\int_{\bbbr^n}\esssup_{x\in\bbbr^n}|f(x+2h)-2f(x+h)+f(x)|^q\frac{dh}{|h|^{n+sq}})^{\frac{1}{q}}
\end{equation}
for every $k\in\bbbz$ and almost every $x\in\bbbr^n$, and from (\ref{eq771}) we have
\begin{equation}\label{eq776}
2^{ks}|\psi_{2^{-k}}*f(x)|\lesssim\esssup_{x,y\in\bbbr^n}\frac{|f(x)+f(y)-2f(\frac{x+y}{2})|}{|x-y|^s}
\end{equation}
for every $k\in\bbbz$ and almost every $x\in\bbbr^n$, therefore $\lim_{k\rightarrow+\infty}|\psi_{2^{-k}}*f(x)|=0$ when $s>0$ and the right sides of (\ref{eq767}) and (\ref{eq771}) are finite. Moreover from (\ref{eq769}) we have
\begin{equation}\label{eq777}
2^{ks}\|\psi_{2^{-k}}*f\|_{L^p(\bbbr^n)}\lesssim
\esssup_{h\in\bbbr^n}|h|^{-s}\cdot(\int_{\bbbr^n}|f(x+2h)-2f(x+h)+f(x)|^p dx)^{\frac{1}{p}}
\end{equation}
for every $k\in\bbbz$, hence $\lim_{k\rightarrow+\infty}\|\psi_{2^{-k}}*f\|_{L^p(\bbbr^n)}=0$ when $s>0$ and the right side of (\ref{eq769}) is finite.
\end{proof}

\section{Preliminaries}\label{Preliminaries}
In this section, we continue introducing some more notations and definitions. In order to write this paper in a more self-included fashion, we also cite useful results directly from the literature and the proofs of cited results can be found from their respective source and then we deduce frequently used remarks right after the closely related citations. Let $C_c^{\infty}(\bbbr^n)$ be the set of smooth functions on $\bbbr^n$ with compact supports and if a function $f$ is in $C_c^{\infty}(\bbbr^n)$, we use $spt.f$ to denote the support set of this function. The notation ``$X\lesssim Y$'' means $X$ is dominated by a constant multiple of $Y$ and the constant is determined by some fixed parameters, and when we want to emphasize the constant is $1$ we still use the usual notation ``$X\leq Y$''. If $X\lesssim Y$ and $Y\lesssim X$, then we consider $X$ and $Y$ are equivalent and write $X\sim Y$. For a sufficiently smooth function $f$ and a multi-index $\alpha=(\alpha_1,\cdots,\alpha_n)$, we denote the derivative by
$$\partial^{\alpha}f(x)=\frac{\partial^{|\alpha|}f(x)}
{\partial x_1^{\alpha_1} \partial x_2^{\alpha_2}\cdots\partial x_n^{\alpha_n}}.$$
We also recall the definition of the famous Hardy-Littlewood maximal function.
\begin{definition}\label{definition4}
If a function $f$ is locally integrable on $\bbbr^n$, then
$$\HLmax_n(f)(x):=\esssup_{\delta>0} \mvint_{B(x,\delta)}|f(y)|dy$$
is the $n$-dimensional Hardy-Littlewood maximal function of $f$ at $x$, and $B(x,\delta)$ is the $n$-dimensional ball centered at $x\in\bbbr^n$ and has radius $\delta$.
\end{definition}
We define the Peetre-Fefferman-Stein maximal function for a function, the Fourier transform of which has compact supports in $\bbbr^n$.
\begin{definition}\label{definition2}
If $f$ is a function defined on $\bbbr^n$ whose distributional Fourier transform is compactly supported in the ball $B(0,t)\subseteq\bbbr^n$ centered at origin with radius $t>0$, then the associated $n$-dimensional Peetre-Fefferman-Stein maximal function of $f$ at $x$ is given by
$$\PFSmax_n f(x)=\esssup_{z\in\bbbr^n}\frac{|f(x-z)|}{(1+t|z|)^{\frac{n}{r}}}$$
where in most cases of this paper we pick $r$ to be a positive number satisfying either $0<r<\min\{p,q\}$ or $0<r<p$.
\end{definition}
In general, the $n$-dimensional Peetre-Fefferman-Stein maximal function can be defined as
$$\PFSmax_n f(x)=\esssup_{z\in\bbbr^n}\frac{|f(x-z)|}{(1+t|z|)^{a}}$$
for any positive real number $a$ but for the convenience of notations in this paper, we choose $a=\frac{n}{r}$ for the specified $r$.
\begin{remark}
For the Fourier transform, Hardy-Littlewood maximal function and Peetre-Fefferman-Stein maximal function, when we want to apply these operations only to some specific coordinates, we use a subscript number different from the dimension of the ambient space $\bbbr^n$. For example if $f(x)\in\Sw(\bbbr^n)$ for $n>1$, let $x=(x_1,x_2,\cdots,x_n)\in\bbbr^n$ and we denote $x_1'=(x_2,\cdots,x_n)$ then $x=(x_1,x_1')$ and $f(x)=f(x_1,x_1')$. If the $1$-dimensional Fourier transform is done with respect to $x_1$ then we use the notation
$$\FT_1 f(\cdot,x_1')(y_1):=\int_{\bbbr}f(x_1,x_1')e^{-2\pi ix_1 y_1}dx_1,$$
and if the $(n-1)$-dimensional Fourier transform is done with respect to $x_1'$ then we use the notation
$$\FT_{n-1}f(x_1,\cdot)(y_1'):=\int_{\bbbr^{n-1}}f(x_1,x_1')e^{-2\pi ix_1'\cdot y_1'}dx_1',$$
where $y_1'\in\bbbr^{n-1}$ and $x_1'\cdot y_1'$ is the inner product in $\bbbr^{n-1}$, $dx_1'=dx_2\cdots dx_n$. Similar notations are used for the inverse Fourier transforms. If we fix $x_1'\in\bbbr^{n-1}$ then the $1$-dimensional Hardy-Littlewood maximal function of $f$, with respect to the first coordinate, centered at $u\in\bbbr$ is given by
$$\HLmax_1(f(\cdot,x_1'))(u):=\esssup_{\delta>0}\mvint_{-\delta<t<\delta}|f(u+t,x_1')|dt,$$
and if furthermore the $1$-dimensional Fourier transform $\FT_1 f(\cdot,x_1')(u)$ is supported in the interval $\{u\in\bbbr:|u|<t\}$, $t\in\bbbr$, then we can also define the associated $1$-dimensional Peetre-Fefferman-Stein maximal function of $f$ at $u$, with respect to the first coordinate, as follows
$$\PFSmax_1 f(\cdot,x_1')(u)=\esssup_{z\in\bbbr}\frac{|f(u-z,x_1')|}{(1+t|z|)^{\frac{1}{r}}}.$$
\end{remark}

\begin{remark}\label{remark2}
Directly from Definition \ref{definition2} and the use of some basic inequalities we can obtain useful inequalities for the Peetre-Fefferman-Stein maximal function. If $x,y\in\bbbr^n$ then
\begin{align*}
\PFSmax_n f(x-y)
&=\esssup_{z\in\bbbr^n}\frac{|f(x-y-z)|}{(1+t|z+y|)^{n/r}}\cdot\frac{(1+t|y+z|)^{n/r}}{(1+t|z|)^{n/r}}\\
&\lesssim\esssup_{z\in\bbbr^n}\frac{|f(x-y-z)|}{(1+t|z+y|)^{n/r}}\cdot\frac{(1+t|y|)^{n/r}\cdot(1+t|z|)^{n/r}}{(1+t|z|)^{n/r}}\\
&=\PFSmax_n f(x)\cdot(1+t|y|)^{n/r}.
\end{align*}
And also, using the above inequality gives that
$$\PFSmax_n f(x)=\PFSmax_n f(x-y+y)\lesssim \PFSmax_n f(x-y)\cdot(1+t|y|)^{n/r}.$$
Therefore we reach the conclusion that for $x,y\in\bbbr^n$
\begin{equation}\label{eq70}
\PFSmax_n f(x)\cdot(1+t|y|)^{-n/r}\lesssim \PFSmax_n f(x-y)\lesssim \PFSmax_n f(x)\cdot(1+t|y|)^{n/r},
\end{equation}
where the constants involved are independent of $t$ and $y$. We also infer from (\ref{eq70}) that if $\PFSmax_n f(x)$ vanishes at some point $x\in\bbbr^n$ then the Peetre-Fefferman-Stein maximal function $\PFSmax_n f$ vanishes on all of $\bbbr^n$ and hence $f$ vanishes on all of $\bbbr^n$. So for a nonzero function $f$, the associated Peetre-Fefferman-Stein maximal function $\PFSmax_n f$ is positive everywhere. And it would be obvious that if $\PFSmax_n f(x)=\infty$ for some $x\in\bbbr^n$ then $\PFSmax_n f=\infty$ on all of $\bbbr^n$.
\end{remark}
\begin{remark}\label{remark3}
If $f$ is a function defined on $\bbbr^n$ and its distributional Fourier transform satisfies $spt.\FT_n f\subseteq B(0,t)$ for $t>0$ and $\varphi$ is a Schwartz function whose Fourier transform is compactly supported in $B(0,1)\subseteq\bbbr^n$ and $\varphi_{1/t}(x)=t^n\varphi(tx)$ then the distributional Fourier transform of the convolution $\varphi_{1/t}*f$ is also compactly supported in $B(0,t)\subseteq\bbbr^n$, and we have
\begin{align*}
\PFSmax_n (\varphi_{1/t}*f)(x)
&=\esssup_{z\in\bbbr^n}\frac{|\varphi_{1/t}*f(x-z)|}{(1+t|z|)^{n/r}}\\
&\lesssim\esssup_{z\in\bbbr^n}\int_{\bbbr^n}\frac{|\varphi_{1/t}(y)f(x-z-y)|\cdot(1+t|z+y|)^{n/r}}
         {(1+t|z|)^{n/r}(1+t|z+y|)^{n/r}}dy\\
&\lesssim\esssup_{z\in\bbbr^n}\int_{\bbbr^n}\frac{|\varphi_{1/t}(y)f(x-z-y)|\cdot(1+t|y|)^{n/r}}
         {(1+t|z+y|)^{n/r}}dy\\
&\lesssim\PFSmax_n f(x)\cdot\int_{\bbbr^n}|\varphi_{1/t}(y)|\cdot(1+t|y|)^{n/r}dy,
\end{align*}
that is,
\begin{equation}\label{eq71}
\PFSmax_n (\varphi_{1/t}*f)(x)\lesssim \PFSmax_n f(x)
\end{equation}
for $t>0, x\in\bbbr^n$ and the constant is independent of $t$.
\end{remark}
For the reader's convenience we also would like to cite some useful results from the well-known literature \cite{14classical, 14modern} below.
\begin{lemma}[cf. Lemma 2.2.3 of \cite{14modern}]\label{lemma4}
Let $0<r<\infty$. Then there exist constants $C_1$ and $C_2$ such that for all $t>0$ and for all $\mathscr{C}^1$ functions $u$ on $\bbbr^n$ whose distributional Fourier transform is supported in the ball $|\xi|\leq t$ we have
\begin{align}
\esssup_{z\in\bbbr^n}\frac{1}{t}\frac{|\nabla u(x-z)|}{(1+t|z|)^{\frac{n}{r}}}
&\leq C_1 \esssup_{z\in\bbbr^n}\frac{|u(x-z)|}{(1+t|z|)^{\frac{n}{r}}},\label{eq72}\\
\esssup_{z\in\bbbr^n}\frac{|u(x-z)|}{(1+t|z|)^{\frac{n}{r}}}
&\leq C_2 \HLmax_n(|u|^r)(x)^{\frac{1}{r}},\label{eq73}
\end{align}
where $\HLmax_n$ denotes the Hardy-Littlewood maximal operator. The constants $C_1$ and $C_2$ depend only on the dimension $n$ and $r$; in particular, they are independent of $t$.
\end{lemma}
\begin{remark}\label{remark12}
The above Lemma \ref{lemma4} is significant in the sense that it provides a pointwise estimate by the famous Hardy-Littlewood maximal function to a function $u$ whose distributional Fourier transform has compact support in the ball $B(0,t)$ of center $0$ and radius $t$ and we have the following
$$|u(x)|\lesssim\esssup_{z\in\bbbr^n}\frac{|u(x-z)|}{(1+t|z|)^{\frac{n}{r}}}=\PFSmax_n u(x)
\lesssim \HLmax_n(|u|^r)(x)^{\frac{1}{r}}$$
where $r$ can be a positive finite number chosen to satisfy particular needs.
\end{remark}
\begin{remark}\label{remark4}
Assuming sufficient smoothness of the function $u$ as a tempered distribution in $\Sw'(\bbbr^n)$ and iterating (\ref{eq72}) of Lemma \ref{lemma4} repeatedly, since for any multi-index $\alpha=(\alpha_1,\cdots,\alpha_n)$, the distributional Fourier transform of $\partial^{\alpha}u$ is also supported in $B(0,t)\subseteq\bbbr^n$, we obtain
\begin{equation}\label{eq74}
\esssup_{z\in\bbbr^n}\frac{1}{t^{|\alpha|}}\frac{|\partial^{\alpha} u(x-z)|}{(1+t|z|)^{\frac{n}{r}}}
\lesssim\esssup_{z\in\bbbr^n}\frac{|u(x-z)|}{(1+t|z|)^{\frac{n}{r}}}
\end{equation}
and the constant is independent of $t$.
\end{remark}
\begin{lemma}[cf. Corollary 2.2.4 of \cite{14modern}]\label{lemma5}
Let $0<p\leq\infty$ and $\alpha$ a multi-index. Then there are constants $C=C(\alpha,n,p)$ and $C'=C(\alpha,n,p)$ such that for all Schwartz functions $u$ on $\bbbr^n$ whose Fourier transform is supported in the ball $B(0,t)$, for some $t>0$, we have
\begin{equation}\label{eq75}
\|\partial^{\alpha}u\|_{L^p(\bbbr^n)}\leq C t^{|\alpha|}\|u\|_{L^p(\bbbr^n)}
\end{equation}
and
\begin{equation}\label{eq76}
\|\partial^{\alpha}u\|_{L^{\infty}(\bbbr^n)}\leq C' t^{|\alpha|+\frac{n}{p}}\|u\|_{L^p(\bbbr^n)}.
\end{equation}
\end{lemma}
\begin{remark}\label{remark5}
Let $p,u,t,\alpha$ be given as in Lemma \ref{lemma5} and $q\in\bbbr$ satisfies $0<p\leq q\leq\infty$ then a simple interpolation with (\ref{eq75}) and (\ref{eq76}) reveals that
\begin{equation}\label{eq77}
\|\partial^{\alpha}u\|_{L^q(\bbbr^n)}\lesssim t^{|\alpha|+n(\frac{1}{p}-\frac{1}{q})}\|u\|_{L^p(\bbbr^n)},
\end{equation}
where $u$ is a Schwartz function on $\bbbr^n$ and $\FT_n u$ is supported in the ball $B(0,t)$. This inequality is also known as the Plancherel-Polya-Nikol'skij inequality and the constant on the right side of (\ref{eq77}) only depends on $\alpha,n,p,q$. For a more general introduction to the Plancherel-Polya-Nikol'skij inequality, we would like to refer the interested reader to section 1.3 of \cite{1983functionspaces}.
\end{remark}
The Plancherel-Polya-Nikol'skij inequality can be generalized to the class of sufficiently smooth functions that are also tempered distributions.
\begin{lemma}\label{lemma7}
Suppose $u(x)$ defined on $\bbbr^n$ is a sufficiently smooth function as a tempered distribution in $\Sw'(\bbbr^n)$, and its $n$-dimensional distributional Fourier transform is supported in the ball $\{\xi\in\bbbr^n:|\xi|\leq t\}$ for some $t>0$. Assume $\alpha=(\alpha_1,\cdots,\alpha_n)$ is a multi-index and $0<p\leq q\leq\infty$, then
\begin{equation}\label{eq517}
\|\partial^{\alpha}u\|_{L^q(\bbbr^n)}\lesssim t^{|\alpha|+n(\frac{1}{p}-\frac{1}{q})}\|u\|_{L^p(\bbbr^n)},
\end{equation}
and the constant on the right side of (\ref{eq517}) only depends on $\alpha,n,p,q$.
\end{lemma}
\begin{proof}
Let $\varphi\in\Sw(\bbbr^n)$ satisfy the following conditions
$$0\leq\FT_n\varphi\leq 1,\text{ }spt.\FT_n\varphi\subseteq\{\xi\in\bbbr^n:|\xi|\leq 2\}\text{ and }\FT_n\varphi=1
\text{ if }|\xi|\leq 1.$$
Then $\FT_n\varphi_{\frac{1}{t}}(\xi)=\FT_n\varphi(\frac{\xi}{t})$ is supported in $\{\xi\in\bbbr^n:|\xi|\leq 2t\}$ and equals $1$ if $|\xi|\leq t$. Let $g$ be an arbitrary element in $\Sw(\bbbr^n)$, then
$$<u,g>=<\FT_n u,\iFT_n g>=<\FT_n u,\FT_n\varphi_{\frac{1}{t}}\cdot\iFT_n g>+
<\FT_n u,(1-\FT_n\varphi_{\frac{1}{t}})\cdot\iFT_n g>.$$
The Schwartz function $(1-\FT_n\varphi_{\frac{1}{t}})\cdot\iFT_n g$ is supported in $\{\xi\in\bbbr^n:|\xi|>t\}$ and the distributional Fourier transform $\FT_n u$ is supported in $\{\xi\in\bbbr^n:|\xi|\leq t\}$, then
$$<\FT_n u,(1-\FT_n\varphi_{\frac{1}{t}})\cdot\iFT_n g>=0.$$
Hence we obtain the following
\begin{equation}\label{eq518}
<u,g>=<\FT_n u,\iFT_n(\tilde{\varphi}_{\frac{1}{t}}*g)>=<\varphi_{\frac{1}{t}}*u,g>,
\end{equation}
where $\tilde{\varphi}(x)=\varphi(-x)$. Equation (\ref{eq518}) shows that $u(x)=\varphi_{\frac{1}{t}}*u(x)$ for almost every $x\in\bbbr^n$. If $1\leq p\leq\infty$, using H\"{o}lder's inequality, we have
\begin{equation}\label{eq519}
\|u\|_{L^{\infty}(\bbbr^n)}\leq\esssup_{x\in\bbbr^n}\int_{\bbbr^n}|u(y)|\cdot t^n|\varphi(tx-ty)|dy\leq
t^{\frac{n}{p}}\|\varphi\|_{L^{p'}(\bbbr^n)}\cdot\|u\|_{L^p(\bbbr^n)},
\end{equation}
where $p'$ is the H\"{o}lder's conjugate of $p$. If $0<p<1$, then we have
\begin{align}\label{eq520}
\|u\|_{L^{\infty}(\bbbr^n)}
&\leq t^n\|\varphi\|_{L^{\infty}(\bbbr^n)}\int_{\bbbr^n}|u(y)|^p\cdot|u(y)|^{1-p}dy \nonumber\\
&\leq t^n\|\varphi\|_{L^{\infty}(\bbbr^n)}\cdot\|u\|_{L^{\infty}(\bbbr^n)}^{1-p}\cdot\|u\|_{L^p(\bbbr^n)}^p.
\end{align}
If $0<p\leq q\leq\infty$, we then use (\ref{eq519}) and (\ref{eq520}) to get
\begin{align}\label{eq521}
\int_{\bbbr^n}|u(x)|^q dx
&=\int_{\bbbr^n}|u(x)|^{q-p}\cdot|u(x)|^p dx \nonumber\\
&\lesssim t^{\frac{n}{p}(q-p)}\|u\|_{L^p(\bbbr^n)}^{q-p}\cdot\|u\|_{L^p(\bbbr^n)}^{p}
=t^{\frac{n}{p}(q-p)}\|u\|_{L^p(\bbbr^n)}^{q}.
\end{align}
Thus we can obtain (\ref{eq517}) if $|\alpha|=0$. When $u$ is sufficiently smooth and $|\alpha|>0$, we use Remark \ref{remark4} and Lemma \ref{lemma4} to obtain
\begin{equation}\label{eq522}
|\partial^{\alpha}u(x)|\lesssim\PFSmax_n(\partial^{\alpha}u)(x)\lesssim t^{|\alpha|}\PFSmax_n(u)(x)
\lesssim t^{|\alpha|}\HLmax_n(|u|^r)(x)^{\frac{1}{r}},
\end{equation}
where the constants depend on $n$ and $r$. By picking $r<q$ and invoking the mapping property of the Hardy-Littlewood maximal function, we have
\begin{equation}\label{eq523}
\|\partial^{\alpha}u\|_{L^q(\bbbr^n)}\lesssim t^{|\alpha|}\|\HLmax_n(|u|^r)(x)^{\frac{1}{r}}\|_{L^q(\bbbr^n)}\lesssim
t^{|\alpha|}\|u\|_{L^q(\bbbr^n)}\lesssim t^{|\alpha|+n(\frac{1}{p}-\frac{1}{q})}\|u\|_{L^p(\bbbr^n)},
\end{equation}
where constants depend only on $n,r,p,q$ and are independent of $t$.
\end{proof}
An example that shows the advantage of Lemma \ref{lemma7} over Remark \ref{remark5} can be given below. Consider the function $f_j=f*\psi_{2^{-j}}$ where $f\in\Sw'(\bbbr^n)$, $j\in\bbbz$, and $\psi\in\Sw(\bbbr^n)$ satisfies (\ref{eq5}), (\ref{eq37}). Then for each $j\in\bbbz$, the distributional Fourier transform of $f_j$ is supported in $\{\xi\in\bbbr^n:2^{j-1}\leq|\xi|<2^{j+1}\}$ and by Theorem 2.3.20 of \cite{14classical}, $f_j$ is a smooth function which has at most polynomial growth at infinity. Then $f_j$ is a smooth tempered distribution and applying Lemma \ref{lemma7}, we have
$$f_j(x)\lesssim\|f_j\|_{L^{\infty}(\bbbr^n)}\lesssim 2^{jn/p}\|f_j\|_{L^p(\bbbr^n)}\quad\text{for all }0<p<\infty.$$
\begin{lemma}[cf. Theorem 5.6.6 of \cite{14classical}]\label{lemma6}
For $1<p,r<\infty$ the Hardy-Littlewood maximal function $\HLmax$ satisfies the vector-valued inequalities
\begin{align}
\|(\sum_j |\HLmax(f_j)|^r)^{\frac{1}{r}}\|_{L^{1,\infty}}
&\leq C'_n(1+(r-1)^{-1})\|(\sum_j |f_j|^r)^{\frac{1}{r}}\|_{L^1},\label{eq78}\\
\|(\sum_j |\HLmax(f_j)|^r)^{\frac{1}{r}}\|_{L^{p}}
&\leq C_n c(p,r) \|(\sum_j |f_j|^r)^{\frac{1}{r}}\|_{L^p},\label{eq79}
\end{align}
where $c(p,r)=(1+(r-1)^{-1})(p+(p-1)^{-1})$.
\end{lemma}
\begin{remark}\label{remark6}
Recall that for each $l\in\bbbz$ the distributional Fourier transform of $f_l:=\psi_{2^{-l}}*f$ is supported in the compact annulus $2^{l-1}\leq|\xi|<2^{l+1}$ and
\begin{equation*}
\PFSmax_n f_l(x)=\esssup_{z\in\bbbr^n}\frac{|\psi_{2^{-l}}*f(x-z)|}{(1+2^{l+1}|z|)^{\frac{n}{r}}}.
\end{equation*}
Applying Lemma \ref{lemma4} first and then Lemma \ref{lemma6} and assuming that $s\in\bbbr, 0<p,q<\infty, 0<r<\min\{p,q\}$ we obtain
\begin{align}\label{eq80}
&(\int_{\bbbr^n}(\sum_{l=-\infty}^{\infty}|2^{ls}\PFSmax_n f_l(y)|^q)^{\frac{p}{q}}dy)^{\frac{1}{p}} \nonumber\\
&\lesssim(\int_{\bbbr^n}(\sum_{l=-\infty}^{\infty}|2^{ls}\HLmax_n(|f_l|^r)(y)^{\frac{1}{r}}|^q)^{\frac{p/r}{q/r}}
         dy)^{\frac{1/r}{p/r}} \nonumber\\
&\lesssim(\int_{\bbbr^n}(\sum_{l=-\infty}^{\infty}|2^{ls}f_l(y)|^q)^{\frac{p}{q}}dy)^{\frac{1}{p}}.
\end{align}
And by definition of the Peetre-Fefferman-Stein maximal function, $f_l(y)\lesssim\PFSmax_n f_l(y)$ for every $y\in\bbbr^n$ and $l\in\bbbz$ thus
\begin{equation}
(\int_{\bbbr^n}(\sum_{l=-\infty}^{\infty}|2^{ls}f_l(y)|^q)^{\frac{p}{q}}dy)^{\frac{1}{p}}
\lesssim(\int_{\bbbr^n}(\sum_{l=-\infty}^{\infty}|2^{ls}\PFSmax_n f_l(y)|^q)^{\frac{p}{q}}dy)^{\frac{1}{p}}.\label{eq81}
\end{equation}
Therefore we have reached the conclusion that for $s\in\bbbr, 0<p,q<\infty$, the following is true
\begin{align}\label{eq82}
\Fspqf
&=(\int_{\bbbr^n}(\sum_{l=-\infty}^{\infty}|2^{ls}f_l(y)|^q)^{\frac{p}{q}}dy)^{\frac{1}{p}} \nonumber\\
&\sim(\int_{\bbbr^n}(\sum_{l=-\infty}^{\infty}|2^{ls}\PFSmax_n f_l(y)|^q)^{\frac{p}{q}}dy)^{\frac{1}{p}}.
\end{align}
\end{remark}
Next, we provide useful lemmas in the proofs of our main theorems.
\begin{lemma}\label{lemma1}
Let $x,y\in\bbbr^n$, $L\in\bbbn$, $r>0$ and for each $k\in\bbbz$ and $j\in\bbbz$, if $h\in\bbbr^n$ and $|h|\sim 2^{-k}$ and $f_j=\psi_{2^{-j}}*f$, then we have the following two estimates
\begin{align}
|(\varDelta^L_h f_j)(x-y)|\cdot(1+\frac{|y|}{|h|})^{\frac{-n}{r}}&\lesssim
2^{(j-k)L}(1+2^{j-k})^{\frac{n}{r}} \PFSmax_n f_j(x),\label{eq363}\\
|(\varDelta^L_h f_j)(x-y)|\cdot(1+\frac{|y|}{|h|})^{\frac{-n}{r}}&\lesssim
(1+2^{j-k})^{\frac{n}{r}} \PFSmax_n f_j(x),\label{eq364}
\end{align}
and constants are independent of $y$ and $h$.
\end{lemma}
\begin{proof}
To prove (\ref{eq363}), we use the mean value theorem and the iteration formula (\ref{eq259}) consecutively and obtain
\begin{equation}\label{eq365}
|(\varDelta^L_h f_j)(x-y)|\lesssim
\sum_{|\alpha|=L}|\partial^{\alpha}f_j(x-y+\sum^L_{l=1}t_{\alpha,l}h)|\cdot|h|^L,
\end{equation}
where $\alpha$ represents a multi-index and each $t_{\alpha,l}$ is in the interval $(0,1)$. For each multi-index $\alpha$ with $|\alpha|=L$ and since $\FT_n\partial^{\alpha}f_j(\xi)$ is supported in $\{\xi\in\bbbr^n:2^{j-1}\leq|\xi|<2^{j+1}\}$, we use the basic inequality
\begin{equation}\label{eq366}
(1+2^{j+1}|y-\sum^L_{l=1}t_{\alpha,l}h|)^{\frac{n}{r}}
\lesssim(1+2^{j-k}\cdot 2^k|y|+2^j|\sum^L_{l=1}t_{\alpha,l}h|)^{\frac{n}{r}}
\lesssim(1+2^k|y|)^{\frac{n}{r}}\cdot(1+2^{j-k})^{\frac{n}{r}}
\end{equation}
and obtain that
\begin{align}\label{eq367}
&|\partial^{\alpha}f_j(x-y+\sum^L_{l=1}t_{\alpha,l}h)| \nonumber\\
&=\frac{|\partial^{\alpha}f_j(x-y+\sum^L_{l=1}t_{\alpha,l}h)|}{(1+2^{j+1}|y-\sum^L_{l=1}t_{\alpha,l}h|)^{\frac{n}{r}}}
  \cdot(1+2^{j+1}|y-\sum^L_{l=1}t_{\alpha,l}h|)^{\frac{n}{r}} \nonumber\\
&\lesssim(1+2^k|y|)^{\frac{n}{r}}\cdot(1+2^{j-k})^{\frac{n}{r}}\PFSmax_n (\partial^{\alpha}f_j)(x).
\end{align}
We also use Remark \ref{remark4} to get $\PFSmax_n (\partial^{\alpha}f_j)(x)\lesssim 2^{jL} \PFSmax_n f_j(x)$. Since $|h|\sim 2^{-k}$, we also have
$$(1+2^k|y|)^{\frac{n}{r}}\lesssim(1+\frac{|y|}{|h|})^{\frac{n}{r}}.$$
We insert these estimates and (\ref{eq367}) into (\ref{eq365}) and obtain (\ref{eq363}). To prove (\ref{eq364}), we use (\ref{eq260}), (\ref{eq261}) and Remark \ref{remark2} to get
\begin{equation}\label{eq368}
|(\varDelta^L_h f_j)(x-y)|\lesssim\sum_{l=0}^L|f_j(x-y+lh)|
\lesssim\sum_{l=0}^L (1+2^j|y-lh|)^{\frac{n}{r}} \PFSmax_n f_j(x).
\end{equation}
Since $|h|\sim 2^{-k}$, then
\begin{equation}\label{eq369}
(1+2^j|y-lh|)^{\frac{n}{r}}\lesssim(1+2^k|y|)^{\frac{n}{r}}\cdot(1+2^{j-k})^{\frac{n}{r}}
\lesssim(1+\frac{|y|}{|h|})^{\frac{n}{r}}\cdot(1+2^{j-k})^{\frac{n}{r}}.
\end{equation}
Inserting (\ref{eq369}) into (\ref{eq368}) yields (\ref{eq364}).
\end{proof}
\begin{lemma}\label{lemma2}
Let $f\in C_c^{\infty}(\bbbr^n)$, $\tau\in[1,2]$ and $\theta\in\unitsph$ satisfy the condition that for all $\xi\in spt.f$, there are positive real numbers $a$ and $b$, which can be sufficiently small and are independent of $\tau$ and $\theta$, such that $a\leq|\theta\cdot\xi|\leq b$. Assume $L$ is a positive integer. Then for any positive integer $N$, there exists a positive constant $C$ such that
\begin{equation}\label{eq407}
|\FT_n(\frac{f(\xi)}{(e^{2\pi i\tau\theta\cdot\xi}-1)^L})(x)|\leq\frac{C}{(1+|x|)^N}\qquad\text{for all }x\in\bbbr^n.
\end{equation}
Most importantly, the constant $C$ may depend on $f,L,N,a,b$ but is independent of $\tau\in[1,2]$ and $\theta\in\unitsph$.
\end{lemma}
\begin{proof}
If $|x|<1$, we use Taylor expansion and obtain
\begin{equation}\label{eq408}
(e^{2\pi i\tau\theta\cdot\xi}-1)^L=(2\pi i\tau\theta\cdot\xi)^L\cdot(1+O(2\pi i\tau\theta\cdot\xi)).
\end{equation}
Since $\tau\in[1,2]$ and if $\xi\in spt.f$, $a$ and $b$ are sufficiently small, then we have
\begin{equation}\label{eq409}
|(e^{2\pi i\tau\theta\cdot\xi}-1)^L|\geq C_1 a^L(1-C_2 b)>0,
\end{equation}
where $C_1$ and $C_2$ are constants independent of $\tau$ and $\theta$. Hence we have
\begin{equation}\label{eq410}
|\FT_n(\frac{f(\xi)}{(e^{2\pi i\tau\theta\cdot\xi}-1)^L})(x)|
\!\leq\!\int_{\bbbr^n}\frac{|f(\xi)|}{|(e^{2\pi i\tau\theta\cdot\xi}-1)^L|}d\xi
\!\lesssim\!\|f\|_{L^1(\bbbr^n)}\!\lesssim\!\frac{C}{(1+|x|)^N}
\end{equation}
if $|x|<1$ and the constant $C$ is independent of $\tau\in[1,2]$ and $\theta\in\unitsph$.

If $|x|\geq 1$, without loss of generality we can assume that $x=(x_1,\cdots,x_n)\in\bbbr^n$ and $|x_1|=\max\{|x_1|,|x_2|,\cdots,|x_n|\}>0$. Using integration by parts with respect to $\xi_1$, the condition that $f\in C_c^{\infty}(\bbbr^n)$ and the basic formula
$$e^{-2\pi ix_1\xi_1}=\frac{\partial}{\partial\xi_1}(\frac{e^{-2\pi ix_1\xi_1}}{-2\pi ix_1}),$$
we can obtain
\begin{equation*}
\FT_n(\frac{f(\xi)}{(e^{2\pi i\tau\theta\cdot\xi}-1)^L})(x)=\frac{1}{2\pi ix_1}
\int_{\bbbr^n}\frac{\partial}{\partial\xi_1}(\frac{f(\xi)}{(e^{2\pi i\tau\theta\cdot\xi}-1)^L})\cdot
e^{-2\pi ix\cdot\xi}d\xi.
\end{equation*}
We can iterate the integration by parts with respect to $\xi_1$ for $N$ times and obtain
\begin{equation}\label{eq411}
\FT_n(\frac{f(\xi)}{(e^{2\pi i\tau\theta\cdot\xi}-1)^L})(x)=\frac{1}{(2\pi ix_1)^N}
\int_{\bbbr^n}\frac{\partial^N}{\partial\xi_1^N}(\frac{f(\xi)}{(e^{2\pi i\tau\theta\cdot\xi}-1)^L})\cdot
e^{-2\pi ix\cdot\xi}d\xi.
\end{equation}
By direct calculation, we have
\begin{equation}\label{eq412}
|\frac{\partial^k}{\partial\xi_1^k}(e^{2\pi i\tau\theta\cdot\xi}-1)^{-L}|\lesssim
|\theta_1|^k\sum_{j=0}^k|(e^{2\pi i\tau\theta\cdot\xi}-1)^{-L-j}|\lesssim
\sum_{j=0}^k|(e^{2\pi i\tau\theta\cdot\xi}-1)^{-L-j}|
\end{equation}
for every nonnegative integer $k$ and the constants are independent of $\tau\in[1,2]$ and $\theta\in\unitsph$. Furthermore, using (\ref{eq408}) and (\ref{eq409}), if $\xi\in spt.f$ we can choose sufficiently small positive numbers $a$ and $b$ so that $a\leq|\theta\cdot\xi|\leq b$ implies
\begin{equation}\label{eq413}
|(e^{2\pi i\tau\theta\cdot\xi}-1)^{L+j}|\geq C_3>0\qquad\text{for all }0\leq j\leq N,
\end{equation}
and $C_3$ is independent of $\tau\in[1,2]$ and $\theta\in\unitsph$. Therefore Leibniz rule, (\ref{eq411}), (\ref{eq412}) and (\ref{eq413}) tell us that
\begin{equation}\label{eq414}
|\FT_n(\frac{f(\xi)}{(e^{2\pi i\tau\theta\cdot\xi}-1)^L})(x)|
\lesssim\frac{1}{|x|^N}\sum_{k=0}^N \|\frac{\partial^k f}{\partial\xi_1^k}\|_{L^1(\bbbr^n)}
\lesssim\frac{C}{(1+|x|)^N}
\end{equation}
if $|x|\geq 1$ and the constants are independent of $\tau\in[1,2]$ and $\theta\in\unitsph$. The proof of Lemma \ref{lemma2} is now complete.
\end{proof}
\begin{lemma}\label{lemma8}
Suppose $f$ is a tempered distribution in $\Sw'(\bbbr^n)$. Recall that $f_l(x)=\psi_{2^{-l}}*f(x)$ for every $l\in\bbbz$, $x=(x_1,x_1')\in\bbbr^n$. Then for every fixed $x_1'=(x_2,\cdots,x_n)\in\bbbr^{n-1}$, the smooth function $x_1\in\bbbr\mapsto f_l(x_1,x_1')$ is an element in $\Sw'(\bbbr)$ and its distributional Fourier transform $\FT_1 f_l(\cdot,x_1')$ is supported in the set $\{u\in\bbbr:|u|\leq 2^{l+1}\}$ and hence the associated $1$-dimensional Peetre-Fefferman-Stein maximal function can be defined as
\begin{equation}\label{eq477}
\PFSmax_1 f_l(\cdot,x_1')(u)=\esssup_{z\in\bbbr}\frac{|f_l(u-z,x_1')|}{(1+2^{l+1}|z|)^{\frac{1}{r}}}.
\end{equation}
\end{lemma}
\begin{proof}
Since $f\in\Sw'(\bbbr^n)$ and $\psi_{2^{-l}}\in\Sw(\bbbr^n)$, by Theorem 2.3.20 of \cite{14classical}, $f_l$ is a smooth function and there exist positive constants $a$ and $b$ such that
$$|f_l(x_1,x_1')|\leq a(1+|x_1|+|x_1'|)^b\leq a(1+|x_1'|)^b(1+|x_1|)^b.$$
This inequality shows for fixed $x_1'\in\bbbr^{n-1}$, the smooth function $x_1\in\bbbr\mapsto f_l(x_1,x_1')$ is in $\Sw'(\bbbr)$. To prove the distributional Fourier transform $\FT_1 f_l(\cdot,x_1')$ is supported in the set $\{u\in\bbbr:|u|\leq 2^{l+1}\}$, we find a sequence $\{\varphi_k\}_{k\in\bbbn}\subseteq C_c^{\infty}(\bbbr^n)$ so that $\{\varphi_k\}_{k\in\bbbn}$ converges to $f$ in the sense of $\Sw'(\bbbr^n)$. Next, we establish the equality
\begin{equation}\label{eq524}
<\FT_1 f_l(\cdot,x_1'),g>=\lim_{k\rightarrow\infty}<\FT_1(\psi_{2^{-l}}*\varphi_k(\cdot,x_1')),g>
\end{equation}
for every $g\in\Sw(\bbbr)$. With an argument like above, the smooth function $x_1\in\bbbr\rightarrow\psi_{2^{-l}}*\varphi_k(x_1,x_1')$ is an element of $\Sw'(\bbbr)$, thus we have the following
\begin{align}\label{eq525}
&<\FT_1(\psi_{2^{-l}}*\varphi_k(\cdot,x_1')),g> \nonumber\\
&=\int_{\bbbr}\psi_{2^{-l}}*\varphi_k(u,x_1')\cdot\FT_1 g(u)du \nonumber\\
&=\int_{\bbbr^n}\varphi_k(y_1,y_1')\int_{\bbbr}\psi_{2^{-l}}(u-y_1,x_1'-y_1')\FT_1 g(u)dudy_1dy_1'.
\end{align}
It is not hard to check by direct calculations that for fixed $x_1'\in\bbbr^{n-1}$, the function $(y_1,y_1')\in\bbbr^n\mapsto\int_{\bbbr}\psi_{2^{-l}}(u-y_1,x_1'-y_1')\FT_1 g(u)du$ is an element of $\Sw(\bbbr^n)$. Therefore we have
\begin{align}\label{eq526}
&\lim_{k\rightarrow\infty}<\FT_1(\psi_{2^{-l}}*\varphi_k(\cdot,x_1')),g> \nonumber\\
&=\lim_{k\rightarrow\infty}<\varphi_k,\int_{\bbbr}\psi_{2^{-l}}(u-\cdot,x_1'-\cdot)\FT_1 g(u)du> \nonumber\\
&=<f,\int_{\bbbr}\psi_{2^{-l}}(u-\cdot,x_1'-\cdot)\FT_1 g(u)du>.
\end{align}
We now justify the equation
\begin{equation}\label{eq527}
<f,\int_{\bbbr}\psi_{2^{-l}}(u-\cdot,x_1'-\cdot)\FT_1 g(u)du\!>=\!
\int_{\bbbr}\!<\!f,\psi_{2^{-l}}(u-\cdot,x_1'-\cdot)\!>\!\FT_1 g(u)du.
\end{equation}
Since $f$ is a tempered distribution in $\Sw'(\bbbr^n)$, (\ref{eq527}) requires the Riemann sums of the integral $\int_{\bbbr}\psi_{2^{-l}}(u-y_1,x_1'-y_1')\FT_1 g(u)du$ converges to that integral in Schwartz seminorms with respect to $y=(y_1,y_1')\in\bbbr^n$. For a sufficiently large $N\in\bbbn$, we consider the interval $[-N,N]$ and decompose it into $2N^2$ disjoint subintervals $\{I_m\}_{m=1}^{2N^2}$ of equal length $\frac{1}{N}$ and pick $u_m\in I_m$, then $[-N,N]=\bigcup_{m=1}^{2N^2}I_m$. The difference between the integral and its $N$-th Riemann sum is given by
\begin{equation}\label{eq528}
\int_{\bbbr}\psi_{2^{-l}}(u-y_1,x_1'-y_1')\FT_1 g(u)du-\sum_{m=1}^{2N^2}|I_m|\cdot
\psi_{2^{-l}}(u_m-y_1,x_1'-y_1')\FT_1 g(u_m)
\end{equation}
and can be written as a sum of the following two terms,
\begin{equation}\label{eq529}
\int_{|u|>N}\psi_{2^{-l}}(u-y_1,x_1'-y_1')\FT_1 g(u)du
\end{equation}
and
\begin{equation}\label{eq530}
\sum_{m=1}^{2N^2}\int_{I_m}\psi_{2^{-l}}(u-y_1,x_1'-y_1')\FT_1 g(u)-\psi_{2^{-l}}(u_m-y_1,x_1'-y_1')\FT_1 g(u_m)du.
\end{equation}
It is sufficient to show both (\ref{eq529}) and (\ref{eq530}) converge to zero in Schwartz seminorms with respect to $y$ as $N\rightarrow\infty$. Let $\alpha=(\alpha_1,\alpha_2,\cdots,\alpha_n)=(\alpha_1,\alpha_1')$ and $\beta=(\beta_1,\beta_2,\cdots,\beta_n)$ denote multi-indices. To estimate the Schwartz seminorms of (\ref{eq529}), we compute as follows
\begin{align}\label{eq531}
&|y^{\alpha}\cdot\partial_{y}^{\beta}(\int_{|u|>N}\psi_{2^{-l}}(u-y_1,x_1'-y_1')\FT_1 g(u)du)| \nonumber\\
&\lesssim|y_1|^{\alpha_1}\cdot|y_1'|^{|\alpha_1'|}\cdot\int_{|u|>N}
          |(\partial^{\beta}\psi)(2^l u-2^l y_1,2^l x_1'-2^l y_1')\FT_1 g(u)|du \nonumber\\
&\lesssim|y_1|^{\alpha_1}\cdot|y_1'|^{|\alpha_1'|}\cdot\int_{|u|>N}
          \frac{1}{(1+4^l|u-y_1|^2+4^l|x_1'-y_1'|^2)^{2M}}\cdot\frac{1}{(1+u^2)^M}du \nonumber\\
&\lesssim|y_1|^{\alpha_1}\cdot|y_1'|^{|\alpha_1'|}\cdot\int_{|u|>N}
          \frac{1}{(1+|u-y_1|^2)^M(1+|x_1'-y_1'|^2)^M(1+u^2)^M}du,
\end{align}
for some arbitrarily large positive integer $M$. We apply the following estimates
\begin{equation}\label{eq532}
(1+y_1^2)^{\frac{M}{2}}\leq 2^{\frac{M}{2}}(1+|u-y_1|^2)^{\frac{M}{2}}\cdot(1+u^2)^{\frac{M}{2}},
\end{equation}
\begin{equation}\label{eq533}
\frac{1}{(1+|u-y_1|^2)^{\frac{M}{2}}}\leq 1,\text{ and }
|y_1'|^{|\alpha_1'|}\leq(|y_1'-x_1'|^2+|x_1'|^2)^{\frac{|\alpha_1'|}{2}},
\end{equation}
then we can estimate (\ref{eq531}) from above by
\begin{equation}\label{eq534}
\frac{|y_1|^{\alpha_1}\cdot|y_1'|^{|\alpha_1'|}}{(1+y_1^2)^{\frac{M}{2}}(1+|x_1'-y_1'|^2)^M}\cdot
\int_{|u|>N}\frac{1}{(1+u^2)^{\frac{M}{2}}}du,
\end{equation}
where the constant may depend on $\alpha,\beta,l,\psi,g,M,x_1'$. And (\ref{eq534}) tends to zero as $N\rightarrow\infty$ uniformly in $y\in\bbbr^n$ if we pick $M\in\bbbn$ to be sufficiently large. This shows the Schwartz seminorms of (\ref{eq529}) with respect to $y$ converge to zero as $N\rightarrow\infty$. To estimate the Schwartz seminorms of each term in the summation of (\ref{eq530}), we use Mean Value Theorem with respect to $u$ and obtain
\begin{align*}
&|y^{\alpha}\cdot\partial_y^{\beta}(\int_{I_m}\psi_{2^{-l}}(u-y_1,x_1'-y_1')\FT_1 g(u)
          -\psi_{2^{-l}}(u_m-y_1,x_1'-y_1')\FT_1 g(u_m)du)|\\
&\lesssim|y^{\alpha}|\int_{I_m}|\frac{d}{du}[(\partial^{\beta}\psi)(2^l u-2^l y_1,2^l x_1'-2^l y_1')\FT_1 g(u)]\bigg|_{u=t_m}|
          \cdot|u-u_m|du\\
&\lesssim\frac{|y_1|^{\alpha_1}\cdot|y_1'|^{|\alpha_1'|}}{N}\int_{I_m}
          \frac{1}{(1+4^l|t_m-y_1|^2+4^l|x_1'-y_1'|^2)^{2M}}\cdot\frac{1}{(1+t_m^2)^M}du\\
&\lesssim\frac{|y_1|^{\alpha_1}\cdot|y_1'|^{|\alpha_1'|}}{N}\int_{I_m}
          \frac{1}{(1+|x_1'-y_1'|^2)^M(1+|t_m-y_1|^2)^M(1+t_m^2)^M}du,
\end{align*}
for some $t_m\in I_m$ between $u$ and $u_m$. When $N$ is large, we have $|u-t_m|\leq\frac{1}{N}<\frac{1}{2}$ and $u^2\leq 2t_m^2+\frac{1}{2}$. We also use the estimate $(1+y_1^2)^{\frac{M}{2}}\leq 2^{\frac{M}{2}}(1+|t_m-y_1|^2)^{\frac{M}{2}}\cdot(1+t_m^2)^{\frac{M}{2}}$ and then we can dominate the above inequality by
\begin{equation}\label{eq536}
\frac{|y_1|^{\alpha_1}\cdot|y_1'|^{|\alpha_1'|}}{N(1+y_1^2)^{\frac{M}{2}}(1+|x_1'-y_1'|^2)^M}\cdot
\int_{I_m}\frac{1}{(1+u^2)^{\frac{M}{2}}}du.
\end{equation}
Summing over $m=1,\cdots,2N^2$ and taking supremum over $y\in\bbbr^n$ yield that the Schwartz seminorms of (\ref{eq530}) with respect to $y$ can be estimated from above by
\begin{equation}\label{eq537}
\frac{1}{N}\int_{[-N,N]}\frac{1}{(1+u^2)^{\frac{M}{2}}}du,
\end{equation}
where $M$ is a sufficiently large positive integer and the constant may depend on $\alpha,\beta,l,\psi,g,M,x_1'$. And (\ref{eq537}) tends to zero as $N\rightarrow\infty$. Therefore the validity of equation (\ref{eq527}) has been proved. Furthermore from equation $(2.3.21)$ on page 127 of \cite{14classical}, we know $<f,\psi_{2^{-l}}(u-\cdot,x_1'-\cdot)>$ can be identified with the function $f_l(u,x_1')$. By inserting this identification into (\ref{eq527}) and combining the result with (\ref{eq526}), we have obtained
\begin{equation}\label{eq538}
<\FT_1 f_l(\cdot,x_1'),g>=<f_l(\cdot,x_1'),\FT_1 g>=\lim_{k\rightarrow\infty}<\FT_1(\psi_{2^{-l}}*\varphi_k(\cdot,x_1')),g>.
\end{equation}
Since $\psi\in\Sw(\bbbr^n)$ and $\varphi_k\in C_c^{\infty}(\bbbr^n)$, we have
\begin{equation}\label{eq539}
\psi_{2^{-l}}*\varphi_k(x)=\iFT_1[\iFT_{n-1}[\FT_n\psi(2^{-l}\xi)\FT_n\varphi_k(\xi)](x_1')](x_1),
\end{equation}
where $\xi=(\xi_1,\xi_1')\in\bbbr^n$ and the $(n-1)$-dimensional inverse Fourier transform is done with respect to $\xi_1'\in\bbbr^{n-1}$ and the $1$-dimensional inverse Fourier transform is done with respect to $\xi_1\in\bbbr$. Hence for every $k\in\bbbn$, we have
\begin{align}\label{eq540}
&<\FT_1(\psi_{2^{-l}}*\varphi_k(\cdot,x_1')),g> \nonumber\\
&=\int_{\bbbr}\FT_1(\psi_{2^{-l}}*\varphi_k(\cdot,x_1'))(u)\cdot g(u)du \nonumber\\
&=\int_{\bbbr}\iFT_{n-1}[\FT_n\psi(2^{-l}u,2^{-l}\xi_1')\FT_n\varphi_k(u,\xi_1')](x_1')\cdot g(u)du \nonumber\\
&=\iFT_{n-1}[\int_{\bbbr}\FT_n\psi(2^{-l}u,2^{-l}\xi_1')\FT_n\varphi_k(u,\xi_1')g(u)du](x_1'),
\end{align}
and from (\ref{eq540}) we also see that if $spt.g$ is contained in the complement of the set $\{u\in\bbbr:|u|\leq 2^{l+1}\}$, then $<\FT_1(\psi_{2^{-l}}*\varphi_k(\cdot,x_1')),g>=0$ for every $k\in\bbbn$. Therefore the distributional Fourier transform $\FT_1 f_l(\cdot,x_1')$ is supported in the set $\{u\in\bbbr:|u|\leq 2^{l+1}\}$ and the proof of Lemma \ref{lemma8} is now concluded.
\end{proof}

\section{Triebel-Lizorkin space and maximal functions of means of differences}\label{proof.of.theorem4}
\begin{proof}[Proof Of Theorem \ref{theorem4}]
In this proof we consider the number $r$ appearing in Definition \ref{definition5} equals the number $r$ appearing in Definition \ref{definition2}. We first prove that
\begin{equation}\label{eq370}
\|\{2^{ks}\esssup_{1\leq\tau<2}S^L_{\tau 2^{-k}}f(x)\}_{k\in\bbbz}\|_{L^p(l^q)}\lesssim\|f\|_{\Fspq}
\end{equation}
when $0<p<\infty$, $0<q\leq\infty$, $\frac{n}{\min\{p,q\}}<s<L$ and for every $r\in\bbbr$ with $\frac{n}{s}<r<\min\{p,q\}$. We denote
\begin{equation}\label{eq371}
S_k^*f(x)
=\sum_{j\in\bbbz}\esssup_{\substack{1\leq\tau<2\\y\in\bbbr^n}}\mvint_{\unitsph}
|\varDelta^L_{\tau 2^{-k}z}f_j(x-y)|\cdot(1+2^k|y|)^{\frac{-n}{r}}d\Haus^{n-1}(z).
\end{equation}
Since $|\tau 2^{-k}z|\sim 2^{-k}$ if $1\leq\tau<2$ and $z\in\unitsph$, then we have
\begin{align*}
&\mvint_{\unitsph}|\varDelta^L_{\tau 2^{-k}z}f_j(x-y)|\cdot(1+2^k|y|)^{\frac{-n}{r}}d\Haus^{n-1}(z)\\
&\lesssim\mvint_{\unitsph}|\varDelta^L_{\tau 2^{-k}z}f_j(x-y)|
          \cdot(1+\frac{|y|}{|\tau 2^{-k}z|})^{\frac{-n}{r}}d\Haus^{n-1}(z).
\end{align*}
We use (\ref{eq363}) for $j\leq k$ and (\ref{eq364}) for $j>k$, and obtain
\begin{align}\label{eq372}
S_k^*f(x)
&\lesssim\sum_{j\leq k}2^{(j-k)L}(1+2^{j-k})^{\frac{n}{r}} \PFSmax_n
          f_j(x)+\sum_{j>k}(1+2^{j-k})^{\frac{n}{r}} \PFSmax_n f_j(x) \nonumber\\
&\lesssim\sum_{j\leq k}2^{(j-k)L} \PFSmax_n f_j(x)+\sum_{j>k}2^{(j-k)\frac{n}{r}} \PFSmax_n f_j(x).
\end{align}
For $0<q<\infty$, we pick $0<\varepsilon<\min\{L-s,s-\frac{n}{r}\}$ and deduce from (\ref{eq372}) the following
\begin{align}\label{eq375}
&\sum_{k\in\bbbz}2^{ksq}(S_k^*f(x))^q \nonumber\\
&\lesssim\sum_{k\in\bbbz}2^{ksq}(\sum_{j\leq k}2^{j\varepsilon}\cdot 2^{-j\varepsilon+(j-k)L} \PFSmax_n f_j(x))^q \nonumber\\
&\quad+\sum_{k\in\bbbz}2^{ksq}(\sum_{j>k}2^{-j\varepsilon}\cdot 2^{j\varepsilon+(j-k)\frac{n}{r}} \PFSmax_n f_j(x))^q \nonumber\\
&\lesssim\sum_{k\in\bbbz}2^{ksq}(\sum_{j\leq k}2^{j\varepsilon})^q\cdot\esssup_{l\leq k}2^{-lq\varepsilon+(l-k)qL} \PFSmax_n
          f_l(x)^q \nonumber\\
&\quad+\sum_{k\in\bbbz}2^{ksq}(\sum_{j>k}2^{-j\varepsilon})^q\cdot
          \esssup_{l>k}2^{lq\varepsilon+(l-k)q\frac{n}{r}} \PFSmax_n f_l(x)^q \nonumber\\
&\lesssim\sum_{k\in\bbbz}2^{kq(s+\varepsilon)}\sum_{l\leq k}2^{-lq\varepsilon+(l-k)qL} \PFSmax_n f_l(x)^q \nonumber\\
&\quad+\sum_{k\in\bbbz}2^{kq(s-\varepsilon)}\sum_{l>k}2^{lq\varepsilon+(l-k)q\frac{n}{r}} \PFSmax_n f_l(x)^q \nonumber\\
&\lesssim\sum_{l\in\bbbz}2^{lsq} \PFSmax_n f_l(x)^q,
\end{align}
where in the last inequality we switched the order of summation. Then we raise the power to $\frac{1}{q}$, apply $\|\cdot\|_{L^p(\bbbr^n)}$-quasinorm to both sides of (\ref{eq375}), use Remark \ref{remark6} and we can obtain the estimate
\begin{equation}\label{eq374}
\|\{2^{ks}S_k^*f(x)\}_{k\in\bbbz}\|_{L^p(l^q)}
\lesssim(\int_{\bbbr^n}(\sum_{j\in\bbbz}2^{jsq} \PFSmax_n f_j(x)^q)^{\frac{p}{q}}dx)^{\frac{1}{p}}
\lesssim\|f\|_{\Fspq}.
\end{equation}
If $q=\infty$, we use the same $\varepsilon$ as above and obtain
\begin{align}\label{eq376}
&\esssup_{k\in\bbbz}2^{ks}S_k^*f(x) \nonumber\\
&\lesssim\esssup_{k\in\bbbz}2^{ks}\sum_{j\leq k}2^{j\varepsilon}\esssup_{l\leq k} 2^{-l\varepsilon+(l-k)L} \PFSmax_n f_l(x) \nonumber\\
&\quad+\esssup_{k\in\bbbz}2^{ks}\sum_{j>k}2^{-j\varepsilon}\esssup_{l>k} 2^{l\varepsilon+(l-k)\frac{n}{r}} \PFSmax_n f_l(x) \nonumber\\
&\lesssim\esssup_{k\in\bbbz}2^{k(s+\varepsilon)}\esssup_{l\leq k} 2^{-l\varepsilon+(l-k)L} \PFSmax_n f_l(x) \nonumber\\
&\quad+\esssup_{k\in\bbbz}2^{k(s-\varepsilon)}\esssup_{l>k} 2^{l\varepsilon+(l-k)\frac{n}{r}} \PFSmax_n f_l(x) \nonumber\\
&=\esssup_{l\in\bbbz}\esssup_{k\geq l}2^{k(s+\varepsilon-L)}\cdot 2^{l(L-\varepsilon)} \PFSmax_n f_l(x) \nonumber\\
&\quad+\esssup_{l\in\bbbz}\esssup_{k<l}2^{k(s-\varepsilon-\frac{n}{r})}\cdot 2^{l(\frac{n}{r}+\varepsilon)} \PFSmax_n f_l(x) \nonumber\\
&\lesssim\esssup_{l\in\bbbz}2^{ls} \PFSmax_n f_l(x).
\end{align}
We apply $\|\cdot\|_{L^p(\bbbr^n)}$-quasinorm to both sides of (\ref{eq376}), use Remark \ref{remark12} and the mapping property of Hardy-Littlewood maximal function to get
\begin{align}\label{eq377}
&\|\{2^{ks}S_k^*f(x)\}_{k\in\bbbz}\|_{L^p(l^{\infty})} \nonumber\\
&\lesssim\|\esssup_{l\in\bbbz}2^{ls}\HLmax_n(|f_l|^r)(x)^{\frac{1}{r}}\|_{L^p(\bbbr^n)} \nonumber\\
&\lesssim\|\HLmax_n(\esssup_{l\in\bbbz}2^{lsr}|f_l|^r)(x)\|_{L^{\frac{p}{r}}(\bbbr^n)}^{\frac{1}{r}} \nonumber\\
&\lesssim\|\esssup_{l\in\bbbz}2^{ls}|f_l|\|_{L^p(\bbbr^n)}=\|f\|_{\Fspinf}.
\end{align}
The above proof also shows for every $k\in\bbbz$,
\begin{equation}\label{eq541}
\|\esssup_{\substack{1\leq\tau<2\\y\in\bbbr^n}}\mvint_{\unitsph}\sum_{j\in\bbbz}
|\varDelta^L_{\tau 2^{-k}z}f_j(x-y)|d\Haus^{n-1}(z)\cdot(1+2^k|y|)^{\frac{-n}{r}}\|_{L^p(\bbbr^n)}<\infty,
\end{equation}
and thus $\sum_{j\in\bbbz}|\varDelta^L_{\tau 2^{-k}z}f_j(x-y)|<\infty$ for almost every $1\leq\tau<2$, $z\in\unitsph$, $x,y\in\bbbr^n$. Therefore we can infer from (\ref{eq513}) that
\begin{equation}\label{eq542}
\varDelta^L_{\tau 2^{-k}z}f(x-y)=\sum_{j\in\bbbz}\varDelta^L_{\tau 2^{-k}z}f_j(x-y)
\end{equation}
in the sense of $\Sw'(\bbbr^n)/\mathscr{P}(\bbbr^n)$ for every $k\in\bbbz$, and almost every $1\leq\tau<2$, $z\in\unitsph$, $x,y\in\bbbr^n$. The above justification of decomposition also tells us that
$$\esssup_{1\leq\tau<2}S^L_{\tau 2^{-k}}f(x)\lesssim S_k^*f(x),$$
and this estimate, combined with (\ref{eq374}) and (\ref{eq377}), finishes the proof of (\ref{eq370}). We also observe that for $0<q<\infty$
\begin{align}\label{eq378}
&\sum_{k\in\bbbz}2^{ksq}\big(\esssup_{0<\tau<2}S^L_{\tau 2^{-k}}f(x)\big)^q \nonumber\\
&\lesssim\sum_{k\in\bbbz}\sum_{j\geq 0}2^{ksq}\big(\esssup_{2^{-j}\leq\tau<2^{1-j}}S^L_{\tau 2^{-k}}f(x)\big)^q \nonumber\\
&=\sum_{j\geq 0}2^{-jsq}\sum_{k\in\bbbz}2^{(k+j)sq}\big(\esssup_{1\leq\tau<2}S^L_{\tau 2^{-k-j}}f(x)\big)^q \nonumber\\
&\lesssim\sum_{k\in\bbbz}2^{ksq}\big(\esssup_{1\leq\tau<2}S^L_{\tau 2^{-k}}f(x)\big)^q,
\end{align}
and for $q=\infty$
\begin{align}\label{eq379}
&\esssup_{k\in\bbbz}2^{ks}\esssup_{0<\tau<2}S^L_{\tau 2^{-k}}f(x) \nonumber\\
&\lesssim\esssup_{k\in\bbbz}\sum_{j\geq 0}2^{-js}\cdot 2^{(k+j)s}\esssup_{1\leq\tau<2}S^L_{\tau 2^{-k-j}}f(x) \nonumber\\
&\lesssim\big(\sum_{j\geq 0}2^{-js}\big)\esssup_{k\in\bbbz}2^{ks}\esssup_{1\leq\tau<2}S^L_{\tau 2^{-k}}f(x),
\end{align}
therefore (\ref{eq352}) can be estimated from above by $\|f\|_{\Fspq}$. Using the same method, we can also estimate (\ref{eq354}) from above by $\|f\|_{\Fspq}$ under the conditions of Theorem \ref{theorem4}. As for (\ref{eq355}), by using the same method as (\ref{eq370}) we can show that
\begin{equation}\label{eq380}
\|\{2^{ks}\esssup_{1\leq|h|<2}D^L_{2^{-k}h}f(x)\}_{k\in\bbbz}\|_{L^p(l^q)}\lesssim\|f\|_{\Fspq}
\end{equation}
when $p,q,s,r,L$ satisfy conditions of Theorem \ref{theorem4}. And then we use the arguments in (\ref{eq378}) and (\ref{eq379}) to prove that
\begin{equation}\label{eq381}
\|\{2^{ks}\esssup_{0<|h|<2}D^L_{2^{-k}h}f(x)\}_{k\in\bbbz}\|_{L^p(l^q)}\lesssim
\|\{2^{ks}\esssup_{1\leq|h|<2}D^L_{2^{-k}h}f(x)\}_{k\in\bbbz}\|_{L^p(l^q)}
\end{equation}
for all $0<p<\infty$, $0<q\leq\infty$, since we have the decomposition
$$\{h\in\bbbr^n:0<|h|<2\}=\bigcup_{j=0}^{\infty}\{h\in\bbbr^n:2^{-j}\leq|h|<2^{1-j}\}.$$

To prove the reverse directions, we first show that for any $0<\tau<2$, $0<p<\infty$, $0<q\leq\infty$ and $s\in\bbbr$, we have the estimate
\begin{equation}\label{eq382}
\|f\|_{\Fspq}\lesssim\|\{2^{ks}S^L_{\tau 2^{-k}}f(x)\}_{k\in\bbbz}\|_{L^p(l^q)}.
\end{equation}
Let $\tau\in(0,2)$ be fixed for now, and let $a$ denote the tempered distribution in $\Sw'(\bbbr^n)$ whose distributional Fourier transform is the function below
\begin{equation}\label{eq383}
\FT_n a(\xi)=\mvint_{\unitsph}(e^{2\pi i\tau\xi\cdot z}-1)^L d\Haus^{n-1}(z).
\end{equation}
For example, we can choose $a=\sum_{m=0}^L\binom{L}{m}(-1)^{L-m}\mvint_{\unitsph}\delta_{-m\tau z}d\Haus^{n-1}(z)$, where $\delta_{-m\tau z}$ is the Dirac mass at $-m\tau z$. Then from (\ref{eq344}) and Definition \ref{definition5} we deduce the following equality
\begin{equation}\label{eq384}
S^L_{\tau 2^{-k}}f(x)=\esssup_{y\in\bbbr^n}\frac{|\iFT_n(\FT_n a(2^{-k}\xi)\cdot\FT_n f(\xi))(x-y)|}{(1+\tau^{-1}2^k|y|)^{\frac{n}{r}}}.
\end{equation}
Using the formula given in Appendix D.3 of \cite{14classical}, we have
\begin{equation}\label{eq385}
\FT_n a(\xi)=C_n\cdot\int_{-1}^1 (e^{2\pi it\tau|\xi|}-1)^L(1-t^2)^{\frac{n-3}{2}}dt,
\end{equation}
where $C_n$ is a positive constant depending on $n$. By using Taylor expansion, we can write
\begin{equation}\label{eq386}
(e^{2\pi it\tau|\xi|}-1)^L=\sum_{k=0}^{\infty}A_{L+k}(t\tau|\xi|)^{L+k},
\end{equation}
and each $A_{L+k}$ is a complex number whose value is independent of $t,\tau,\xi$ and satisfies $|A_{L+k}|>0$. Hence we have the expression
\begin{equation}\label{eq387}
\FT_n a(\xi)=C_n\cdot\sum_{k=0}^{\infty}A_{L+k}B_{L+k}|\tau\xi|^{L+k}\text{ for every $\xi\in\bbbr^n$ when $|\xi|$ is small,}
\end{equation}
where
\begin{equation}\label{eq388}
B_{L+k}=\int_{-1}^1 t^{L+k}(1-t^2)^{\frac{n-3}{2}}dt\qquad\text{for }k\geq 0,
\end{equation}
and $B_{L+k}=0$ if $L+k$ is an odd integer, $B_{L+k}>0$ if $L+k$ is an even integer. If $L$ is an even integer, then
\begin{equation*}
\FT_n a(\xi)=C_n A_L B_L|\tau\xi|^L(1+O(|\tau\xi|^2))
\end{equation*}
and $|\FT_n a(\xi)|\sim|\xi|^L>0$ if $|\xi|>0$ is sufficiently small. If $L$ is an odd integer, then
\begin{equation*}
\FT_n a(\xi)=C_n A_{L+1} B_{L+1}|\tau\xi|^{L+1}(1+O(|\tau\xi|^2))
\end{equation*}
and $|\FT_n a(\xi)|\sim|\xi|^{L+1}>0$ if $|\xi|>0$ is sufficiently small. Therefore we can pick a sufficiently large positive integer $m_1$ so that $|\FT_n a(\xi)|>0$ if $0<|\xi|<2^{1-m_1}$, and hence $\frac{\FT_n\psi(2^{m_1}\xi)}{\FT_n a(\xi)}$ is a well-defined function in $C_c^{\infty}(\bbbr^n)$ and $\iFT_n(\frac{\FT_n\psi(2^{m_1}\xi)}{\FT_n a(\xi)})(\cdot)\in\Sw(\bbbr^n)$. Furthermore by using (\ref{eq384}) we have for each $k\in\bbbz$
\begin{align}\label{eq389}
&|\iFT_n(\FT_n\psi(2^{m_1-k}\xi)\FT_n f(\xi))(x)| \nonumber\\
&\lesssim\int_{\bbbr^n}|\iFT_n(\frac{\FT_n\psi(2^{m_1-k}\xi)}{\FT_n a(2^{-k}\xi)})(y)|\cdot
          |\iFT_n(\FT_n a(2^{-k}\xi)\FT_n f(\xi))(x-y)|dy \nonumber\\
&\lesssim\int_{\bbbr^n}|\iFT_n(\frac{\FT_n\psi(2^{m_1}\xi)}{\FT_n a(\xi)})(2^k y)|\cdot 2^{kn}(1+\tau^{-1}2^k|y|)^{\frac{n}{r}}dy
          \cdot S^L_{\tau 2^{-k}}f(x) \nonumber\\
&\lesssim S^L_{\tau 2^{-k}}f(x),
\end{align}
and the constants are independent of $k\in\bbbz$. By using (\ref{eq389}) above, we reach the conclusion that for $0<p<\infty$, $0<q\leq\infty$ and $s\in\bbbr$
\begin{align}\label{eq390}
\|f\|_{\Fspq}
&=2^{-sm_1}\|\{2^{ks}\iFT_n(\FT_n\psi(2^{m_1-k}\xi)\FT_n f(\xi))(x)\}_{k\in\bbbz}\|_{L^p(l^q)} \nonumber\\
&\lesssim\|\{2^{ks}S^L_{\tau 2^{-k}}f(x)\}_{k\in\bbbz}\|_{L^p(l^q)}.
\end{align}
We let $\tau=1$ in (\ref{eq390}) and get that $\|f\|_{\Fspq}$ can be estimated from above by (\ref{eq351}). Therefore we have shown that (\ref{eq351}) and (\ref{eq352}) are equivalent quasinorms in $\Fspq$ when parameters $p,q,s,r,L$ satisfy the conditions of Theorem \ref{theorem4}.

To show that when $0<p<\infty$, $0<q\leq\infty$ and $s\in\bbbr$, the quasinorm $\|f\|_{\Fspq}$ can be estimated from above by (\ref{eq353}), we consider the tempered distribution $b\in\Sw'(\bbbr^n)$ whose distributional Fourier transform is the function below
\begin{equation}\label{eq391}
\FT_n b(\xi)=\mvint_A (e^{2\pi i\xi\cdot z}-1)^L dz
=\frac{1}{|A|}\int_1^2 \tau^{n-1}\int_{\unitsph}(e^{2\pi i\tau\xi\cdot z}-1)^L d\Haus^{n-1}(z)d\tau.
\end{equation}
For example, we can choose $b=\sum_{m=0}^L\binom{L}{m}(-1)^{L-m}\mvint_{A}\delta_{-mz}dz$, where $\delta_{-mz}$ is the Dirac mass at $-mz$. Using (\ref{eq383}), (\ref{eq385}), (\ref{eq386}) and (\ref{eq387}), we obtain that
\begin{equation}\label{eq392}
\FT_n b(\xi)
=C'_n\sum_{k=0}^{\infty}\int_1^2\tau^{n+k+L-1}d\tau\cdot A_{L+k}B_{L+k}|\xi|^{L+k}
=C'_n\sum_{k=0}^{\infty}A'_{L+k}B_{L+k}|\xi|^{L+k}
\end{equation}
for every $\xi\in\bbbr^n$ when $|\xi|$ is small, where $C'_n$ is a positive constant depending on $n$, each $A'_{L+k}$ is a complex number satisfying $|A'_{L+k}|>0$, and each $B_{L+k}$ is defined by (\ref{eq388}). Therefore using a similar analysis like the one for $\FT_n a(\xi)$, we can find a sufficiently large positive integer $m_2$ so that $|\FT_n b(\xi)|>0$ if $0<|\xi|<2^{1-m_2}$. Hence $\frac{\FT_n\psi(2^{m_2}\xi)}{\FT_n b(\xi)}$ is a well-defined function in $C_c^{\infty}(\bbbr^n)$ and $\iFT_n(\frac{\FT_n\psi(2^{m_2}\xi)}{\FT_n b(\xi)})(\cdot)\in\Sw(\bbbr^n)$. Using a similar argument like the one to deduce (\ref{eq389}) and the estimate
\begin{equation}\label{eq393}
|\iFT_n(\FT_n b(2^{-k}\xi)\FT_n f(\xi))(x\!-\!y)|\!=\!|\mvint_A \varDelta^L_{2^{-k}z}f(x\!-\!y)dz|
\!\leq\!(1\!+\!2^k|y|)^{\frac{n}{r}}\cdot V^L_{2^{-k}}f(x),
\end{equation}
we can obtain
\begin{equation}\label{eq394}
|\iFT_n(\FT_n\psi(2^{m_2-k}\xi)\FT_n f(\xi))(x)|\lesssim V^L_{2^{-k}}f(x)
\qquad\text{for every }k\in\bbbz,
\end{equation}
and the constant is independent of $k$. By using (\ref{eq394}) above, we reach the conclusion that for $0<p<\infty$, $0<q\leq\infty$ and $s\in\bbbr$
\begin{align}\label{eq395}
\|f\|_{\Fspq}
&=2^{-sm_2}\|\{2^{ks}\iFT_n(\FT_n\psi(2^{m_2-k}\xi)\FT_n f(\xi))(x)\}_{k\in\bbbz}\|_{L^p(l^q)} \nonumber\\
&\lesssim\|\{2^{ks}V^L_{2^{-k}}f(x)\}_{k\in\bbbz}\|_{L^p(l^q)}.
\end{align}
By using their defining expressions in Definition \ref{definition5}, it is easy to see that $V^L_{2^{-k}}f(x)\lesssim\esssup_{0<|h|<2}D^L_{2^{-k}h}f(x)$ for every $x\in\bbbr^n$ and $k\in\bbbz$, thus $\|f\|_{\Fspq}$ can be estimated from above by (\ref{eq355}) for all $0<p<\infty$, $0<q\leq\infty$ and $s\in\bbbr$. Hereby we conclude the proof of Theorem \ref{theorem4}.
\end{proof}

\section{Besov-Lipschitz space and maximal functions of means of differences}\label{proof.of.theorem5}
\begin{proof}[Proof Of Theorem \ref{theorem5}]
The proof of Theorem \ref{theorem5} is alike to the proof of Theorem \ref{theorem4} and thus we will just sketch it. In this proof we still consider the number $r$ appearing in Definition \ref{definition5} equals the number $r$ appearing in Definition \ref{definition2}. We first prove the counterpart of (\ref{eq370}), that is,
\begin{equation}\label{eq396}
\|\{2^{ks}\esssup_{1\leq\tau<2}S^L_{\tau 2^{-k}}f\}_{k\in\bbbz}\|_{l^q(L^p)}\lesssim\|f\|_{\Bspq}
\end{equation}
when $0<p\leq\infty$, $0<q\leq\infty$, $\frac{n}{p}<s<L$ and for every $r\in\bbbr$ with $\frac{n}{s}<r<p$. By using Lemma \ref{lemma1}, we still have (\ref{eq372}) with $S_k^*f(x)$ given in (\ref{eq371}). If $1\leq p\leq\infty$, we use Minkowski's inequality for $L^p(\bbbr^n)$-norms, Remark \ref{remark12} and the mapping property of Hardy-Littlewood maximal function in a sequence and obtain the following
\begin{equation}\label{eq397}
\|S_k^*f\|_{L^p(\bbbr^n)}\lesssim
\sum_{j\leq k}2^{(j-k)L}\|f_j\|_{L^p(\bbbr^n)}+\sum_{j>k}2^{(j-k)\frac{n}{r}}\|f_j\|_{L^p(\bbbr^n)}.
\end{equation}
With (\ref{eq397}), we use the calculation method of (\ref{eq375}) when $0<q<\infty$ and the calculation method of (\ref{eq376}) when $q=\infty$ and justify the decomposition in a similar way like (\ref{eq542}), then we can obtain (\ref{eq396}) for the case $1\leq p\leq\infty$. If $0<p<1$, we raise the power of both sides of (\ref{eq372}) to $p$ and integrate over $\bbbr^n$ with respect to $x$, use Remark \ref{remark12} and the mapping property of Hardy-Littlewood maximal function in a sequence and obtain the following
\begin{equation}\label{eq398}
\|S_k^*f\|_{L^p(\bbbr^n)}^p\lesssim
\sum_{j\leq k}2^{(j-k)Lp}\|f_j\|_{L^p(\bbbr^n)}^p+\sum_{j>k}2^{(j-k)\frac{np}{r}}\|f_j\|_{L^p(\bbbr^n)}^p.
\end{equation}
With (\ref{eq398}), we use the calculation method of (\ref{eq375}) when $0<\frac{q}{p}<\infty$ and the calculation method of (\ref{eq376}) when $\frac{q}{p}=\infty$ and justify the decomposition then we can obtain (\ref{eq396}) for the case $0<p<1$. Next, we show that
\begin{equation}\label{eq399}
\|\{2^{ks}\esssup_{0<\tau<2}S^L_{\tau 2^{-k}}f\}_{k\in\bbbz}\|_{l^q(L^p)}\lesssim
\|\{2^{ks}\esssup_{1\leq\tau<2}S^L_{\tau 2^{-k}}f\}_{k\in\bbbz}\|_{l^q(L^p)}
\end{equation}
for $0<p\leq\infty$, $0<q\leq\infty$ and $s\in\bbbr$. We have the following pointwise estimate
\begin{equation}\label{eq400}
\esssup_{0<\tau<2}S^L_{\tau 2^{-k}}f(x)\lesssim\sum_{j=0}^{\infty}
\esssup_{1\leq\tau<2}S^L_{\tau 2^{-k-j}}f(x)\qquad\text{for every }x\in\bbbr^n.
\end{equation}
If $1\leq p\leq\infty$, then we use Minkowski's inequality for $L^p(\bbbr^n)$-norms and get
\begin{equation}\label{eq401}
\|\{2^{ks}\!\esssup_{0<\tau<2}S^L_{\tau 2^{-k}}f\}_{k\in\bbbz}\|_{l^q(L^p)}\!\lesssim\!
\|\{ 2^{ks}\!\sum_{j=0}^{\infty}\|\esssup_{1\leq\tau<2}S^L_{\tau 2^{-k-j}}f\|_{L^p(\bbbr^n)}\}_{k\in\bbbz}\|_{l^q}.
\end{equation}
When $0<q<1$, we can switch the order of summation and obtain
\begin{align}\label{eq402}
(\ref{eq401})
&\lesssim(\sum_{k\in\bbbz}2^{ksq}\sum_{j=0}^{\infty}
          \|\esssup_{1\leq\tau<2}S^L_{\tau 2^{-k-j}}f\|_{L^p(\bbbr^n)}^q)^{\frac{1}{q}} \nonumber\\
&=(\sum_{j=0}^{\infty}2^{-jsq}\sum_{k\in\bbbz}2^{(k+j)sq}
          \|\esssup_{1\leq\tau<2}S^L_{\tau 2^{-k-j}}f\|_{L^p(\bbbr^n)}^q)^{\frac{1}{q}} \nonumber\\
&\lesssim\|\{2^{ks}\esssup_{1\leq\tau<2}S^L_{\tau 2^{-k}}f\}_{k\in\bbbz}\|_{l^q(L^p)}.
\end{align}
When $1\leq q\leq\infty$, we use Minkowski's inequality for $l^q$-norms and obtain
\begin{align}\label{eq403}
(\ref{eq401})
&\lesssim\sum_{j=0}^{\infty}2^{-js}\|\{2^{(k+j)s}
          \|\esssup_{1\leq\tau<2}S^L_{\tau 2^{-k-j}}f\|_{L^p(\bbbr^n)}\}_{k\in\bbbz}\|_{l^q} \nonumber\\
&\lesssim\|\{2^{ks}\esssup_{1\leq\tau<2}S^L_{\tau 2^{-k}}f\}_{k\in\bbbz}\|_{l^q(L^p)}.
\end{align}
If $0<p<1$, then we raise the power of both sides of (\ref{eq400}) to $p$ and integrate over $\bbbr^n$ with respect to $x$ to obtain
\begin{equation}\label{eq404}
\|\esssup_{0<\tau<2}S^L_{\tau 2^{-k}}f\|_{L^p(\bbbr^n)}^p\lesssim
\sum_{j=0}^{\infty}\|\esssup_{1\leq\tau<2}S^L_{\tau 2^{-k-j}}f\|_{L^p(\bbbr^n)}^p.
\end{equation}
When $0<\frac{q}{p}<1$, we use (\ref{eq404}) to obtain
\begin{align}\label{eq405}
&\|\{2^{ks}\esssup_{0<\tau<2}S^L_{\tau 2^{-k}}f\}_{k\in\bbbz}\|_{l^q(L^p)} \nonumber\\
&\lesssim(\sum_{k\in\bbbz}(\sum_{j=0}^{\infty}2^{ksp}
          \|\esssup_{1\leq\tau<2}S^L_{\tau 2^{-k-j}}f\|_{L^p(\bbbr^n)}^p)^{\frac{q}{p}})^{\frac{1}{q}} \nonumber\\
&\lesssim(\sum_{k\in\bbbz}\sum_{j=0}^{\infty}2^{ksq}
          \|\esssup_{1\leq\tau<2}S^L_{\tau 2^{-k-j}}f\|_{L^p(\bbbr^n)}^q )^{\frac{1}{q}} \nonumber\\
&=(\sum_{j=0}^{\infty}2^{-jsq}\sum_{k\in\bbbz}2^{(k+j)sq}
          \|\esssup_{1\leq\tau<2}S^L_{\tau 2^{-k-j}}f\|_{L^p(\bbbr^n)}^q )^{\frac{1}{q}} \nonumber\\
&\lesssim\|\{2^{ks}\esssup_{1\leq\tau<2}S^L_{\tau 2^{-k}}f\}_{k\in\bbbz}\|_{l^q(L^p)}.
\end{align}
When $1\leq\frac{q}{p}\leq\infty$, we use (\ref{eq404}) and Minkowski's inequality for $l^{\frac{q}{p}}$-norms and obtain
\begin{align}\label{eq406}
&\|\{2^{ks}\esssup_{0<\tau<2}S^L_{\tau 2^{-k}}f\}_{k\in\bbbz}\|_{l^q(L^p)} \nonumber\\
&\lesssim\|\{\sum_{j=0}^{\infty}2^{ksp}\|\esssup_{1\leq\tau<2}S^L_{\tau 2^{-k-j}}f\|_{L^p(\bbbr^n)}^p
          \}_{k\in\bbbz}\|_{l^{\frac{q}{p}}}^{\frac{1}{p}} \nonumber\\
&\lesssim(\sum_{j=0}^{\infty}2^{-jsp}\|\{2^{(k+j)s}\|\esssup_{1\leq\tau<2}S^L_{\tau 2^{-k-j}}f\|_{L^p(\bbbr^n)}
          \}_{k\in\bbbz}\|_{l^q}^p)^{\frac{1}{p}} \nonumber\\
&\lesssim\|\{2^{ks}\esssup_{1\leq\tau<2}S^L_{\tau 2^{-k}}f\}_{k\in\bbbz}\|_{l^q(L^p)}.
\end{align}
From (\ref{eq401}), (\ref{eq402}), (\ref{eq403}), (\ref{eq405}) and (\ref{eq406}), we see that (\ref{eq399}) has been proved. Combining (\ref{eq396}) and (\ref{eq399}) gives that (\ref{eq357}) can be estimated from above by $\|f\|_{\Bspq}$ when conditions of Theorem \ref{theorem5} are satisfied. Using the same method, we also prove that (\ref{eq359}) and (\ref{eq360}) can be estimated from above by $\|f\|_{\Bspq}$ under the conditions of Theorem \ref{theorem5}.

To prove the reverse directions, we just notice that (\ref{eq389}) and (\ref{eq394}) are pointwise estimates for every $x\in\bbbr^n$ and then we use the same method given in (\ref{eq390}) and (\ref{eq395}) to prove that $\|f\|_{\Bspq}$ can be estimated from above by (\ref{eq356}) and (\ref{eq358}) for all $0<p\leq\infty$, $0<q\leq\infty$ and $s\in\bbbr$. Furthermore (\ref{eq358}) can be estimated from above by (\ref{eq360}) by using Definition \ref{definition5}. The proof of Theorem \ref{theorem5} is complete.
\end{proof}

\section{Triebel-Lizorkin space and iterated differences}\label{proof.of.theorem2}
\begin{proof}[Proof Of Theorem \ref{theorem2}]
We first prove Theorem \ref{theorem2} (i). Let $f\in\Fspq$ be an element of $\Sw'(\bbbr^n)$. We note that $\tilde{\sigma}_{pq}<s$ implies $\frac{nq}{n+sq}<p$. And we recall the notation $A_k=\{h\in\bbbr^n:2^{-k}\leq|h|<2^{1-k}\}$ for $k\in\bbbz$. For $|h|\lesssim 2^{-k}$ and $f_j=\psi_{2^{-j}}*f$, we deduce two estimates for $|(\varDelta^L_h f_j)(x)|$. Using mean value theorem and the iteration formula (\ref{eq259}) consecutively, we get
\begin{equation}\label{eq267}
|(\varDelta^L_h f_j)(x)|\lesssim
\sum_{|\alpha|=L}|\partial^{\alpha}f_j(x+\sum^L_{l=1}t_{\alpha,l}h)|\cdot|h|^L,
\end{equation}
where $\alpha$ represents a multi-index and each $t_{\alpha,l}$ is in $(0,1)$. Since the $n$-dimensional distributional Fourier transform $\FT_n f_j$ is supported in $\{\xi\in\bbbr^n:2^{j-1}\leq|\xi|<2^{j+1}\}$, we use Remark \ref{remark4} to get
\begin{equation}\label{eq268}
|\partial^{\alpha}f_j(x+\sum^L_{l=1}t_{\alpha,l}h)|\lesssim
\PFSmax_n (\partial^{\alpha}f_j)(x+\sum^L_{l=1}t_{\alpha,l}h)\lesssim
2^{jL} \PFSmax_n f_j(x+\sum^L_{l=1}t_{\alpha,l}h).
\end{equation}
Since $|\sum^L_{l=1}t_{\alpha,l}h|\lesssim L2^{-k}$, by (\ref{eq70}) of Remark \ref{remark2} we have
\begin{equation}\label{eq269}
\PFSmax_n f_j(x+\sum^L_{l=1}t_{\alpha,l}h)\lesssim\PFSmax_n f_j(x)\cdot(1+L2^{j-k})^{n/r},
\end{equation}
where $r$ is the chosen positive number in Definition \ref{definition2} and satisfies $0<r<\min\{p,q\}$. We infer from (\ref{eq267}), (\ref{eq268}) and (\ref{eq269}) the first estimate
\begin{equation}\label{eq270}
|(\varDelta^L_h f_j)(x)|\lesssim 2^{(j-k)L}(1+L2^{j-k})^{n/r} \PFSmax_n f_j(x)\quad\text{for}\quad|h|\lesssim 2^{-k},
\end{equation}
and the constant is independent of $h\in\bbbr^n$, $j,k\in\bbbz$. Also by using (\ref{eq260}), we get
\begin{equation}\label{eq271}
|(\varDelta^L_h f_j)(x)|\lesssim\sum^L_{l=0}|f_j(x+lh)|.
\end{equation}
If $0\leq l\leq L$, $|h|\lesssim 2^{-k}$ and $j>k$, we recall that $0<r<\min\{p,q\}$ and the value of $r$ will be determined later, then using Remark \ref{remark2}, Lemma \ref{lemma4} and the definition of Peetre-Fefferman-Stein maximal function, we obtain
\begin{equation}\label{eq272}
\PFSmax_n f_j(x+lh)\lesssim(1+2^{j+1}l|h|)^{n/r}\PFSmax_n f_j(x)\lesssim 2^{(j-k)n/r}\HLmax_n(|f_j|^r)(x)^{1/r},
\end{equation}
and the constant in (\ref{eq272}) is independent of $h\in\bbbr^n$, $0\leq l\leq L$ and $j,k\in\bbbz$. Using the proper change of variable, we also have
\begin{equation}\label{eq543}
2^{kn}\int_{A_k}|f_j(x+lh)|^r dh\lesssim\mvint_{l2^{-k}\leq|y|<l2^{-k+1}}|f_j(x+y)|^r dy\lesssim\HLmax_n(|f_j|^r)(x)
\end{equation}
for $0<l\leq L$, and $|f_j(x)|^r\leq\HLmax_n(|f_j|^r)(x)$ by Lebesgue's differentiation theorem. Applying (\ref{eq272}) and (\ref{eq543}), we can obtain the second estimate
\begin{align}\label{eq544}
&2^{kn}\int_{A_k}|(\varDelta^L_h f_j)(x)|^q dh \nonumber\\
&\lesssim\sum_{l=0}^L 2^{kn}\int_{A_k}|f_j(x+lh)|^r\cdot|f_j(x+lh)|^{q-r}dh \nonumber\\
&\lesssim\sum_{l=0}^L 2^{kn}\int_{A_k}|f_j(x+lh)|^r dh\cdot\PFSmax_n f_j(x+lh)^{q-r} \nonumber\\
&\lesssim 2^{(j-k)n(\frac{q}{r}-1)}\HLmax_n(|f_j|^r)(x)^{\frac{q}{r}}.
\end{align}
And estimate (\ref{eq544}) is true for $0<q<\infty$ and $j>k$. Now we consider estimating the expression below
\begin{equation}\label{eq460}
(\sum_{k\in\bbbz}2^{k(sq+n)}\int_{A_k}(\sum_{j\in\bbbz}|(\varDelta^L_{h}f_j)(x)|)^q dh)^{\frac{1}{q}}
\end{equation}
by the sum of the following two expressions,
\begin{equation}\label{eq274}
(\sum_{k\in\bbbz}2^{k(sq+n)}\int_{A_k}(\sum_{j\leq k}|(\varDelta^L_{h}f_j)(x)|)^q dh)^{\frac{1}{q}}
\end{equation}
and
\begin{equation}\label{eq275}
(\sum_{k\in\bbbz}2^{k(sq+n)}\int_{A_k}(\sum_{j>k}|(\varDelta^L_{h}f_j)(x)|)^q dh)^{\frac{1}{q}}.
\end{equation}
For $0<q<\infty$, we pick $0<\varepsilon<\min\{s,L-s\}$ and the value of $\varepsilon$ will be determined later. Then we have
\begin{align}\label{eq276}
(\sum_{j\leq k}|(\varDelta^L_{h}f_j)(x)|)^q
&=(\sum_{j\leq k}2^{j\varepsilon}\cdot 2^{-j\varepsilon}|(\varDelta^L_{h}f_j)(x)|)^q \nonumber\\
&\lesssim(\sum_{j\leq k}2^{j\varepsilon})^q\cdot\esssup_{j\leq k} 2^{-jq\varepsilon}|(\varDelta^L_{h}f_j)(x)|^q \nonumber\\
&\lesssim 2^{kq\varepsilon}\cdot\sum_{j\leq k} 2^{-jq\varepsilon}|(\varDelta^L_{h}f_j)(x)|^q,
\end{align}
and
\begin{align}\label{eq277}
(\sum_{j>k}|(\varDelta^L_{h}f_j)(x)|)^q
&=(\sum_{j>k}2^{-j\varepsilon}\cdot 2^{j\varepsilon}|(\varDelta^L_{h}f_j)(x)|)^q \nonumber\\
&\lesssim(\sum_{j>k}2^{-j\varepsilon})^q\cdot\esssup_{j>k} 2^{jq\varepsilon}|(\varDelta^L_{h}f_j)(x)|^q \nonumber\\
&\lesssim 2^{-kq\varepsilon}\cdot\sum_{j>k} 2^{jq\varepsilon}|(\varDelta^L_{h}f_j)(x)|^q.
\end{align}
Using (\ref{eq276}) and (\ref{eq270}), we can estimate (\ref{eq274}) from above by
\begin{align}\label{eq278}
&(\sum_{k\in\bbbz}\sum_{j\leq k}2^{k(n+qs+q\varepsilon)}2^{-jq\varepsilon}
       \int_{A_k}|(\varDelta^L_{h}f_j)(x)|^q dh)^{\frac{1}{q}} \nonumber\\
&\lesssim(\sum_{k\in\bbbz}\sum_{j\leq k}2^{k(qs+q\varepsilon)}2^{-jq\varepsilon}2^{(j-k)Lq}
       (1+L2^{j-k})^{\frac{nq}{r}} \PFSmax_n f_j(x)^q)^{\frac{1}{q}}.
\end{align}
We notice that $(1+L2^{j-k})^{nq/r}\lesssim C$ if $j\leq k$ and $C$ is a constant determined by $n,q,r,L$ and we switch the order of summation to obtain
\begin{align}\label{eq279}
(\ref{eq278})
&\lesssim(\sum_{j\in\bbbz}\sum_{k\geq j}2^{kq(s+\varepsilon-L)}2^{-jq\varepsilon+jqL}\PFSmax_n f_j(x)^q)^{\frac{1}{q}} \nonumber\\
&=(\sum_{j\in\bbbz}2^{jqs}\PFSmax_n f_j(x)^q)^{\frac{1}{q}} \nonumber\\
&\lesssim(\sum_{j\in\bbbz}2^{jqs}\HLmax_n(|f_j|^r)(x)^{\frac{q}{r}})^{\frac{1}{q}},
\end{align}
where we also used Remark \ref{remark12} and the condition that $\varepsilon<L-s$. Using (\ref{eq277}) and (\ref{eq544}) and switching the order of summation, we can estimate (\ref{eq275}) from above by
\begin{align}\label{eq284}
&(\sum_{k\in\bbbz}\sum_{j>k}2^{k(n+qs-q\varepsilon)}2^{jq\varepsilon}
          \int_{A_k}|(\varDelta^L_{h}f_j)(x)|^q dh)^{\frac{1}{q}} \nonumber\\
&\lesssim(\sum_{k\in\bbbz}\sum_{j>k}2^{kq(s-\varepsilon)+jq\varepsilon}\cdot
          2^{(j-k)n(\frac{q}{r}-1)}\HLmax_n(|f_j|^r)(x)^{\frac{q}{r}})^{\frac{1}{q}} \nonumber\\
&\lesssim(\sum_{j\in\bbbz}\sum_{k<j}2^{kq[s-\varepsilon-n(\frac{1}{r}-\frac{1}{q})]}\cdot
          2^{jq[\varepsilon+n(\frac{1}{r}-\frac{1}{q})]}\HLmax_n(|f_j|^r)(x)^{\frac{q}{r}})^{\frac{1}{q}}.
\end{align}
If $q\leq p<\infty$, we have
$$\lim_{\varepsilon\rightarrow 0,r\rightarrow q}s-\varepsilon-n(\frac{1}{r}-\frac{1}{q})=s>0.$$
If $\frac{nq}{n+sq}<p<q$, we have
$$\lim_{\varepsilon\rightarrow 0,r\rightarrow p}s-\varepsilon-n(\frac{1}{r}-\frac{1}{q})=s-n(\frac{1}{p}-\frac{1}{q})>0.$$
Therefore if we pick $\varepsilon$ sufficiently small and $r$ sufficiently close to $\min\{p,q\}$, then we can make $s-\varepsilon-n(\frac{1}{r}-\frac{1}{q})$ a positive finite number and hence
\begin{equation}\label{eq285}
\sum_{k<j}2^{kq[s-\varepsilon-n(\frac{1}{r}-\frac{1}{q})]}\lesssim 2^{jq[s-\varepsilon-n(\frac{1}{r}-\frac{1}{q})]}.
\end{equation}
Inserting (\ref{eq285}) into (\ref{eq284}) yields
\begin{equation}\label{eq545}
(\sum_{k\in\bbbz}2^{k(sq+n)}\int_{A_k}(\sum_{j>k}|(\varDelta^L_{h}f_j)(x)|)^q dh)^{\frac{1}{q}}\lesssim
(\sum_{j\in\bbbz}2^{jqs}\HLmax_n(|f_j|^r)(x)^{\frac{q}{r}})^{\frac{1}{q}}.
\end{equation}
Combining (\ref{eq274}), (\ref{eq278}), (\ref{eq279}), (\ref{eq275}) and (\ref{eq545}) and also invoking Lemma \ref{lemma6}, we can obtain
\begin{equation}\label{eq546}
\|(\sum_{k\in\bbbz}2^{k(sq+n)}\int_{A_k}(\sum_{j\in\bbbz}|(\varDelta^L_{h}f_j)(x)|)^q dh)^{\frac{1}{q}}\|_{L^p(\bbbr^n)}
%\lesssim\|(\sum_{j\in\bbbz}2^{jqs}\HLmax_n(|f_j|^r)(x)^{\frac{q}{r}})^{\frac{1}{q}}\|_{L^p(\bbbr^n)}
\lesssim\|f\|_{\Fspq},
\end{equation}
when $0<p,q<\infty$ and $\tilde{\sigma}_{pq}<s<L$. From the assumption $f\in\Fspq$, we know inequality (\ref{eq546}) also shows $\sum_{j\in\bbbz}|(\varDelta^L_{h}f_j)(x)|<\infty$ for every $k\in\bbbz$ and for almost every $x\in\bbbr^n,h\in A_k$. Together with (\ref{eq513}), we have reached the conclusion that
\begin{equation}\label{eq448}
\varDelta^L_h f=\sum_{j\in\bbbz}\varDelta^L_h f_j(x)\text{ in the sense of }\Sw'(\bbbr^n)/\mathscr{P}(\bbbr^n)
\end{equation}
for every $k\in\bbbz$ and almost every $h\in A_k, x\in\bbbr^n$, and the tempered distribution $\varDelta^L_h f$ has a function representative which is the pointwise limit of the series $\sum_{j\in\bbbz}\varDelta^L_h f_j(x)$. Furthermore, integration of $\varDelta^L_h f$ with respect to the Lebesgue measure is justified, and the inequality
\begin{align}\label{eq547}
&\|(\int_{\bbbr^n}|h|^{-sq}|\varDelta^L_{h}f|^q\frac{dh}{|h|^n})^{\frac{1}{q}}\|_{L^p(\bbbr^n)} \nonumber\\
&\lesssim\|(\sum_{k\in\bbbz}2^{k(sq+n)}\int_{A_k}(\sum_{j\in\bbbz}|(\varDelta^L_{h}f_j)(x)|)^q dh)^{\frac{1}{q}}\|_{L^p(\bbbr^n)}
\end{align}
is also validated. Therefore the proof of Theorem \ref{theorem2} (i) is now complete.

Now we prove Theorem \ref{theorem2} (ii) when $f$ is a function, $0<p<\infty$, $0<q<1$ and $\sigma_{pq}+\tilde{\sigma}_{pq}<s<\infty$. Without loss of generality, we also assume the right side of (\ref{eq415}) is finite, otherwise inequality (\ref{eq415}) is trivial. To do this, recall that $spt.\FT_n\psi\subseteq A'=\{\xi\in\bbbr^n:\frac{1}{2}\leq|\xi|<2\}$ and by Taylor expansion of $e^{2\pi it}$, we have
\begin{equation}\label{eq416}
(e^{2\pi it}-1)^L=(2\pi it)^L(1+O(2\pi it))
\end{equation}
and there exists a sufficiently large positive integer $m_0$ such that
\begin{equation}\label{eq417}
0<|t|<2^{2-m_0}\text{ implies }|(e^{2\pi it}-1)^L|>0.
\end{equation}
For a unit vector $\theta\in\unitsph$, we can find $\delta>0$ so small that if $\xi\in A'\subseteq\bbbr^n$ and $\frac{1}{4}\leq|\theta\cdot\xi|<2$, then for all other $\theta'$ in the spherical cap $C_{\theta}:=\{\theta'\in\unitsph:|\theta'-\theta|<\delta\}$, we also have $\frac{1}{4}\leq|\theta'\cdot\xi|<2$. We choose properly distributed unit vectors $\theta_1,\theta_2,\cdots,\theta_M$ where $M\in\bbbn$ is sufficiently large so that the spherical caps $C_1,C_2,\cdots,C_M$, respectively associated with $\theta_1,\theta_2,\cdots,\theta_M$ in the above way, cover the unit sphere $\unitsph$. For each cap $C_l$, $1\leq l\leq M$, we consider the set
\begin{equation}\label{eq418}
P_l:=\{\xi\in\bbbr^n:\frac{1}{2}\leq|\xi|<2,\frac{\xi}{|\xi|}\in C_l\},
\end{equation}
then from the construction of $\{C_l\}_{l=1}^M$, we have that
\begin{equation}\label{eq419}
\frac{1}{4}\leq|\theta\cdot\xi|<2\qquad\text{for all }\xi\in P_l\text{ and }\theta\in C_l
\end{equation}
and
\begin{equation}\label{eq420}
\bigcup_{l=1}^M P_l=A'.
\end{equation}
We use a partition of unity associated with $\{P_l\}_{l=1}^M$ by smooth functions $\{\rho_l\}_{l=1}^M$ with compact supports and $\{\rho_l\}_{l=1}^M$ also satisfy
\begin{equation}\label{eq421}
\sum_{l=1}^M\rho_l(\xi)=1\text{ if }\xi\in A'\qquad\text{and}\qquad spt.\rho_l\bigcap A'\subseteq P_l\text{ for each }l.
\end{equation}
Recall the definition of $\FT_n\phi$ given in (\ref{eq345}), we pick a large positive integer $J>m_0$ and the value of $J$ will be determined later, then we have for each $k\in\bbbz$
\begin{equation}\label{eq422}
\FT_n\phi(2^{m_0-J-k}\xi)=1\qquad\text{if}\qquad|\xi|\leq 2^{k+J-m_0},
\end{equation}
and by (\ref{eq347}),
\begin{equation}\label{eq423}
\FT_n\phi(2^{m_0-J-k}\xi)=1-\sum_{j=1}^{\infty}\FT_n\psi(2^{m_0-J-k-j}\xi)=1-\sum_{j=J+1}^{\infty}\FT_n\psi(2^{m_0-k-j}\xi).
\end{equation}
Furthermore if $\tau\in[1,2]$, $\theta\in C_l$, $2^{m_0-k}\xi\in spt.\FT_n\psi\bigcap spt.\rho_l\subseteq P_l$, $1\leq l\leq M$ then (\ref{eq417}) and (\ref{eq419}) tell us that
\begin{equation}\label{eq424}
0<2^{-m_0-2}\leq 2^{-k}\tau|\xi\cdot\theta|<2^{2-m_0},
\end{equation}
and
\begin{equation}\label{eq425}
|(e^{2\pi i2^{-k}\tau\theta\cdot\xi}-1)^L|>0.
\end{equation}
Hence if we let
\begin{equation}\label{eq426}
\lambda_{l,\tau\theta}(\xi):=\frac{\FT_n\psi(2^{m_0}\xi)\rho_l(2^{m_0}\xi)}{(e^{2\pi i\tau\theta\cdot\xi}-1)^L},
\end{equation}
then $\lambda_{l,\tau\theta}(2^{-k}\xi)$ is a well-defined function in $C_c^{\infty}(\bbbr^n)$ for every $k\in\bbbz$. Using formula (\ref{eq344}), we have
\begin{align}\label{eq427}
&|\iFT_n(\lambda_{l,\tau\theta}(2^{-k}\xi))*[\varDelta^L_{2^{-k}\tau\theta}(\phi_{2^{m_0-k-J}}*f)](x)| \nonumber\\
&\lesssim\int_{\bbbr^n}|\iFT_n(\lambda_{l,\tau\theta}(2^{-k}\xi))(y)\cdot
          \varDelta^L_{2^{-k}\tau\theta}(\phi_{2^{m_0-k-J}}*f)(x-y)|dy.
\end{align}
The Fourier transform of the Schwartz function
$$y\mapsto\iFT_n(\lambda_{l,\tau\theta}(2^{-k}\xi))(y)\cdot\varDelta^L_{2^{-k}\tau\theta}(\phi_{2^{m_0-k-J}}*f)(x-y)$$
is supported in $\{\xi\in\bbbr^n:|\xi|\lesssim 2^{k+J-m_0}\}$. Since $0<r<\min\{p,q\}<1$ as mentioned in Definition \ref{definition2}, we use (\ref{eq422}), observe the simple fact that both $\int_1^2\frac{d\tau}{\tau}$ and $\Haus^{n-1}(C_l)$ are fixed positive finite constants, and then apply Lemma \ref{lemma7} to (\ref{eq427}) and obtain
\begin{align}\label{eq428}
&|\iFT_n(\FT_n\psi(2^{m_0-k}\xi)\rho_l(2^{m_0-k}\xi)\FT_n f)(x)|^r \nonumber\\
&=\mvint_{[1,2]}\mvint_{C_l}|\iFT_n(\FT_n\psi(2^{m_0-k}\xi)\rho_l(2^{m_0-k}\xi) \nonumber\\
&\quad\cdot\FT_n\phi(2^{m_0-J-k}\xi)\FT_n f)(x)|^r
   d\Haus^{n-1}(\theta)\frac{d\tau}{\tau} \nonumber\\
&=\mvint_{[1,2]}\mvint_{C_l}|\iFT_n(\frac{\FT_n\psi(2^{m_0-k}\xi)\rho_l(2^{m_0-k}\xi)}{(e^{2\pi i2^{-k}\tau\theta\cdot\xi}-1)^L}
 \nonumber \\
&\quad\cdot(e^{2\pi i2^{-k}\tau\theta\cdot\xi}-1)^L\FT_n\phi(2^{m_0-J-k}\xi)\FT_n f)(x)|^r
   d\Haus^{n-1}(\theta)\frac{d\tau}{\tau} \nonumber\\
&=\mvint_{[1,2]}\mvint_{C_l}
   |\iFT_n(\lambda_{l,\tau\theta}(2^{-k}\xi))*[\varDelta^L_{2^{-k}\tau\theta}(\phi_{2^{m_0-k-J}}*f)](x)|^r
   d\Haus^{n-1}(\theta)\frac{d\tau}{\tau}\nonumber \\
&\lesssim 2^{(J+k-m_0)n(1-r)}\int_1^2\int_{C_l}\int_{\bbbr^n}
         |\iFT_n(\lambda_{l,\tau\theta}(2^{-k}\xi))(y)|^r \nonumber\\
&\quad\cdot|\varDelta^L_{2^{-k}\tau\theta}(\phi_{2^{m_0-k-J}}*f)(x-y)|^r
         dy d\Haus^{n-1}(\theta) \frac{d\tau}{\tau}.
\end{align}
We let $k=0$ in (\ref{eq424}) and pick $m_0$ so large that conditions of Lemma \ref{lemma2} are satisfied. Applying Lemma \ref{lemma2} to the smooth function $\FT_n\psi(2^{m_0}\xi)\rho_l(2^{m_0}\xi)$ whose support set is compactly contained in $P_l$ yields that for a sufficiently large positive integer $N$, whose value will be determined later, we can find a constant $C$ such that
\begin{equation}\label{eq429}
|\iFT_n\lambda_{l,\tau\theta}(x)|\leq\frac{C}{(1+|x|)^N}\qquad\text{for all }x\in\bbbr^n,
\end{equation}
and the constant $C$ may depend on $\psi,\rho_l,m_0,L,N$ but it is independent of $\tau\in[1,2]$ and $\theta\in C_l$. Recall that $A_{k-m}$ denotes the annulus $\{y\in\bbbr^n:2^m\leq 2^k|y|<2^{m+1}\}$ for integers $k,m$. With (\ref{eq429}), we can estimate the most inside integral at the end of (\ref{eq428}) as follows,
\begin{align}\label{eq430}
&\int_{\bbbr^n}|\iFT_n(\lambda_{l,\tau\theta}(2^{-k}\xi))(y)|^r\cdot
   |\varDelta^L_{2^{-k}\tau\theta}(\phi_{2^{m_0-k-J}}*f)(x-y)|^r dy \nonumber\\
&=\sum_{m\in\bbbz}\int_{A_{k-m}} 2^{knr}|\iFT_n\lambda_{l,\tau\theta}(2^k y)|^r\cdot
   |\varDelta^L_{2^{-k}\tau\theta}(\phi_{2^{m_0-k-J}}*f)(x-y)|^r dy \nonumber\\
&\lesssim\sum_{m<0}2^{knr}\int_{A_{k-m}}|\varDelta^L_{2^{-k}\tau\theta}(\phi_{2^{m_0-k-J}}*f)(x-y)|^r dy \nonumber\\
&\quad+\sum_{m\geq 0}2^{knr-mNr}\int_{A_{k-m}}|\varDelta^L_{2^{-k}\tau\theta}(\phi_{2^{m_0-k-J}}*f)(x-y)|^r dy \nonumber\\
&\lesssim\sum_{m<0}2^{kn(r-1)+mn}\mvint_{A_{k-m}}
         |\varDelta^L_{2^{-k}\tau\theta}(\phi_{2^{m_0-k-J}}*f)(x-y)|^r dy \nonumber\\
&\quad+\sum_{m\geq 0}2^{kn(r-1)+m(n-Nr)}\mvint_{A_{k-m}}|\varDelta^L_{2^{-k}\tau\theta}(\phi_{2^{m_0-k-J}}*f)(x-y)|^r dy.
\end{align}
We insert (\ref{eq430}) into (\ref{eq428}), apply Fubini's Theorem to switch the order of integration, use the following simple estimate
\begin{align}\label{eq431}
&\mvint_{A_{k-m}}\int_1^2\int_{C_l}|\varDelta^L_{2^{-k}\tau\theta}(\phi_{2^{m_0-k-J}}*f)(x-y)|^r
          d\Haus^{n-1}(\theta)\frac{d\tau}{\tau}dy \nonumber\\
&\lesssim\mvint_{|y|\leq 2^{m+1-k}}\int_1^2\int_{\unitsph}|\varDelta^L_{2^{-k}\tau\theta}(\phi_{2^{m_0-k-J}}*f)(x-y)|^r
          d\Haus^{n-1}(\theta)\frac{d\tau}{\tau}dy \nonumber\\
&\lesssim\HLmax_n(\int_1^2\int_{\unitsph}|\varDelta^L_{2^{-k}\tau\theta}(\phi_{2^{m_0-k-J}}*f)|^r
          d\Haus^{n-1}(\theta)\frac{d\tau}{\tau})(x),
\end{align}
and also pick $N$ so that $n-Nr<0$, then we obtain the estimate
\begin{align}\label{eq432}
&|\iFT_n(\FT_n\psi(2^{m_0-k}\xi)\rho_l(2^{m_0-k}\xi)\FT_n f)(x)|^r \nonumber\\
&\lesssim 2^{(J-m_0)n(1-r)}(\sum_{m<0}2^{mn}+\sum_{m\geq 0}2^{m(n-Nr)}) \nonumber\\
&\quad\cdot\HLmax_n(\int_1^2\int_{\unitsph}|\varDelta^L_{2^{-k}\tau\theta}(\phi_{2^{m_0-k-J}}*f)|^r
      d\Haus^{n-1}(\theta)\frac{d\tau}{\tau})(x) \nonumber\\
&\lesssim 2^{(J-m_0)n(1-r)}\HLmax_n(\int_1^2\int_{\unitsph}|\varDelta^L_{2^{-k}\tau\theta}(\phi_{2^{m_0-k-J}}*f)|^r
          d\Haus^{n-1}(\theta)\frac{d\tau}{\tau})(x),
\end{align}
where $\HLmax_n$ is the Hardy-Littlewood maximal function, and we can obtain all these inequalities above because the constant $C$ in (\ref{eq429}) does not rely on $\tau\in[1,2]$ and $\theta\in C_l\subseteq\unitsph$. Recall (\ref{eq421}) and the fact that $0<r<\min\{p,q\}<1$, then we have
\begin{align}\label{eq433}
&|\psi_{2^{m_0-k}}*f(x)| \nonumber\\
&=|\sum_{l=1}^M\iFT_n(\FT_n\psi(2^{m_0-k}\xi)\rho_l(2^{m_0-k}\xi)\FT_n f)(x)| \nonumber\\
&\lesssim(\sum_{l=1}^M|\iFT_n(\FT_n\psi(2^{m_0-k}\xi)\rho_l(2^{m_0-k}\xi)\FT_n f)(x)|^r)^{\frac{1}{r}} \nonumber\\
&\lesssim 2^{(J-m_0)n(\frac{1}{r}-1)}\HLmax_n(\int_1^2\int_{\unitsph}|\varDelta^L_{2^{-k}\tau\theta}(\phi_{2^{m_0-k-J}}*f)|^r
          d\Haus^{n-1}(\theta)\frac{d\tau}{\tau})(x)^{\frac{1}{r}}.
\end{align}
We insert (\ref{eq433}) into $\|f\|_{\Fspq}$ below, incorporate those coefficients that contain $m_0$ into constants since $m_0$ will be fixed, apply Lemma \ref{lemma6} and then we can obtain
\begin{align}\label{eq434}
&\|f\|_{\Fspq} \nonumber\\
&=2^{-sm_0}\|\{2^{ks}|\psi_{2^{m_0-k}}*f|\}_{k\in\bbbz}\|_{L^p(l^q)} \nonumber\\
&\lesssim 2^{Jn(\frac{1}{r}-1)}\|\{2^{ks}(\int_1^2\int_{\unitsph}|\varDelta^L_{2^{-k}\tau\theta}(\phi_{2^{m_0-k-J}}*f)(\cdot)|^r \nonumber\\
&\quad d\Haus^{n-1}(\theta)\frac{d\tau}{\tau})^{\frac{1}{r}}\}_{k\in\bbbz}\|_{L^p(l^q)}.
\end{align}
We use H\"{o}lder's inequality to obtain
\begin{align}\label{eq435}
&(\int_1^2\int_{\unitsph}|\varDelta^L_{2^{-k}\tau\theta}(\phi_{2^{m_0-k-J}}*f)(\cdot)|^r
          d\Haus^{n-1}(\theta)\frac{d\tau}{\tau})^{\frac{1}{r}} \nonumber\\
&\lesssim(\int_1^2\int_{\unitsph}|\varDelta^L_{2^{-k}\tau\theta}(\phi_{2^{m_0-k-J}}*f)(\cdot)|^q
          d\Haus^{n-1}(\theta)\frac{d\tau}{\tau})^{\frac{1}{q}}.
\end{align}
Inserting (\ref{eq435}) into (\ref{eq434}) yields $\|f\|_{\Fspq}$ can be estimated from above by
\begin{equation}\label{eq436}
2^{Jn(\frac{1}{r}-1)}
\|\{2^{ks}(\!\int_1^2\!\!\int_{\unitsph}\!|\varDelta^L_{2^{-k}\tau\theta}(\phi_{2^{m_0-k-J}}*f)(\cdot)|^q
d\Haus^{n-1}(\theta)\frac{d\tau}{\tau})^{\frac{1}{q}}\}_{k\in\bbbz}\|_{L^p(l^q)}.
\end{equation}
Recall (\ref{eq423}) and the notation $f_j=\psi_{2^{-j}}*f$ then we have
\begin{equation}\label{eq500}
\phi_{2^{m_0-k-J}}*f=f-\sum_{j=J+1}^{\infty}f_{k+j-m_0}\qquad\text{in the sense of }\Sw'(\bbbr^n).
\end{equation}
We can use an argument like the one for deducing (\ref{eq514}) to obtain
\begin{equation}\label{eq548}
\varDelta^L_{2^{-k}\tau\theta}(\phi_{2^{m_0-k-J}}*f)=\varDelta^L_{2^{-k}\tau\theta}f
-\sum_{j=J+1}^{\infty}\varDelta^L_{2^{-k}\tau\theta}f_{k+j-m_0}\qquad\text{in the sense of }\Sw'(\bbbr^n).
\end{equation}
Inferring from (\ref{eq500}) and assuming the validity of decomposition for now, then (\ref{eq436}) can be estimated from above by the sum of the following two terms,
\begin{equation}\label{eq437}
2^{Jn(\frac{1}{r}-1)}
\|\{2^{ks}(\int_1^2\int_{\unitsph}|\varDelta^L_{2^{-k}\tau\theta}f(\cdot)|^q
d\Haus^{n-1}(\theta)\frac{d\tau}{\tau})^{\frac{1}{q}}\}_{k\in\bbbz}\|_{L^p(l^q)},
\end{equation}
and
\begin{equation}\label{eq438}
2^{Jn(\frac{1}{r}-1)}
\|\{2^{ks}(\int_1^2\int_{\unitsph}|\varDelta^L_{2^{-k}\tau\theta}(\sum_{j=J+1}^{\infty}f_{k+j-m_0})(\cdot)|^q
d\Haus^{n-1}(\theta)\frac{d\tau}{\tau})^{\frac{1}{q}}\}_{k\in\bbbz}\|_{L^p(l^q)}.
\end{equation}
For the first term, we use the change of variable formulas $t=2^{-k}\tau$ for $\tau\in[1,2]$ and $h=t\theta$ for $\theta\in\unitsph$ in a sequence and we can get
\begin{equation}\label{eq439}
\int_1^2\int_{\unitsph}|\varDelta^L_{2^{-k}\tau\theta}f(x)|^q
d\Haus^{n-1}(\theta)\frac{d\tau}{\tau}=
\int_{A_k}\frac{|\varDelta^L_{h}f(x)|^q}{|h|^n}dh,
\end{equation}
where $A_k$ is the annulus $\{h\in\bbbr^n:2^{-k}\leq|h|<2^{1-k}\}$, and hence
\begin{equation}\label{eq440}
(\ref{eq437})\lesssim 2^{Jn(\frac{1}{r}-1)}\cdot
\|(\int_{\bbbr^n}\frac{|\varDelta^L_{h}f(\cdot)|^q}{|h|^{n+sq}}dh)^{\frac{1}{q}}\|_{L^p(\bbbr^n)}<\infty,
\end{equation}
and the value of the large positive integer $J$ will be determined later. For the second term (\ref{eq438}), we begin with the same change of variable as in (\ref{eq439}) and obtain
\begin{align}\label{eq549}
&\int_1^2\int_{\unitsph}(\sum_{j=J+1}^{\infty}|\varDelta^L_{2^{-k}\tau\theta}f_{k+j-m_0}(x)|)^q
         d\Haus^{n-1}(\theta)\frac{d\tau}{\tau} \nonumber\\
&\lesssim 2^{kn}\int_{A_k}(\sum_{j=J+1}^{\infty}|\varDelta^L_{h}f_{k+j-m_0}(x)|)^q dh \nonumber\\
&\lesssim\sum_{j=J+1}^{\infty}2^{kn}\int_{A_k}|\varDelta^L_{h}f_{k+j-m_0}(x)|^q dh \nonumber\\
&\lesssim\sum_{j=J+1}^{\infty}2^{(j-m_0)n(\frac{q}{r}-1)}\HLmax_n(|f_{k+j-m_0}|^r)(x)^{\frac{q}{r}},
\end{align}
where in the penultimate expression of (\ref{eq549}) we used the condition $0<q<1$, and in the last expression of (\ref{eq549}) we used estimate (\ref{eq544}) since $k+j-m_0>k+J-m_0>k$. Therefore we have the estimate
\begin{align}\label{eq551}
&2^{Jn(\frac{1}{r}-1)}
\|\{2^{ks}(\int_1^2\!\!\!\!\int_{\unitsph}\!\!(\!\!\sum_{j=J+1}^{\infty}\!\!|\varDelta^L_{2^{-k}\tau\theta}f_{k+j-m_0}(\cdot)|)^q
d\Haus^{n-1}(\theta)\frac{d\tau}{\tau})^{\frac{1}{q}}\}_{k\in\bbbz}\|_{L^p(l^q)} \nonumber\\
&\lesssim 2^{Jn(\frac{1}{r}-1)}
\|(\sum_{k\in\bbbz}\sum_{j=J+1}^{\infty}2^{ksq}\cdot 2^{jqn(\frac{1}{r}-\frac{1}{q})}
\HLmax_n(|f_{k+j-m_0}|^r)(\cdot)^{\frac{q}{r}})^{\frac{1}{q}}\|_{L^p(\bbbr^n)} \nonumber\\
&\lesssim 2^{Jn(\frac{1}{r}-1)}
\|(\sum_{j=J+1}^{\infty}2^{jq[n(\frac{1}{r}-\frac{1}{q})-s]}\cdot
\sum_{k\in\bbbz}2^{ksq}\HLmax_n(|f_{k}|^r)(\cdot)^{\frac{q}{r}})^{\frac{1}{q}}\|_{L^p(\bbbr^n)} \nonumber\\
&\lesssim 2^{J[n(\frac{1}{r}-1)+n(\frac{1}{r}-\frac{1}{q})-s]}\|f\|_{\Fspq},
\end{align}
where in the above calculation we incorporate coefficients containing $m_0$ into constants since $m_0$ is fixed. Recall the conditions $0<q<1$, $\sigma_{pq}+\tilde{\sigma}_{pq}<s$ and $0<r<\min\{p,q\}$. If $\min\{p,q\}=q$, then the condition $\sigma_{pq}+\tilde{\sigma}_{pq}<s$ means $s>n(\frac{1}{q}-1)$ and we can pick $r$ sufficiently close to $q$ so that
\begin{equation}\label{eq446}
s>n(\frac{1}{r}-1)+n(\frac{1}{r}-\frac{1}{q})>n(\frac{1}{r}-\frac{1}{q}).
\end{equation}
If $\min\{p,q\}=p$, then the condition $\sigma_{pq}+\tilde{\sigma}_{pq}<s$ means $s>n(\frac{1}{p}-1)+n(\frac{1}{p}-\frac{1}{q})$ and we can pick $r$ sufficiently close to $p$ so that (\ref{eq446}) still holds true. Hence by invoking Lemma \ref{lemma6}, the last inequality (\ref{eq551}) is justified. We also infer from the assumption $f\in\Fspq$ and inequality (\ref{eq551}) that $\sum_{j=J+1}^{\infty}|\varDelta^L_{2^{-k}\tau\theta}f_{k+j-m_0}(x)|<\infty$ for every $k\in\bbbz$, almost every $\tau\in[1,2],\theta\in\unitsph,x\in\bbbr^n$. Therefore (\ref{eq548}), (\ref{eq440}), the above inference and the supposition of $f$ being a function validate the decomposition
\begin{equation}\label{eq552}
\varDelta^L_{2^{-k}\tau\theta}(\phi_{2^{m_0-k-J}}*f)(x)=\varDelta^L_{2^{-k}\tau\theta}f(x)
-\sum_{j=J+1}^{\infty}\varDelta^L_{2^{-k}\tau\theta}f_{k+j-m_0}(x)
\end{equation}
in the sense of $\Sw'(\bbbr^n)$ for every $k\in\bbbz$, almost every $\tau\in[1,2],\theta\in\unitsph,x\in\bbbr^n$ when $0<p<\infty$, $0<q<1$ and $\sigma_{pq}+\tilde{\sigma}_{pq}<s<\infty$, furthermore estimating (\ref{eq436}) from above by the sum of (\ref{eq437}) and (\ref{eq438}) is justified, moreover (\ref{eq438}) can be estimated from above by the beginning expression of (\ref{eq551}) and hence by the ending expression of (\ref{eq551}). We have reached the conclusion
\begin{align}\label{eq447}
\|f\|_{\Fspq}
&\leq C'2^{Jn(\frac{1}{r}-1)}\cdot\|(\int_{\bbbr^n}\frac{|\varDelta^L_{h}f(\cdot)|^q}{|h|^{n+sq}}dh)^{\frac{1}{q}}\|_{L^p(\bbbr^n)} \nonumber\\
&\quad+C'2^{J[n(\frac{1}{r}-1)+n(\frac{1}{r}-\frac{1}{q})-s]}\cdot\|f\|_{\Fspq},
\end{align}
where the constant $C'$ is independent of $J$. From (\ref{eq446}) we see that if we pick $J$ sufficiently large so that the coefficient $C'2^{J[n(\frac{1}{r}-1)+n(\frac{1}{r}-\frac{1}{q})-s]}$ is less than $\frac{1}{2}$ and then shift the term $C'2^{J[n(\frac{1}{r}-1)+n(\frac{1}{r}-\frac{1}{q})-s]}\cdot\|f\|_{\Fspq}$ to the left side of (\ref{eq447}), then we can finish the proof of the first part of Theorem \ref{theorem2} (ii).

Next we prove the second part of Theorem \ref{theorem2} (ii) when $f$ is a function, $0<p<\infty$, $1\leq q<\infty$ and $-n<s<\infty$. It seems that the same method as in the proof of the first part of Theorem \ref{theorem2} (ii) produces a worse result in the case $0<p<1\leq q<\infty$, therefore we use a different method to prove the second part. Still, we assume the right side of (\ref{eq415}) is finite. We use equalities (\ref{eq260}) and (\ref{eq261}) and integrate
$[(-1)^{L+1}\varDelta^L_{2^{-k}z}f(x)]$ against a Schwartz function $g(z)$ of chosen properties. We let $g$ be a radial Schwartz function whose radial Fourier transform $\FT_n g$ satisfies
\begin{equation}\label{eq297}
0\leq\FT_n g\leq 1,\quad\FT_n g\text{ is supported in }\{\xi\in\bbbr^n:\frac{1}{4}\leq|\xi|<4L\}
\end{equation}
and
\begin{equation}\label{eq298}
\FT_n g(\xi)=1\text{ on }\{\xi\in\bbbr^n:\frac{1}{2}\leq|\xi|<2L\}.
\end{equation}
Since $\FT_n g(0)=0$, then $\int_{\bbbr^n}g(z)dz=0$ and we obtain the equality
\begin{equation}\label{eq299}
\!\int_{\bbbr^n}\!\!\!\!g(z)[(-1)^{L+1}(\varDelta^L_{2^{-k}z}f)(x)]dz
=\!\int_{\bbbr^n}\!\!\!\!g(z)[\sum^L_{j=1}d_j f(x+2^{-k}jz)]dz
=\!\sum^L_{j=1}d_j g_{2^{-k}j}*f(x),
\end{equation}
where the kernel $G_k(z):=\sum^L_{j=1}d_j g_{2^{-k}j}(z)$ satisfies
\begin{equation}\label{eq300}
spt.\FT_n G_k\subseteq\{\xi\in\bbbr^n:2^{k-2}/L\leq|\xi|\leq 2^{k+2}L\},\text{ }
\FT_n G_k(\xi)=1\text{ if }2^{k-1}\leq|\xi|<2^{k+1}.
\end{equation}
We first estimate the term $\|\{2^{ks}G_k*f\}_{k\in\bbbz}\|_{L^p(l^q)}$. Since $g(z)$ is a bounded Schwartz function, then $|g(z)|\lesssim|z|^{-N'}$ for $0\neq z\in\bbbr^n$ and $N'$ can be a sufficiently large positive integer whose value will be determined later. Recall that $A_l=\{z\in\bbbr^n:2^{-l}\leq|z|<2^{1-l}\}$ and $A_0=\{h\in\bbbr^n:1\leq|h|<2\}$ and by (\ref{eq299}), we have for every $x\in\bbbr^n$
\begin{align}\label{eq553}
2^{ks}|(G_k*f)(x)|
&\lesssim 2^{ks}\int_{\bbbr^n}|g(z)|\cdot|(\varDelta^L_{2^{-k}z}f)(x)|dz \nonumber\\
&=2^{ks}\sum_{l\in\bbbz}\int_{A_l}|g(z)|\cdot|(\varDelta^L_{2^{-k}z}f)(x)|dz \nonumber\\
&=\sum_{l\in\bbbz}2^{ks-ln}\int_{A_0}|g(2^{-l}h)|\cdot|(\varDelta^L_{2^{-k-l}h}f)(x)|dh \nonumber\\
&\lesssim\sum_{l\geq 0}2^{ks-ln}\int_{A_0}|(\varDelta^L_{2^{-k-l}h}f)(x)|dh \nonumber\\
&\quad+\sum_{l<0}2^{ks+l(N'-n)}\int_{A_0}|(\varDelta^L_{2^{-k-l}h}f)(x)|dh.
\end{align}
Applying Minkowski's inequality for $\|\cdot\|_{l^q}$-norm for $1\leq q<\infty$ to the above inequality yields
\begin{align}\label{eq306}
&\|\{2^{ks}(G_k*f)(x)\}_{k\in\bbbz}\|_{l^q} \nonumber\\
&\lesssim\sum_{l\geq 0}2^{-l(n+s)}(\sum_{k\in\bbbz}2^{(k+l)sq}(\int_{A_0}|(\varDelta^L_{2^{-k-l}h}f)(x)|dh)^q)^{\frac{1}{q}} \nonumber\\
&\quad+\sum_{l<0}2^{l(N'-n-s)}(\sum_{k\in\bbbz}2^{(k+l)sq}(\int_{A_0}|(\varDelta^L_{2^{-k-l}h}f)(x)|dh)^q)^{\frac{1}{q}} \nonumber\\
&\lesssim(\sum_{k\in\bbbz}2^{ksq}(\int_{A_0}|(\varDelta^L_{2^{-k}h}f)(x)|dh)^q)^{\frac{1}{q}},
\end{align}
if $N'$ is chosen so that $N'>n+s>0$. Using H\"{o}lder's inequality for $1\leq q<\infty$, we have
\begin{equation}\label{eq307}
(\int_{A_0}|(\varDelta^L_{2^{-k}h}f)(x)|dh)^q\lesssim\int_{A_0}|(\varDelta^L_{2^{-k}h}f)(x)|^q dh.
\end{equation}
Inserting (\ref{eq307}) into (\ref{eq306}) and applying the appropriate change of variable $z=2^{-k}h$ and then inserting the resulting inequality into $\|\cdot\|_{L^p(\bbbr^n)}$ quasinorm yield
\begin{align}\label{eq308}
&\|\{2^{ks}G_k*f\}_{k\in\bbbz}\|_{L^p(l^q)} \nonumber\\
&\lesssim\|(\sum_{k\in\bbbz}2^{k(sq+n)}\int_{A_k}|(\varDelta^L_{z}f)(\cdot)|^q dz)^{\frac{1}{q}}\|_{L^p(\bbbr^n)} \nonumber\\
&\lesssim\|(\int_{\bbbr^n}|z|^{-sq}\cdot|(\varDelta^L_{z}f)(\cdot)|^q\frac{dz}{|z|^n})^{\frac{1}{q}}\|_{L^p(\bbbr^n)}.
\end{align}
From the first line of (\ref{eq553}) and the inequality (\ref{eq308}), we also deduce for every $k\in\bbbz$ and almost every $x\in\bbbr^n$, the integral on the left end of (\ref{eq299}) is absolutely convergent and hence well-defined. Thus for each $k\in\bbbz$ and almost every $x\in\bbbr^n$
\begin{equation}\label{eq301}
f_k(x)=\iFT_n(\FT_n\psi(2^{-k}\xi)\FT_n f)(x)=\psi_{2^{-k}}*G_k*f(x)
\end{equation}
and we argue as in Remark \ref{remark3} to obtain
\begin{align}\label{eq302}
&|\psi_{2^{-k}}*G_k*f(x)| \nonumber\\
&\lesssim\esssup_{z\in\bbbr^n}\frac{|\psi_{2^{-k}}*G_k*f(x-z)|}{(1+2^{k+2}L|z|)^{n/r}} \nonumber\\
&\lesssim\esssup_{z\in\bbbr^n}\int_{\bbbr^n}|\psi_{2^{-k}}(y)|(1+2^{k+2}L|y|)^{n/r}\cdot
          \frac{|G_k*f(x-z-y)|}{(1+2^{k+2}L|z+y|)^{n/r}}dy \nonumber\\
&\lesssim\PFSmax_n (G_k*f)(x)\cdot\int_{\bbbr^n}|\psi_{2^{-k}}(y)|(1+2^{k+2}L|y|)^{n/r}dy \nonumber\\
&\lesssim\PFSmax_n (G_k*f)(x),
\end{align}
and we recall that $\psi_{2^{-k}}(y)=2^{kn}\psi(2^ky)$ thus the constant in (\ref{eq302}) is independent of $k\in\bbbz$. From (\ref{eq301}), (\ref{eq302}) and (\ref{eq82}) of Remark \ref{remark6}, we deduce that
\begin{align}\label{eq303}
\|f\|_{\Fspq}
&=(\int_{\bbbr^n}(\sum_{k\in\bbbz}2^{ksq}|f_k(x)|^q)^{p/q}dx)^{1/p} \nonumber\\
&\lesssim(\int_{\bbbr^n}(\sum_{k\in\bbbz}2^{ksq}|\PFSmax_n (G_k*f)(x)|^q)^{p/q}dx)^{1/p} \nonumber\\
&\sim(\int_{\bbbr^n}(\sum_{k\in\bbbz}2^{ksq}|(G_k*f)(x)|^q)^{p/q}dx)^{1/p}.
\end{align}
Combining (\ref{eq303}) and (\ref{eq308}), we conclude the proof of the second part of Theorem \ref{theorem2} (ii).

For the case $q=\infty$, we first prove Theorem \ref{theorem2} (iii). We begin with estimating the term $\esssup_{k\in\bbbz}\esssup_{h\in A_k}2^{ks}\sum_{j\in\bbbz}|(\varDelta^L_{h}f_j)(x)|$ from above by the following
\begin{equation}\label{eq310}
\esssup_{k\in\bbbz}\sum_{j\leq k}2^{ks}\esssup_{h\in A_k}|(\varDelta^L_{h}f_j)(x)|
+\esssup_{k\in\bbbz}\sum_{j>k}2^{ks}\esssup_{h\in A_k}|(\varDelta^L_{h}f_j)(x)|.
\end{equation}
We pick $0<\varepsilon<\min\{s,L-s\}$ and estimate the first term of (\ref{eq310}) as follows
\begin{align}\label{eq311}
&\esssup_{k\in\bbbz}\sum_{j\leq k}2^{ks}\esssup_{h\in A_k}|(\varDelta^L_{h}f_j)(x)| \nonumber\\
&=\esssup_{k\in\bbbz}\sum_{j\leq k}2^{j\varepsilon}2^{-j\varepsilon+ks}\esssup_{h\in A_k}|(\varDelta^L_{h}f_j)(x)| \nonumber\\
&\lesssim\esssup_{k\in\bbbz}\sum_{j\leq k}2^{j\varepsilon}\esssup_{l\leq k} 2^{-l\varepsilon+ks}
          \esssup_{h\in A_k}|(\varDelta^L_{h}f_l)(x)| \nonumber\\
&\lesssim\esssup_{k\in\bbbz}\esssup_{l\leq k} 2^{-l\varepsilon+k(s+\varepsilon)}
          \esssup_{h\in A_k}|(\varDelta^L_{h}f_l)(x)|.
\end{align}
We use estimate (\ref{eq270}) and Remark \ref{remark12} to get
\begin{align}\label{eq312}
(\ref{eq311})
&\lesssim\esssup_{l\in\bbbz}\esssup_{k\geq l}2^{k(s+\varepsilon-L)}2^{l(L-\varepsilon)}\PFSmax_n f_l(x) \nonumber\\
&\lesssim\esssup_{l\in\bbbz}2^{ls}\PFSmax_n f_l(x) \nonumber\\
&\lesssim\esssup_{l\in\bbbz}2^{ls}\HLmax_n(|f_l|^r)(x)^{\frac{1}{r}} \nonumber\\
&\lesssim\HLmax_n(\esssup_{l\in\bbbz}2^{lsr}|f_l|^r)(x)^{\frac{1}{r}}.
\end{align}
We estimate the second term of (\ref{eq310}) from above by
\begin{align}\label{eq313}
&\esssup_{k\in\bbbz}\sum_{j>k}2^{ks}\esssup_{h\in A_k}|(\varDelta^L_{h}f_j)(x)| \nonumber\\
&=\esssup_{k\in\bbbz}\sum_{j>k}2^{-j\varepsilon}\cdot 2^{j\varepsilon+ks}\esssup_{h\in A_k}|(\varDelta^L_{h}f_j)(x)| \nonumber\\
&\lesssim\esssup_{k\in\bbbz}\sum_{j>k}2^{-j\varepsilon}\esssup_{l>k} 2^{l\varepsilon+ks}\esssup_{h\in
         A_k}|(\varDelta^L_{h}f_l)(x)| \nonumber\\
&\lesssim\esssup_{l\in\bbbz}\esssup_{k<l}2^{k(s-\varepsilon)}2^{l\varepsilon}\esssup_{h\in A_k}|(\varDelta^L_{h}f_l)(x)|.
\end{align}
From Lebesgue's differentiation theorem, we have $|f_l(x)|\lesssim\HLmax_n(|f_l|^r)(x)^{\frac{1}{r}}$ for almost every $x\in\bbbr^n$. Putting (\ref{eq272}) back into (\ref{eq271}) yields the estimate
\begin{equation}\label{eq273}
|(\varDelta^L_h f_l)(x)|\lesssim 2^{(l-k)n/r}\HLmax_n(|f_l|^r)(x)^{1/r}\quad\text{for }|h|\lesssim 2^{-k},l>k,
\end{equation}
where the constant in (\ref{eq273}) is independent of $h\in\bbbr^n$, $l,k\in\bbbz$. The value of $r\in(0,p)$ will be determined later. Inserting (\ref{eq273}) into (\ref{eq313}) yields
\begin{equation}\label{eq315}
(\ref{eq313})\lesssim\esssup_{l\in\bbbz}\esssup_{k<l}2^{k(s-\varepsilon-\frac{n}{r})}2^{l(\varepsilon+\frac{n}{r})}
\HLmax_n(|f_l|^r)(x)^{\frac{1}{r}}.
\end{equation}
Since $\frac{n}{s}<p$ and $s-\varepsilon-\frac{n}{r}\rightarrow s-\frac{n}{p}>0$ as $\varepsilon\rightarrow 0$ and $r\rightarrow p$, we can pick $\varepsilon$ sufficiently small and $r$ sufficiently close to $p$ so that $s-\varepsilon-\frac{n}{r}$ is a positive number and hence
\begin{equation}\label{eq316}
(\ref{eq315})\lesssim\esssup_{l\in\bbbz}2^{ls}\HLmax_n(|f_l|^r)(x)^{\frac{1}{r}}
\lesssim\HLmax_n(\esssup_{l\in\bbbz}2^{lsr}|f_l|^r)(x)^{\frac{1}{r}}.
\end{equation}
From the above discussion and the $L^{\frac{p}{r}}(\bbbr^n)$-boundedness of the Hardy-Littlewood maximal function, we have proven that
\begin{align}\label{eq317}
&\|\esssup_{k\in\bbbz}\esssup_{h\in A_k}2^{ks}\sum_{j\in\bbbz}|(\varDelta^L_{h}f_j)(\cdot)|\|_{L^p(\bbbr^n)} \nonumber\\
&\lesssim\|\HLmax_n(\esssup_{l\in\bbbz}2^{lsr}|f_l|^r)(\cdot)^{\frac{1}{r}}\|_{L^p(\bbbr^n)}\lesssim\|f\|_{\Fspinf}
\end{align}
for all $0<p<\infty$ and $\frac{n}{p}<s<L$. The above inequality and the assumption $f\in\Fspinf$ have shown $\sum_{j\in\bbbz}|(\varDelta^L_{h}f_j)(x)|<\infty$ for every $k\in\bbbz$, almost every $h\in A_k,x\in\bbbr^n$. In conjunction with (\ref{eq513}), we have justified the claim that
\begin{equation}\label{eq554}
\varDelta^L_h f=\sum_{j\in\bbbz}\varDelta^L_h f_j(x)\text{ in the sense of }\Sw'(\bbbr^n)/\mathscr{P}(\bbbr^n)
\end{equation}
for every $k\in\bbbz$, almost every $h\in A_k,x\in\bbbr^n$, and hence also the inequality
\begin{equation}\label{eq555}
\|\esssup_{h\in\bbbr^n}\frac{|\varDelta^L_{h}f|}{|h|^s}\|_{L^p(\bbbr^n)}\lesssim
\|\esssup_{k\in\bbbz}\esssup_{h\in A_k}2^{ks}\sum_{j\in\bbbz}|(\varDelta^L_{h}f_j)(\cdot)|\|_{L^p(\bbbr^n)}.
\end{equation}
Now (\ref{eq555}) and (\ref{eq317}) conclude the proof of Theorem \ref{theorem2} (iii).

To prove Theorem \ref{theorem2} (iv), we assume the right side of (\ref{eq318}) is finite and use the inequality (\ref{eq553}) with $N'>n+s>0$, and then we can deduce that
\begin{align}\label{eq320}
&\|\esssup_{k\in\bbbz}2^{ks}|G_k*f(\cdot)|\|_{L^p(\bbbr^n)} \nonumber\\
&\lesssim(\sum_{l\geq 0}2^{-l(n+s)}+\sum_{l<0}2^{l(N'-n-s)}) \nonumber\\
&\quad\cdot\|\esssup_{k\in\bbbz}2^{ks}\int_{A_0}|(\varDelta^L_{2^{-k}h}f)(\cdot)|dh\|_{L^p(\bbbr^n)} \nonumber\\
&\lesssim\|\esssup_{k\in\bbbz}\esssup_{h\in A_k} 2^{ks}|(\varDelta^L_{h}f)(\cdot)|\|_{L^p(\bbbr^n)} \nonumber\\
&\lesssim\|\esssup_{h\in\bbbr^n}\frac{|(\varDelta^L_{h}f)(\cdot)|}{|h|^s}\|_{L^p(\bbbr^n)}.
\end{align}
And the above estimate, in conjunction with the first line of (\ref{eq553}), shows the absolute convergence of the integral on the left end of (\ref{eq299}) for every $k\in\bbbz$ and almost every $x\in\bbbr^n$. Using equality (\ref{eq301}), estimate (\ref{eq302}), Remark \ref{remark12} and the $L^{\frac{p}{r}}(\bbbr^n)$-boundedness of Hardy-Littlewood maximal function in a sequence, we can also obtain
\begin{align}\label{eq319}
\|f\|_{\Fspinf}
&=\|\esssup_{k\in\bbbz}|2^{ks}f_k(\cdot)|\|_{L^p(\bbbr^n)} \nonumber\\
&\lesssim\|\esssup_{k\in\bbbz}2^{ks}\HLmax_n(|G_k*f|^r)(\cdot)^{\frac{1}{r}}\|_{L^p(\bbbr^n)} \nonumber\\
&\lesssim\|\HLmax_n(\esssup_{k\in\bbbz}2^{ksr}|G_k*f|^r)(\cdot)^{\frac{1}{r}}\|_{L^p(\bbbr^n)} \nonumber\\
&\lesssim\|\esssup_{k\in\bbbz}2^{ks}|G_k*f(\cdot)|\|_{L^p(\bbbr^n)}.
\end{align}
Inequalities (\ref{eq320}) and (\ref{eq319}) finish the proof of Theorem \ref{theorem2} (iv). The proof of Theorem \ref{theorem2} is now complete.
\end{proof}

\section{Triebel-Lizorkin space and iterated differences along coordinate axes}\label{proof.of.theorem6}
\begin{proof}[Proof Of Theorem \ref{theorem6}]
We first prove inequality (\ref{eq450}) when $0<p,q<\infty$, $\tilde{\sigma}^1_{pq}<s<L$ and $f$ is a tempered distribution in $\Sw'(\bbbr^n)$. Without loss of generality, we only need to prove the inequality for $j=1$, the cases for $j=2,\cdots,n$ can be proved in the same way. Recall that for $x\in\bbbr^n$, $x=(x_1,x_1')$ and $x_1'=(x_2,\cdots,x_n)\in\bbbr^{n-1}$. We still denote $f_l=\psi_{2^{-l}}*f$ and begin with estimating $\sum_{k\in\bbbz}2^{k(sq+1)}\int_{2^{-k}}^{2^{1-k}}(\sum_{l\in\bbbz}|\varDelta^L_{t,1}f_l(x)|)^q dt$ from above by the following
\begin{equation}\label{eq472}
\sum_{k\in\bbbz}2^{k(sq+1)}\int_{2^{-k}}^{2^{1-k}}(\sum_{l\leq k}|\varDelta^L_{t,1}f_l(x)|)^q dt+
\sum_{k\in\bbbz}2^{k(sq+1)}\int_{2^{-k}}^{2^{1-k}}(\sum_{l>k}|\varDelta^L_{t,1}f_l(x)|)^q dt.
\end{equation}
Using the same calculation technique exhibited in (\ref{eq276}) and (\ref{eq277}), if $0<\varepsilon<\min\{s,L-s\}$, we can obtain
\begin{align}
(\sum_{l\leq k}|\varDelta^L_{t,1}f_l(x)|)^q&\lesssim 2^{kq\varepsilon}\cdot
\sum_{l\leq k}2^{-lq\varepsilon}|\varDelta^L_{t,1}f_l(x)|^q,\label{eq473}\\
(\sum_{l>k}|\varDelta^L_{t,1}f_l(x)|)^q&\lesssim 2^{-kq\varepsilon}\cdot
\sum_{l>k}2^{lq\varepsilon}|\varDelta^L_{t,1}f_l(x)|^q.\label{eq474}
\end{align}
Inserting these estimates into (\ref{eq472}), we can estimate (\ref{eq472}) from above by
\begin{align}\label{eq475}
&\sum_{k\in\bbbz}2^{k(sq+\varepsilon q+1)}\sum_{l\leq k}
2^{-lq\varepsilon}\int_{2^{-k}}^{2^{1-k}} |\varDelta^L_{t,1}f_l(x)|^q dt \nonumber\\
&+\sum_{k\in\bbbz}2^{k(sq-\varepsilon q+1)}\sum_{l>k}2^{lq\varepsilon}
\int_{2^{-k}}^{2^{1-k}} |\varDelta^L_{t,1}f_l(x)|^q dt.
\end{align}
Now we give an important estimate for $\varDelta^L_{t,1}f_l(x)$. By using Mean Value Theorem consecutively with respect to the first coordinate, we obtain
$$\varDelta^L_{t,1}f_l(x)=\partial^{\alpha}f_l(x_1+\lambda t,x_1')\cdot t^L$$
for some $\lambda$ between $0$ and $L$ and $\alpha=(L,0,\cdots,0)$ is a multi-index. From Lemma \ref{lemma8}, we know for fixed $x_1'\in\bbbr^{n-1}$, the $1$-dimensional Peetre-Fefferman-Stein maximal function of $f_l(\cdot,x_1')$ is well-defined. Using the $1$-dimensional version of Remark \ref{remark4}, we have
$$|\partial^{\alpha} f_l(x_1+\lambda t,x_1')|\lesssim\PFSmax_1\partial^{\alpha} f_l(\cdot,x_1')(x_1+\lambda t)
\lesssim 2^{lL}\PFSmax_1 f_l(\cdot,x_1')(x_1+\lambda t).$$
Using $1$-dimensional version of Remark \ref{remark2} and assuming $t\lesssim 2^{-k}$, we can further obtain
$$\PFSmax_1 f_l(\cdot,x_1')(x_1+\lambda t)\lesssim(1+2^{l-k})^{\frac{1}{r}}\PFSmax_1 f_l(\cdot,x_1')(x_1),$$
where $0<r<\min\{p,q\}$. Therefore the estimate is given as follows
\begin{equation}\label{eq478}
|\varDelta^L_{t,1}f_l(x)|\lesssim 2^{(l-k)L}(1+2^{l-k})^{\frac{1}{r}}\PFSmax_1 f_l(\cdot,x_1')(x_1)
\quad\text{for}\quad|t|\lesssim 2^{-k}.
\end{equation}
For the first term in (\ref{eq475}), we use the above estimate (\ref{eq478}) and $1$-dimensional version of Lemma \ref{lemma4} to obtain
\begin{align}\label{eq479}
&\sum_{k\in\bbbz}2^{k(sq+\varepsilon q+1)}\sum_{l\leq k}
          2^{-lq\varepsilon}\int_{2^{-k}}^{2^{1-k}}\!\!\!\!\!\! |\varDelta^L_{t,1}f_l(x)|^q dt \nonumber\\
&\lesssim\sum_{l\in\bbbz}(\sum_{k\geq l}2^{kq(s+\varepsilon-L)})\cdot 2^{lq(L-\varepsilon)}\PFSmax_1 f_l(\cdot,x_1')(x_1)^q \nonumber\\
&\lesssim\sum_{l\in\bbbz}2^{lqs}\PFSmax_1 f_l(\cdot,x_1')(x_1)^q \nonumber\\
&\lesssim\sum_{l\in\bbbz}2^{lqs}\HLmax_1(|f_l(\cdot,x_1')|^r)(x_1)^{\frac{q}{r}},
\end{align}
since $(1+2^{l-k})^{\frac{q}{r}}$ is bounded from above by a constant when $l\leq k$, and $$\HLmax_1(|f_l(\cdot,x_1')|^r)(x_1)$$ is the $1$-dimensional Hardy-Littlewood maximal function of $|f_l(\cdot,x_1')|^r$ centered at $x_1$. For the second term in (\ref{eq475}), we use (\ref{eq260}) to get
\begin{equation}\label{eq480}
|\varDelta^L_{t,1}f_l(x)|\lesssim\sum_{m=0}^L|f_l(x_1+mt,x_1')|.
\end{equation}
When $0\leq m\leq L$, $|t|\lesssim 2^{-k}$ and $l>k$, we also have
\begin{equation}\label{eq481}
|f_l(x_1+mt,x_1')|
\lesssim(1+2^{l-k})^{\frac{1}{r}}\PFSmax_1 f_l(\cdot,x_1')(x_1)
\lesssim 2^{\frac{l-k}{r}}\HLmax_1(|f_l(\cdot,x_1')|^r)(x_1)^{\frac{1}{r}},
\end{equation}
by using $1$-dimensional versions of Remark \ref{remark12}, Remark \ref{remark2} and Lemma \ref{lemma4}, and constants are independent of $l,k,m,t$. And the following inequality is true for $0\leq m\leq L$
\begin{align}\label{eq556}
&\mvint_{[2^{-k},2^{1-k}]}|f_l(x_1+mt,x_1')|^r dt \nonumber\\
&\lesssim\mvint_{|t|\leq 2^{1-k}}|f_l(x_1+mt,x_1')|^r dt
\lesssim\HLmax_1(|f_l(\cdot,x_1')|^r)(x_1).
\end{align}
Applying estimates (\ref{eq480}), (\ref{eq481}) and (\ref{eq556}) yields the following
\begin{align}\label{eq557}
&2^k\int_{2^{-k}}^{2^{1-k}}|\varDelta^L_{t,1}f_l(x)|^q dt \nonumber\\
&\lesssim\sum_{m=0}^L 2^k\int_{2^{-k}}^{2^{1-k}}|f_l(x_1+mt,x_1')|^r\cdot|f_l(x_1+mt,x_1')|^{q-r}dt \nonumber\\
&\lesssim\sum_{m=0}^L\mvint_{[2^{-k},2^{1-k}]}|f_l(x_1+mt,x_1')|^r dt\cdot 2^{(l-k)(\frac{q}{r}-1)}
          \HLmax_1(|f_l(\cdot,x_1')|^r)(x_1)^{\frac{q}{r}-1} \nonumber\\
&\lesssim 2^{(l-k)(\frac{q}{r}-1)}\HLmax_1(|f_l(\cdot,x_1')|^r)(x_1)^{\frac{q}{r}}.
\end{align}
And estimate (\ref{eq557}) is true for $0<q<\infty$ and $l>k$. Therefore we can estimate the second term in (\ref{eq475}) as follows
\begin{align}\label{eq483}
&\sum_{k\in\bbbz}2^{k(sq-\varepsilon q+1)}\sum_{l>k}2^{lq\varepsilon}
          \int_{2^{-k}}^{2^{1-k}}|\varDelta^L_{t,1}f_l(x)|^q dt \nonumber\\
&\lesssim\sum_{k\in\bbbz}2^{kq(s-\varepsilon)}\sum_{l>k}2^{lq\varepsilon}\cdot
          2^{(l-k)(\frac{q}{r}-1)}\HLmax_1(|f_l(\cdot,x_1')|^r)(x_1)^{\frac{q}{r}} \nonumber\\
&\lesssim\sum_{l\in\bbbz}\sum_{k<l}2^{kq(s-\varepsilon-\frac{1}{r}+\frac{1}{q})}\cdot
          2^{lq(\varepsilon+\frac{1}{r}-\frac{1}{q})}\HLmax_1(|f_l(\cdot,x_1')|^r)(x_1)^{\frac{q}{r}} \nonumber\\
&\lesssim\sum_{l\in\bbbz}2^{lqs}\HLmax_1(|f_l(\cdot,x_1')|^r)(x_1)^{\frac{q}{r}},
\end{align}
where the last step is because the assumption $\tilde{\sigma}^1_{pq}<s$ indicates that $s-\varepsilon-\frac{1}{r}+\frac{1}{q}>0$ if we pick $\varepsilon$ sufficiently close to $0$ and $r$ sufficiently close to $\min\{p,q\}$. Combining (\ref{eq472}), (\ref{eq475}), (\ref{eq479}) and (\ref{eq483}) altogether, raising the power to $\frac{1}{q}$ and inserting the result into $\|\cdot\|_{L^p(\bbbr^n)}$ quasinorm yield
\begin{align}\label{eq484}
&\|(\sum_{k\in\bbbz}2^{k(sq+1)}\int_{2^{-k}}^{2^{1-k}}(\sum_{l\in\bbbz}
          |\varDelta^L_{t,1}f_l(x)|)^q dt)^{\frac{1}{q}}\|_{L^p(\bbbr^n)} \nonumber\\
&\lesssim(\int_{\bbbr^{n-1}}\int_{\bbbr}(\sum_{l\in\bbbz}2^{lqs}\HLmax_1(|f_l(\cdot,x_1')|^r)(x_1)^{\frac{q}{r}}
          )^{\frac{p/r}{q/r}}dx_1 dx_1')^{\frac{1}{p}} \nonumber\\
&\lesssim(\int_{\bbbr^{n-1}}\int_{\bbbr}(\sum_{l\in\bbbz}2^{lqs}|f_l(x_1,x_1')|^q)^{\frac{p}{q}}dx_1 dx_1')^{\frac{1}{p}}
          =\|f\|_{\Fspq},
\end{align}
where we also used the $1$-dimensional version of Lemma \ref{lemma6}, and inequality (\ref{eq484}) is true for $0<p,q<\infty$, $\tilde{\sigma}^1_{pq}<s<L$. The assumption $f\in\Fspq$ and inequality (\ref{eq484}) also tell the absolute convergence of the series $\sum_{l\in\bbbz}\varDelta^L_{t,1}f_l(x)$ for every $k\in\bbbz$ and almost every $t\in[2^{-k},2^{1-k}],x\in\bbbr^n$. In conjunction with (\ref{eq513}), we have proven the following claim that
\begin{equation}\label{eq510}
\varDelta^L_{t,1}f=\sum_{l\in\bbbz}\varDelta^L_{t,1}f_l(x)\text{ in the sense of }\Sw'(\bbbr^n)/\mathscr{P}(\bbbr^n)
\end{equation}
for every $k\in\bbbz$ and almost every $t\in[2^{-k},2^{1-k}],x\in\bbbr^n$ when $0<p,q<\infty$, $\tilde{\sigma}^1_{pq}<s<L$. Therefore the tempered distribution $\varDelta^L_{t,1}f$ has a function representative which is the pointwise limit of the series $\sum_{l\in\bbbz}\varDelta^L_{t,1}f_l(x)$ and integration of $\varDelta^L_{t,1}f$ with respect to Lebesgue measure is justified. Furthermore, we have obtained the following inequality
\begin{align}\label{eq558}
&\|(\int_0^{\infty}t^{-sq}|\varDelta^L_{t,1}f|^q\frac{dt}{t})^{\frac{1}{q}}\|_{L^p(\bbbr^n)} \nonumber\\
&\lesssim\|(\sum_{k\in\bbbz}2^{k(sq+1)}\int_{2^{-k}}^{2^{1-k}}
(\sum_{l\in\bbbz}|\varDelta^L_{t,1}f_l(x)|)^q dt)^{\frac{1}{q}}\|_{L^p(\bbbr^n)},
\end{align}
and then (\ref{eq484}) and (\ref{eq558}) conclude the proof of Theorem \ref{theorem6} (i).

Next, we show that inequality (\ref{eq451}) is true under the conditions of Theorem \ref{theorem6} (ii). We assume the right side of (\ref{eq451}) is finite, otherwise the inequality is trivial. We still use the sufficiently large positive integer $m_0$ given in (\ref{eq417}). Observe that if $\xi\in spt.\FT_n\psi\subseteq A'=\{\xi\in\bbbr^n:\frac{1}{2}\leq|\xi|<2\}$, then
$$\frac{\xi_1^2+\xi_2^2+\cdots+\xi_n^2}{n}\geq\frac{1}{4n}.$$
This means given a sufficiently small positive number $\delta$, there exists at least one $\xi_j$ such that $\delta\leq|\xi_j|<2$. Therefore we obtain the decomposition $A'=\bigcup_{j=1}^n A'_j$, where $A'_j=\{\xi\in\bbbr^n:\frac{1}{2}\leq|\xi|<2,\delta\leq|\xi_j|<2\}$. Let $\{\rho_j\}_{j=1}^n$ be the partition of unity associated with this decomposition, that is, each $\rho_j$ is a smooth function with compact support in $\bbbr^n$, and $spt.\rho_j$ is contained in a small neighborhood of $A'_j$, furthermore
\begin{equation}\label{eq589}
\sum_{j=1}^n\rho_j(\xi)=1\text{ if }\xi\in A'.
\end{equation}
Without loss of generality, we can assume that
\begin{equation}\label{eq590}
\frac{1}{2}\leq|\xi|<2\text{ and }\delta\leq|\xi_j|<2\text{ for }\xi\in spt.\rho_j.
\end{equation}
Then we have
\begin{equation}\label{eq485}
\|f\|_{\Fspq}\lesssim\sum_{j=1}^n 2^{-sm_0}
\|\{2^{ks}\iFT_n[\FT_n\psi(2^{m_0-k}\xi)\rho_j(2^{m_0-k}\xi)\FT_n f](x)\}_{k\in\bbbz}\|_{L^p(l^q)}.
\end{equation}
Thus to prove (\ref{eq451}), it is sufficient to prove
\begin{align}\label{eq486}
&\|\{2^{ks}\iFT_n[\FT_n\psi(2^{m_0-k}\xi)\rho_j(2^{m_0-k}\xi)\FT_n f](x)\}_{k\in\bbbz}\|_{L^p(l^q)} \nonumber\\
&\lesssim C_{0}\|(\int_0^{\infty}t^{-sq}|\varDelta^L_{t,j}f(\cdot)|^q\frac{dt}{t})^{\frac{1}{q}}\|_{L^p(\bbbr^n)}
          +C_{00}\|f\|_{\Fspq}
\end{align}
for every $j\in\{1,2,\cdots,n\}$ and $C_{00}$ is a positive constant that can be arbitrarily small. We only need to prove (\ref{eq486}) for $j=1$, and the cases for $j=2,\cdots,n$ can be proved in the same way. Notice that both $\FT_n\psi(2^{m_0-k}\xi)$ and $\rho_1(2^{m_0-k}\xi)$ are supported in a ball centered at $0$ of radius $2^{k+1-m_0}$ in $\bbbr^n$, thus using the argument of Remark \ref{remark3}, we have
\begin{align}\label{eq487}
&|\iFT_n[\FT_n\psi(2^{m_0-k}\xi)\rho_1(2^{m_0-k}\xi)\FT_n f](x)| \nonumber\\
&\lesssim\PFSmax_n\{\iFT_n[\FT_n\psi(2^{m_0-k}\xi)\rho_1(2^{m_0-k}\xi)\FT_n f]\}(x) \nonumber\\
&\lesssim\PFSmax_n\{\iFT_n[\rho_1(2^{m_0-k}\xi)\FT_n f]\}(x).
\end{align}
By Lemma \ref{lemma4}, Remark \ref{remark12} and Lemma \ref{lemma6}, we have
\begin{align}\label{eq488}
&\|\{2^{ks}\iFT_n[\FT_n\psi(2^{m_0-k}\xi)\rho_1(2^{m_0-k}\xi)\FT_n f](x)\}_{k\in\bbbz}\|_{L^p(l^q)} \nonumber\\
&\lesssim\|\{2^{ks}\iFT_n[\rho_1(2^{m_0-k}\xi)\FT_n f](x)\}_{k\in\bbbz}\|_{L^p(l^q)}.
\end{align}
If $1<\min\{p,q\}$ and $s\in\bbbr$, then we have
\begin{align}\label{eq489}
&\iFT_n[\rho_1(2^{m_0-k}\xi)\FT_n f](x) \nonumber\\
&=\mvint_{[1,2]}\iFT_n[\frac{\rho_1(2^{m_0-k}\xi)}{(e^{2\pi i\cdot 2^{-k}t\xi_1}-1)^L}\cdot
   (e^{2\pi i\cdot 2^{-k}t\xi_1}-1)^L \FT_n f](x)\frac{dt}{t} \nonumber\\
&=\mvint_{[1,2]}\iFT_n[\frac{\rho_1(2^{m_0-k}\xi)}{(e^{2\pi i\cdot 2^{-k}t\xi_1}-1)^L}]*
   \varDelta^L_{2^{-k}t,1}f(x)\frac{dt}{t}.
\end{align}
According to the support condition of $\rho_1$, when $\rho_1(2^{m_0}\xi)\neq 0$ and $t\in[1,2]$, we have $0<2^{-m_0}\delta\leq t|\xi_1|<2^{2-m_0}$ and thus $|(e^{2\pi i\cdot t\xi_1}-1)^L|\geq c>0$ for some constant $c$ independent of $t$ and $\xi_1$. Using the same method as in Lemma \ref{lemma2}, in particular since a similar condition like (\ref{eq409}) is satisfied because of the assumption on $\rho_1$, we can obtain
\begin{equation}\label{eq490}
|\iFT_n[\frac{\rho_1(2^{m_0}\xi)}{(e^{2\pi i\cdot t\xi_1}-1)^L}](y)|\lesssim
\frac{1}{(1+|y|)^N}
\end{equation}
for arbitrarily large positive integer $N$, and the constant is independent of $t$. We still use the notation $A_{k-l}=\{y\in\bbbr^n:2^{l-k}\leq|y|<2^{1+l-k}\}$ and hence
\begin{align}\label{eq491}
&|\iFT_n[\frac{\rho_1(2^{m_0-k}\xi)}{(e^{2\pi i\cdot 2^{-k}t\xi_1}-1)^L}]*\varDelta^L_{2^{-k}t,1}f(x)| \nonumber\\
&\lesssim\sum_{l\in\bbbz}\int_{A_{k-l}}2^{kn}|\iFT_n[\frac{\rho_1(2^{m_0}\xi)}{(e^{2\pi i\cdot t\xi_1}-1)^L}](2^k y)|\cdot
          |\varDelta^L_{2^{-k}t,1}f(x-y)|dy \nonumber\\
&\lesssim\sum_{l\leq 0}2^{ln}\mvint_{A_{k-l}}|\varDelta^L_{2^{-k}t,1}f(x-y)|dy \nonumber\\
&\quad+\sum_{l>0}2^{l(n-N)}\mvint_{A_{k-l}}|\varDelta^L_{2^{-k}t,1}f(x-y)|dy.
\end{align}
Insert (\ref{eq491}) into (\ref{eq489}), exchange the order of integration, and use the inequality
$$\mvint_{A_{k-l}}\int_{[1,2]}|\varDelta^L_{2^{-k}t,1}f(x-y)|\frac{dt}{t}dy\lesssim
\HLmax_n(\int_{[1,2]}|\varDelta^L_{2^{-k}t,1}f(\cdot)|\frac{dt}{t})(x),$$
then we can obtain
\begin{align}\label{eq492}
&|\iFT_n[\rho_1(2^{m_0-k}\xi)\FT_n f](x)| \nonumber\\
&\lesssim(\sum_{l\leq 0}2^{ln}+\sum_{l>0}2^{l(n-N)})\HLmax_n(\int_{[1,2]}|\varDelta^L_{2^{-k}t,1}f(\cdot)|\frac{dt}{t})(x) \nonumber\\
&\lesssim\HLmax_n(\int_{[1,2]}|\varDelta^L_{2^{-k}t,1}f(\cdot)|\frac{dt}{t})(x),
\end{align}
if we pick $N>n$. Inserting (\ref{eq492}) into (\ref{eq488}), applying Lemma \ref{lemma6} which requires the condition $1<\min\{p,q\}$, and also using H\"{o}lder's inequality for $1<q$ yield the following
\begin{align}\label{eq493}
&\|\{2^{ks}\iFT_n[\FT_n\psi(2^{m_0-k}\xi)\rho_1(2^{m_0-k}\xi)\FT_n f](x)\}_{k\in\bbbz}\|_{L^p(l^q)} \nonumber\\
&\lesssim\|\{2^{ks}\HLmax_n(\int_{[1,2]}|\varDelta^L_{2^{-k}t,1}f(\cdot)|\frac{dt}{t})(x)\}_{k\in\bbbz}\|_{L^p(l^q)} \nonumber\\
&\lesssim\|\{2^{ks}\int_{[1,2]}|\varDelta^L_{2^{-k}t,1}f(x)|\frac{dt}{t}\}_{k\in\bbbz}\|_{L^p(l^q)} \nonumber\\
&\lesssim\|(\sum_{k\in\bbbz}2^{ksq}\int_{[1,2]}|\varDelta^L_{2^{-k}t,1}f(x)|^q\frac{dt}{t})^{\frac{1}{q}}\|_{L^p(\bbbr^n)} \nonumber\\
&\lesssim\|(\int_0^{\infty}t^{-sq}|\varDelta^L_{t,1}f(x)|^q\frac{dt}{t})^{\frac{1}{q}}\|_{L^p(\bbbr^n)},
\end{align}
and this inequality is true for any $s\in\bbbr$. By now we have proven (\ref{eq486}) for $1\leq j\leq n$ and $C_{00}=0$, and inserting these inequalities back into (\ref{eq485}) proves (\ref{eq451}) under the conditions of Theorem \ref{theorem6} (ii) when $1<\min\{p,q\}$, $q<\infty$ and $s\in\bbbr$. Now we show that (\ref{eq451}) is still true under the conditions of Theorem \ref{theorem6} (ii) when $\min\{p,q\}\leq 1$, $q<\infty$ and $\sigma_{pq}+\tilde{\sigma}^1_{pq}<s<\infty$, we will have to use the $n$-dimensional Plancherel-Polya-Nikol'skij inequality and hence introduce $\sigma_{pq}$, a number depending on the dimension $n$, into the restriction of $s$. We use the function $\phi$ satisfying conditions (\ref{eq346}), (\ref{eq422}) and (\ref{eq423}), and $J$ is still a large positive integer whose value will be determined later. Because $spt.\rho_1(2^{m_0-k}\xi)\subseteq\{\xi\in\bbbr^n:|\xi|<2^{k+1-m_0}\}$, the $n$-dimensional Fourier transform of the Schwartz function
$$y\mapsto\iFT_n[\frac{\rho_1(2^{m_0-k}\xi)}{(e^{2\pi i\cdot 2^{-k}t\xi_1}-1)^L}](y)\cdot
\varDelta^L_{2^{-k}t,1}(\phi_{2^{m_0-J-k}}*f)(x-y)$$
is supported in a ball of radius about $2^{J+k-m_0}$, centered at the origin in $\bbbr^n$. Therefore by using Plancherel-Polya-Nikol'skij inequality or the more general Lemma \ref{lemma7} and the condition $0<r<\min\{p,q\}\leq 1$, we have
\begin{align}\label{eq494}
&|\iFT_n[\frac{\rho_1(2^{m_0-k}\xi)}{(e^{2\pi i\cdot 2^{-k}t\xi_1}-1)^L}]*
          [\varDelta^L_{2^{-k}t,1}(\phi_{2^{m_0-J-k}}*f)(\cdot)](x)|^r \nonumber\\
&\lesssim(\int_{\bbbr^n}|\iFT_n[\frac{\rho_1(2^{m_0-k}\xi)}{(e^{2\pi i\cdot 2^{-k}t\xi_1}-1)^L}](y)\cdot
          \varDelta^L_{2^{-k}t,1}(\phi_{2^{m_0-J-k}}*f)(x-y)|dy)^r \nonumber\\
&\lesssim 2^{(J+k-m_0)n(1-r)}\int_{\bbbr^n}2^{knr}
          |\iFT_n[\frac{\rho_1(2^{m_0}\xi)}{(e^{2\pi i\cdot t\xi_1}-1)^L}](2^k y)|^r \nonumber\\
&\quad\cdot|\varDelta^L_{2^{-k}t,1}(\phi_{2^{m_0-J-k}}*f)(x-y)|^r dy.
\end{align}
Recall that $\bbbr^n=\bigcup_{l\in\bbbz}A_{k-l}$ where $A_{k-l}$ is the annulus $\{y\in\bbbr^n:2^{l-k}\leq|y|<2^{1+l-k}\}$ and use (\ref{eq490}) with a sufficiently large positive integer $N'>\frac{n}{r}$, then we can estimate (\ref{eq494}) from above by
\begin{align}\label{eq495}
&2^{Jn(1-r)}\cdot\{\sum_{l\leq 0}2^{ln}\mvint_{A_{k-l}}|\varDelta^L_{2^{-k}t,1}(\phi_{2^{m_0-J-k}}*f)(x-y)|^r dy \nonumber\\
&\quad+\sum_{l>0}2^{l(n-N'r)}\mvint_{A_{k-l}}|\varDelta^L_{2^{-k}t,1}(\phi_{2^{m_0-J-k}}*f)(x-y)|^r dy\}.
\end{align}
Therefore from (\ref{eq417}) we know that when $\rho_1(2^{m_0-k}\xi)\neq 0$ and $t\in[1,2]$, $$|(e^{2\pi i\cdot 2^{-k}t\xi_1}-1)^L|>0,$$ and we can obtain
\begin{align}\label{eq496}
&|\iFT_n[\rho_1(2^{m_0-k}\xi)\FT_n f](x)|^r \nonumber\\
&=\mvint_{[1,2]}|\iFT_n[\frac{\rho_1(2^{m_0-k}\xi)}{(e^{2\pi i\cdot 2^{-k}t\xi_1}-1)^L}\cdot
   (e^{2\pi i\cdot 2^{-k}t\xi_1}-1)^L \nonumber\\
&\quad\cdot\FT_n\phi(2^{m_0-J-k}\xi)\FT_n f](x)|^r\frac{dt}{t} \nonumber\\
&=\mvint_{[1,2]}|\iFT_n[\frac{\rho_1(2^{m_0-k}\xi)}{(e^{2\pi i\cdot 2^{-k}t\xi_1}-1)^L}]*
   [\varDelta^L_{2^{-k}t,1}(\phi_{2^{m_0-J-k}}*f)(\cdot)](x)|^r\frac{dt}{t} \nonumber\\
&\lesssim 2^{Jn(1-r)}\cdot\{\sum_{l\leq 0}2^{ln}\mvint_{A_{k-l}}
          \int_1^2|\varDelta^L_{2^{-k}t,1}(\phi_{2^{m_0-J-k}}*f)(x-y)|^r\frac{dt}{t}dy \nonumber\\
&\quad+\sum_{l>0}2^{l(n-N'r)}\mvint_{A_{k-l}}\int_1^2|\varDelta^L_{2^{-k}t,1}(\phi_{2^{m_0-J-k}}*f)(x-y)|^r\frac{dt}{t}dy\}.
\end{align}
Using the fact that $\mvint_{A_{k-l}}\int_1^2|\varDelta^L_{2^{-k}t,1}(\phi_{2^{m_0-J-k}}*f)(x-y)|^r\frac{dt}{t}dy$ can be dominated by $$\HLmax_n(\int_1^2|\varDelta^L_{2^{-k}t,1}(\phi_{2^{m_0-J-k}}*f)(\cdot)|^r\frac{dt}{t})(x),$$ we obtain
\begin{align}\label{eq497}
&|\iFT_n[\rho_1(2^{m_0-k}\xi)\FT_n f](x)| \nonumber\\
&\lesssim 2^{Jn(\frac{1}{r}-1)}
          \HLmax_n(\int_1^2|\varDelta^L_{2^{-k}t,1}(\phi_{2^{m_0-J-k}}*f)(\cdot)|^r\frac{dt}{t})(x)^{\frac{1}{r}}.
\end{align}
Inserting (\ref{eq497}) into (\ref{eq488}) and using Lemma \ref{lemma6} and H\"{o}lder's inequality since $0<r<\min\{p,q\}$ yield
\begin{align}\label{eq498}
&\|\{2^{ks}\iFT_n[\FT_n\psi(2^{m_0-k}\xi)\rho_1(2^{m_0-k}\xi)\FT_n f](\cdot)\}_{k\in\bbbz}\|_{L^p(l^q)} \nonumber\\
&\lesssim 2^{Jn(\frac{1}{r}-1)}\cdot
          \|\{2^{ks}(\int_1^2|\varDelta^L_{2^{-k}t,1}(\phi_{2^{m_0-J-k}}*f)(\cdot)|^r\frac{dt}{t}
          )^{\frac{1}{r}}\}_{k\in\bbbz}\|_{L^p(l^q)} \nonumber\\
&\lesssim 2^{Jn(\frac{1}{r}-1)}\cdot
          \|(\sum_{k\in\bbbz}2^{ksq}\int_1^2|\varDelta^L_{2^{-k}t,1}(\phi_{2^{m_0-J-k}}*f)(\cdot)|^q\frac{dt}{t}
          )^{\frac{1}{q}}\|_{L^p(\bbbr^n)}.
\end{align}
Inferring from (\ref{eq500}), (\ref{eq514}) and assuming the validity of decomposition for now, then (\ref{eq498}) can be estimated from above by the sum of the two terms
\begin{align}\label{eq499}
&2^{Jn(\frac{1}{r}-1)}\cdot
 \|(\sum_{k\in\bbbz}2^{ksq}\int_1^2|\varDelta^L_{2^{-k}t,1}f(\cdot)|^q\frac{dt}{t})^{\frac{1}{q}}\|_{L^p(\bbbr^n)} \nonumber\\
& \quad\sim 2^{Jn(\frac{1}{r}-1)}\cdot
      \|(\int_0^{\infty}t^{-sq}|\varDelta^L_{t,1}f(\cdot)|^q\frac{dt}{t})^{\frac{1}{q}}\|_{L^p(\bbbr^n)}<\infty,
\end{align}
and
\begin{equation}\label{eq501}
2^{Jn(\frac{1}{r}-1)}\cdot
\|(\sum_{k\in\bbbz}2^{ksq}\int_1^2|\varDelta^L_{2^{-k}t,1}(\sum_{l=J+1}^{\infty}f_{k+l-m_0})
(\cdot)|^q\frac{dt}{t})^{\frac{1}{q}}\|_{L^p(\bbbr^n)}.
\end{equation}
To estimate (\ref{eq501}), we begin by applying the calculation method used for obtaining the second term of (\ref{eq475}) to the following term below and obtain
\begin{align}\label{eq561}
&2^{Jn(\frac{1}{r}-1)}\|(\sum_{k\in\bbbz}2^{ksq}\int_1^2(\sum_{l=J+1}^{\infty}|\varDelta^L_{2^{-k}t,1}f_{k+l-m_0}
          (\cdot)|)^q\frac{dt}{t})^{\frac{1}{q}}\|_{L^p(\bbbr^n)} \nonumber\\
&\lesssim 2^{Jn(\frac{1}{r}-1)-J\varepsilon}\|(\sum_{k\in\bbbz}\sum_{l>J}2^{ksq+k+lq\varepsilon}
          \int_{2^{-k}}^{2^{1-k}}|\varDelta^L_{t,1}f_{k+l-m_0}(\cdot)|^q dt)^{\frac{1}{q}}\|_{L^p(\bbbr^n)},
\end{align}
where we only need $\varepsilon>0$. Considering $x_1'=(x_2,\cdots,x_n)\in\bbbr^{n-1}$ is fixed for now, we use Lemma \ref{lemma8} and estimate (\ref{eq557}) since $k+l-m_0>k+J-m_0>k$ and then we can obtain the following
\begin{equation}\label{eq504}
2^k\int_{2^{-k}}^{2^{1-k}}|\varDelta^L_{t,1}f_{k+l-m_0}(x)|^q dt\lesssim
2^{l(\frac{q}{r}-1)}\HLmax_1(|f_{k+l-m_0}(\cdot,x_1')|^r)(x_1)^{\frac{q}{r}}.
\end{equation}
Inserting (\ref{eq504}) into (\ref{eq561}), because the assumption $\sigma_{pq}+\tilde{\sigma}^1_{pq}<s<\infty$ implies
\begin{equation}\label{eq505}
n(\frac{1}{r}-1)+(\frac{1}{r}-\frac{1}{q})<s\qquad\text{and}\qquad
\varepsilon+\frac{1}{r}-\frac{1}{q}<s
\end{equation}
when $r$ is sufficiently close to $\min\{p,q\}$ and $\varepsilon$ is sufficiently close to $0$, we can estimate (\ref{eq561}) from above by the following
\begin{equation}\label{eq506}
2^{Jn(\frac{1}{r}-1)-J\varepsilon}\cdot(\int_{\bbbr^n}(\sum_{l>J}2^{lq(\varepsilon+\frac{1}{r}-\frac{1}{q}-s)}\cdot
\sum_{k\in\bbbz}2^{ksq}\HLmax_1(|f_{k}(\cdot,x_1')|^r)(x_1)^{\frac{q}{r}})^{\frac{p}{q}}dx)^{\frac{1}{p}},
\end{equation}
and this term can be further estimated from above by
\begin{equation}\label{eq507}
2^{J[n(\frac{1}{r}-1)+(\frac{1}{r}-\frac{1}{q})-s]}\|f\|_{\Fspq},
\end{equation}
due to the $1$-dimensional version of Lemma \ref{lemma6}. Putting together (\ref{eq561}), (\ref{eq506}) and (\ref{eq507}) yields
\begin{align}\label{eq559}
&2^{Jn(\frac{1}{r}-1)}\|(\sum_{k\in\bbbz}2^{ksq}\int_1^2(\sum_{l=J+1}^{\infty}|\varDelta^L_{2^{-k}t,1}f_{k+l-m_0}
(\cdot)|)^q\frac{dt}{t})^{\frac{1}{q}}\|_{L^p(\bbbr^n)} \nonumber\\
&\lesssim 2^{J[n(\frac{1}{r}-1)+(\frac{1}{r}-\frac{1}{q})-s]}\|f\|_{\Fspq}.
\end{align}
Inequality (\ref{eq559}) and the assumption of $f$ being a member of $\Fspq$ also suggest that $\sum_{l=J+1}^{\infty}|\varDelta^L_{2^{-k}t,1}f_{k+l-m_0}(x)|<\infty$ for every $k\in\bbbz$ and almost every $t\in[1,2],x\in\bbbr^n$. From (\ref{eq514}), (\ref{eq499}), the above inference and the supposition of $f$ being a function, we can deduce that
\begin{equation}\label{eq560}
\varDelta^L_{2^{-k}t,1}(\phi_{2^{m_0-J-k}}*f)(x)=\varDelta^L_{2^{-k}t,1}f(x)
-\sum_{l=J+1}^{\infty}\varDelta^L_{2^{-k}t,1}f_{k+l-m_0}(x)
\end{equation}
in the sense of $\Sw'(\bbbr^n)$ for every $k\in\bbbz$ and almost every $t\in[1,2],x\in\bbbr^n$ when $\min\{p,q\}\leq 1$, $q<\infty$ and $\sigma_{pq}+\tilde{\sigma}^1_{pq}<s<\infty$, furthermore estimating (\ref{eq498}) from above by the sum of (\ref{eq499}) and (\ref{eq501}) is justified, moreover (\ref{eq501}) can be estimated from above by (\ref{eq561}) and hence by (\ref{eq559}). We have reached the conclusion
\begin{align}\label{eq508}
&\|\{2^{ks}\iFT_n[\FT_n\psi(2^{m_0-k}\xi)\rho_1(2^{m_0-k}\xi)\FT_n f(\xi)](\cdot)\}_{k\in\bbbz}\|_{L^p(l^q)} \nonumber\\
&\leq C_1 2^{Jn(\frac{1}{r}-1)}
      \|(\int_0^{\infty}t^{-sq}|\varDelta^L_{t,1}f(\cdot)|^q\frac{dt}{t})^{\frac{1}{q}}\|_{L^p(\bbbr^n)} \nonumber\\
&\quad+C_1' 2^{J[n(\frac{1}{r}-1)+(\frac{1}{r}-\frac{1}{q})-s]}\|f\|_{\Fspq}.
\end{align}
In a similar way, we can also prove (\ref{eq508}) if we replace $\rho_1$, $\varDelta^L_{t,1}$, $C_1$, $C_1'$ by $\rho_j$, $\varDelta^L_{t,j}$, $C_j$, $C_j'$ respectively for $j=2,\cdots,n$. From (\ref{eq485}) we have obtained
\begin{align}\label{eq509}
\|f\|_{\Fspq}
&\leq C' 2^{Jn(\frac{1}{r}-1)}\sum_{j=1}^n
      \|(\int_0^{\infty}t^{-sq}|\varDelta^L_{t,j}f(\cdot)|^q\frac{dt}{t})^{\frac{1}{q}}\|_{L^p(\bbbr^n)} \nonumber\\
&\quad+C'' 2^{J[n(\frac{1}{r}-1)+(\frac{1}{r}-\frac{1}{q})-s]}\|f\|_{\Fspq}.
\end{align}
By (\ref{eq505}) we can pick a sufficiently large positive integer $J$ so that the coefficient
$$C'' 2^{J[n(\frac{1}{r}-1)+(\frac{1}{r}-\frac{1}{q})-s]}<\frac{1}{2},$$ and shift the second term on the right side of (\ref{eq509}) to its left side and hence complete the proof of (\ref{eq451}) when $\min\{p,q\}\leq 1$, $q<\infty$ and $\sigma_{pq}+\tilde{\sigma}^1_{pq}<s<\infty$.

Now we prove Theorem \ref{theorem6} (iii). We only need to prove inequality (\ref{eq562}) when $j=1$ and the other cases when $j=2,\cdots,n$ can be proved in the same way. We begin with estimating $\|\esssup_{t>0}t^{-s}\sum_{l\in\bbbz}|\varDelta^L_{t,1}f_l(\cdot)|\|_{L^p(\bbbr^n)}$ from above by the sum
\begin{align}\label{eq564}
&\|\esssup_{k\in\bbbz}\esssup_{t\in[2^{-k},2^{1-k})}2^{ks}\sum_{l\leq k}|\varDelta^L_{t,1}f_l(\cdot)|\|_{L^p(\bbbr^n)} \nonumber\\
&\quad+\|\esssup_{k\in\bbbz}\esssup_{t\in[2^{-k},2^{1-k})}2^{ks}\sum_{l>k}|\varDelta^L_{t,1}f_l(\cdot)|\|_{L^p(\bbbr^n)}.
\end{align}
Pick $0<\varepsilon<\min\{s-\frac{1}{p},L-s\}$ and let $x_1'=(x_2,\cdots,x_n)\in\bbbr^{n-1}$ be fixed for now. When $l\leq k$, we use Lemma \ref{lemma8} and (\ref{eq478}) and calculate as follows
\begin{align}\label{eq565}
&\esssup_{k\in\bbbz}2^{ks}\sum_{l\leq k}\esssup_{t\in[2^{-k},2^{1-k})}|\varDelta^L_{t,1}f_l(x_1,x_1')| \nonumber\\
&\lesssim\esssup_{k\in\bbbz}2^{k(s-L)}\sum_{l\leq k}2^{l\varepsilon}\cdot 2^{l(L-\varepsilon)}
          \PFSmax_1 f_l(\cdot,x_1')(x_1) \nonumber\\
&\lesssim\esssup_{l\in\bbbz}\esssup_{k\geq l}2^{k(s+\varepsilon-L)}\cdot 2^{l(L-\varepsilon)}\PFSmax_1 f_l(\cdot,x_1')(x_1) \nonumber\\
&\lesssim\esssup_{l\in\bbbz}2^{ls}\HLmax_1(|f_l(\cdot,x_1')|^r)(x_1)^{\frac{1}{r}},
\end{align}
where we also used the $1$-dimensional version of Lemma \ref{lemma4} and $0<r<p$. Inserting (\ref{eq565}) into $\|\cdot\|_{L^p(\bbbr^n)}$ quasinorm and invoking the mapping property of the $1$-dimensional Hardy-Littlewood maximal function, we can estimate the first term in (\ref{eq564}) from above by the following
\begin{align}\label{eq566}
&\|\esssup_{l\in\bbbz}2^{ls}\HLmax_1(|f_l(\cdot,x_1')|^r)(x_1)^{\frac{1}{r}}\|_{L^p(\bbbr^n)} \nonumber\\
&\lesssim(\int_{\bbbr^{n-1}}\int_{\bbbr}\HLmax_1(\esssup_{l\in\bbbz}2^{lsr}|f_l(\cdot,x_1')|^r)(x_1)^{\frac{p}{r}}
          dx_1 dx_1')^{\frac{1}{p}} \nonumber\\
&\lesssim(\int_{\bbbr^{n-1}}\int_{\bbbr}\esssup_{l\in\bbbz}2^{lsp}|f_l(x_1,x_1')|^p dx_1 dx_1')^{\frac{1}{p}}
          =\|f\|_{\Fspinf}.
\end{align}
When $l>k$ and $|t|\lesssim 2^{-k}$, we use (\ref{eq480}) and (\ref{eq481}) with temporarily fixed $x_1'\in\bbbr^{n-1}$ to obtain
\begin{equation}\label{eq567}
\esssup_{t\in[2^{-k},2^{1-k})}|\varDelta^L_{t,1}f_l(x_1,x_1')|\lesssim 2^{\frac{l-k}{r}}\HLmax_1(|f_l(\cdot,x_1')|^r)(x_1)^{\frac{1}{r}}.
\end{equation}
Since $\frac{1}{p}<s<L$ by the assumption, we can pick $\varepsilon>0$ to be sufficiently close to $0$ and $r$ to be sufficiently close to $p$ so that $s-\frac{1}{r}-\varepsilon$ is a positive finite number, and then we can obtain
\begin{align}\label{eq568}
&\esssup_{k\in\bbbz}2^{ks}\sum_{l>k}\esssup_{t\in[2^{-k},2^{1-k})}|\varDelta^L_{t,1}f_l(x_1,x_1')| \nonumber\\
&\lesssim\esssup_{k\in\bbbz}2^{k(s-\frac{1}{r})}\sum_{l>k}2^{-l\varepsilon}\cdot 2^{l(\varepsilon+\frac{1}{r})}
          \HLmax_1(|f_l(\cdot,x_1')|^r)(x_1)^{\frac{1}{r}} \nonumber\\
&\lesssim\esssup_{l\in\bbbz}\esssup_{k<l}2^{k(s-\frac{1}{r}-\varepsilon)}\cdot 2^{l(\varepsilon+\frac{1}{r})}
          \HLmax_1(|f_l(\cdot,x_1')|^r)(x_1)^{\frac{1}{r}} \nonumber\\
&\lesssim\esssup_{l\in\bbbz}2^{ls}\HLmax_1(|f_l(\cdot,x_1')|^r)(x_1)^{\frac{1}{r}}.
\end{align}
Inserting (\ref{eq568}) into $\|\cdot\|_{L^p(\bbbr^n)}$ quasinorm and proceeding as in (\ref{eq566}), we can estimate the second term in (\ref{eq564}) from above by $\|f\|_{\Fspinf}$. Therefore we have obtained the inequality
\begin{equation}\label{eq569}
\|\esssup_{k\in\bbbz}\esssup_{t\in[2^{-k},2^{1-k})} 2^{ks}\sum_{l\in\bbbz}|\varDelta^L_{t,1}f_l(\cdot)|\|_{L^p(\bbbr^n)}\lesssim\|f\|_{\Fspinf},
\end{equation}
when $0<p<\infty$, $q=\infty$ and $\frac{1}{p}<s<L$. The assumption $f\in\Fspinf$ and the above inequality also show $\sum_{l\in\bbbz}|\varDelta^L_{t,1}f_l(x)|<\infty$ for every $k\in\bbbz$ and almost every $t\in[2^{-k},2^{1-k}),x\in\bbbr^n$. In conjunction with (\ref{eq513}), we have reached the conclusion that
\begin{equation}\label{eq570}
\varDelta^L_{t,1}f=\sum_{l\in\bbbz}\varDelta^L_{t,1}f_l(x)
\end{equation}
in the sense of $\Sw'(\bbbr^n)/\mathscr{P}(\bbbr^n)$ for every $k\in\bbbz$ and almost every $t\in[2^{-k},2^{1-k}),x\in\bbbr^n$ when $0<p<\infty$, $q=\infty$ and $\frac{1}{p}<s<L$, furthermore we also obtain
\begin{equation}\label{eq571}
\|\esssup_{t>0}\frac{|\varDelta^L_{t,1}f|}{t^s}\|_{L^p(\bbbr^n)}\lesssim
\|\esssup_{k\in\bbbz}\esssup_{t\in[2^{-k},2^{1-k})}2^{ks}\sum_{l\in\bbbz}|\varDelta^L_{t,1}f_l(\cdot)|\|_{L^p(\bbbr^n)}.
\end{equation}
Inequalities (\ref{eq569}) and (\ref{eq571}) conclude the proof of Theorem \ref{theorem6} (iii).

Finally, we come to the proof of Theorem \ref{theorem6} (iv). We assume the right side of (\ref{eq563}) is finite, otherwise the inequality is trivial. We also use the sufficiently large positive integer $m_0$ given in (\ref{eq417}) and let $0<r<p$. We continue using the partition of unity $\{\rho_j\}_{j=1}^n$ associated with the set $A'=\{\xi\in\bbbr^n:\frac{1}{2}\leq|\xi|<2\}$ introduced at the beginning of the proof of Theorem \ref{theorem6} (ii), and also continue assuming $\frac{1}{2}\leq|\xi|<2$ and $\delta\leq|\xi_j|<2$ for $\xi\in spt.\rho_j$ and $\delta$ being a sufficiently small positive number. Then by using Remark \ref{remark3}, Lemma \ref{lemma4}, and the mapping property of the Hardy-Littlewood maximal function, we can obtain the estimate
\begin{align}\label{eq572}
&\|f\|_{\Fspinf} \nonumber\\
&\lesssim 2^{-m_0 s}\sum_{j=1}^n\|\esssup_{k\in\bbbz}2^{ks}|\iFT_n(\FT_n\psi(2^{m_0-k}\xi)\rho_j(2^{m_0-k}\xi)
          \FT_n f)(\cdot)|\|_{L^p(\bbbr^n)} \nonumber\\
&\lesssim\sum_{j=1}^n\|\esssup_{k\in\bbbz}2^{ks}\PFSmax_n(\iFT_n[\rho_j(2^{m_0-k}\xi)\FT_n f])(\cdot)\|_{L^p(\bbbr^n)} \nonumber\\
&\lesssim\sum_{j=1}^n\|\esssup_{k\in\bbbz}2^{ks}
          \HLmax_n(|\iFT_n[\rho_j(2^{m_0-k}\xi)\FT_n f]|^r)(\cdot)^{\frac{1}{r}}\|_{L^p(\bbbr^n)} \nonumber\\
&\lesssim\sum_{j=1}^n\|\esssup_{k\in\bbbz}2^{ks}|\iFT_n[\rho_j(2^{m_0-k}\xi)\FT_n f](\cdot)|\|_{L^p(\bbbr^n)}.
\end{align}
We now estimate each term in the sum of (\ref{eq572}). It suffices to provide estimate for the term $\|\esssup_{k\in\bbbz}2^{ks}|\iFT_n[\rho_1(2^{m_0-k}\xi)\FT_n f](\cdot)|\|_{L^p(\bbbr^n)}$ when $j=1$, and estimates for the terms when $j=2,\cdots,n$ can be obtained in the same way. When $1<p<\infty$, we have
\begin{align}\label{eq657}
&|\iFT_n[\rho_1(2^{m_0-k}\xi)\FT_n f](x)| \nonumber\\
&=\esssup_{t\in[1,2]}|\iFT_n[\rho_1(2^{m_0-k}\xi)\FT_n f](x)| \nonumber\\
&=\esssup_{t\in[1,2]}|\iFT_n[\frac{\rho_1(2^{m_0-k}\xi)}{(e^{2\pi i\cdot 2^{-k}t\xi_1}-1)^L}
   \cdot(e^{2\pi i\cdot 2^{-k}t\xi_1}-1)^L\FT_n f](x)| \nonumber\\
&=\esssup_{t\in[1,2]}|\iFT_n[\frac{\rho_1(2^{m_0-k}\xi)}{(e^{2\pi i\cdot 2^{-k}t\xi_1}-1)^L}]*\varDelta^L_{2^{-k}t,1}f(x)|.
\end{align}
We use (\ref{eq491}) and the following estimate
\begin{equation}\label{eq573}
\mvint_{A_{k-l}}\esssup_{t\in[1,2]}|\varDelta^L_{2^{-k}t,1}f(x-y)|dy\lesssim
\HLmax_n(\esssup_{t\in[1,2]}|\varDelta^L_{2^{-k}t,1}f(\cdot)|)(x),
\end{equation}
and then we can estimate $|\iFT_n[\rho_1(2^{m_0-k}\xi)\FT_n f](x)|$ from above by the following
\begin{equation}\label{eq574}
(\sum_{l\leq 0}2^{ln}+\sum_{l>0}2^{l(n-N)})\HLmax_n(\esssup_{t\in[1,2]}|\varDelta^L_{2^{-k}t,1}f|)(x)
\lesssim\HLmax_n(\esssup_{t\in[1,2]}|\varDelta^L_{2^{-k}t,1}f|)(x)
\end{equation}
for every $x\in\bbbr^n$ if we pick $N>n$. Therefore when $1<p<\infty$, $q=\infty$ and $s\in\bbbr$, we invoke the mapping property of Hardy-Littlewood maximal function and obtain
\begin{align}\label{eq575}
&\|\esssup_{k\in\bbbz}2^{ks}|\iFT_n[\rho_1(2^{m_0-k}\xi)\FT_n f]|\|_{L^p(\bbbr^n)} \nonumber\\
&\lesssim\|\HLmax_n(\esssup_{k\in\bbbz}\esssup_{t\in[1,2]}2^{ks}|\varDelta^L_{2^{-k}t,1}f|)\|_{L^p(\bbbr^n)} \nonumber\\
&\lesssim\|\esssup_{t>0}\frac{|\varDelta^L_{t,1}f(\cdot)|}{t^s}\|_{L^p(\bbbr^n)}.
\end{align}
Inequality (\ref{eq575}) is still true if we replace $\rho_1,\varDelta^L_{2^{-k}t,1},\varDelta^L_{t,1}$ by $\rho_j,\varDelta^L_{2^{-k}t,j},\varDelta^L_{t,j}$ respectively for $j=2,\cdots,n$. Inserting these inequalities back into (\ref{eq572}) proves the first part of Theorem \ref{theorem6} (iv). To prove the second part of Theorem \ref{theorem6} (iv), we use the function $\phi$ satisfying conditions (\ref{eq346}), (\ref{eq422}), (\ref{eq423}), and the same $m_0$ as in (\ref{eq417}), and the large positive integer $J>m_0$ whose value will be determined later. From (\ref{eq494}), (\ref{eq495}) and the inequality
\begin{align}\label{eq576}
&\esssup_{t\in[1,2]}\mvint_{A_{k-l}}|\varDelta^L_{2^{-k}t,1}(\phi_{2^{m_0-J-k}}*f)(x-y)|^r dy \nonumber\\
&\lesssim\HLmax_n(\esssup_{t\in[1,2]}|\varDelta^L_{2^{-k}t,1}(\phi_{2^{m_0-J-k}}*f)|^r)(x),
\end{align}
we can obtain
\begin{align}\label{eq577}
&|\iFT_n[\rho_1(2^{m_0-k}\xi)\FT_n f](x)|^r \nonumber\\
&=\esssup_{t\in[1,2]}|\iFT_n[\frac{\rho_1(2^{m_0-k}\xi)}{(e^{2\pi i\cdot 2^{-k}t\xi_1}-1)^L}]*
   [\varDelta^L_{2^{-k}t,1}(\phi_{2^{m_0-J-k}}*f)](x)|^r \nonumber\\
&\lesssim 2^{Jn(1-r)}\cdot(\sum_{l\leq 0}2^{ln}+\sum_{l>0}2^{l(n-N'r)}) \nonumber\\
&\quad\cdot\HLmax_n(\esssup_{t\in[1,2]}|\varDelta^L_{2^{-k}t,1}(\phi_{2^{m_0-J-k}}*f)|^r)(x) \nonumber\\
&\lesssim 2^{Jn(1-r)}\HLmax_n(\esssup_{t\in[1,2]}|\varDelta^L_{2^{-k}t,1}(\phi_{2^{m_0-J-k}}*f)|^r)(x),
\end{align}
if we pick $N'>\frac{n}{r}$. From this inequality and the mapping property of the Hardy-Littlewood maximal function, we deduce the estimate
\begin{align}\label{eq578}
&\|\esssup_{k\in\bbbz}2^{ks}|\iFT_n[\rho_1(2^{m_0-k}\xi)\FT_n f](\cdot)|\|_{L^p(\bbbr^n)} \nonumber\\
&\lesssim 2^{Jn(\frac{1}{r}-1)}
\|\esssup_{k\in\bbbz}\esssup_{t\in[1,2]}2^{ks}|\varDelta^L_{2^{-k}t,1}(\phi_{2^{m_0-J-k}}*f)(\cdot)|\|_{L^p(\bbbr^n)}.
\end{align}
Assuming the validity of decomposition for now, then (\ref{eq578}) can be estimated from above by the sum of
\begin{align}\label{eq579}
&2^{Jn(\frac{1}{r}-1)}
\|\esssup_{k\in\bbbz}\esssup_{t\in[1,2]}2^{ks}|\varDelta^L_{2^{-k}t,1}f(\cdot)|\|_{L^p(\bbbr^n)} \nonumber\\
&\sim2^{Jn(\frac{1}{r}-1)} \|\esssup_{t>0}\frac{|\varDelta^L_{t,1}f(\cdot)|}{t^s}\|_{L^p(\bbbr^n)}<\infty,
\end{align}
and
\begin{equation}\label{eq580}
2^{Jn(\frac{1}{r}-1)}
\|\esssup_{k\in\bbbz}\esssup_{t\in[1,2]}2^{ks}|\varDelta^L_{2^{-k}t,1}(\sum_{l=J+1}^{\infty}f_{k+l-m_0})(\cdot)|\|_{L^p(\bbbr^n)}.
\end{equation}
To estimate (\ref{eq580}), we use (\ref{eq480}), (\ref{eq481}), Lebesgue's differentiation theorem and the fact that $k+l-m_0>k+J-m_0>k$ to obtain for $k\in\bbbz$, $l>J$ and $t\in[1,2]$,
\begin{equation}\label{eq581}
|\varDelta^L_{2^{-k}t,1}f_{k+l-m_0}(x)|\lesssim 2^{\frac{l}{r}}\HLmax_1(|f_{k+l-m_0}(\cdot,x_1')|^r)(x_1)^{\frac{1}{r}},
\end{equation}
where the constant is independent of $l,k,t$. Therefore we have
\begin{align}\label{eq582}
&2^{Jn(\frac{1}{r}-1)}
\|\esssup_{k\in\bbbz}\esssup_{t\in[1,2]}2^{ks}\sum_{l=J+1}^{\infty}|\varDelta^L_{2^{-k}t,1}f_{k+l-m_0}(\cdot)|\|_{L^p(\bbbr^n)} \nonumber\\
&=2^{Jn(\frac{1}{r}-1)}
\|\esssup_{k\in\bbbz}\esssup_{t\in[1,2]}2^{ks}\sum_{l=J+1}^{\infty}2^{-l\varepsilon}\cdot 2^{l\varepsilon} |\varDelta^L_{2^{-k}t,1}f_{k+l-m_0}(\cdot)|\|_{L^p(\bbbr^n)} \nonumber\\
&\lesssim 2^{Jn(\frac{1}{r}-1)-J\varepsilon}
\|\esssup_{\substack{l>J\\k\in\bbbz}}\esssup_{t\in[1,2]}
2^{ks+l\varepsilon}|\varDelta^L_{2^{-k}t,1}f_{k+l-m_0}(\cdot)|\|_{L^p(\bbbr^n)} \nonumber\\
&\lesssim 2^{Jn(\frac{1}{r}-1)-J\varepsilon}
\|\esssup_{\substack{l>J\\k\in\bbbz}}2^{l(\varepsilon+\frac{1}{r}-s)}\cdot 2^{(k+l-m_0)s} \nonumber\\
&\quad\cdot\HLmax_1(|f_{k+l-m_0}(\cdot,x_1')|^r)(x_1)^{\frac{1}{r}}\|_{L^p(\bbbr^n)}.
\end{align}
The assumption $\sigma_p+\frac{1}{p}<s<\infty$ indicates
\begin{equation}\label{eq583}
\varepsilon+\frac{1}{r}<n(\frac{1}{r}-1)+\frac{1}{r}<s
\end{equation}
if we pick $\varepsilon>0$ to be sufficiently small and $r$ to be sufficiently close to $p$ when $0<r<p\leq 1$. Thus by the mapping property of the $1$-dimensional Hardy-Littlewood maximal function, we can continue from (\ref{eq582}) and obtain the following estimate
\begin{align}\label{eq584}
&2^{Jn(\frac{1}{r}-1)}
\|\esssup_{k\in\bbbz}\esssup_{t\in[1,2]}2^{ks}\sum_{l=J+1}^{\infty}|\varDelta^L_{2^{-k}t,1}f_{k+l-m_0}(\cdot)|\|_{L^p(\bbbr^n)} \nonumber\\
&\lesssim 2^{J[n(\frac{1}{r}-1)+\frac{1}{r}-s]}
(\int_{\bbbr^{n-1}}\int_{\bbbr}\HLmax_1(\esssup_{k\in\bbbz}2^{ksr}|f_k(\cdot,x_1')|^r)(x_1)^{\frac{p}{r}}dx_1 dx_1')^{\frac{1}{p}} \nonumber\\
&\lesssim 2^{J[n(\frac{1}{r}-1)+\frac{1}{r}-s]}\|f\|_{\Fspinf}.
\end{align}
The assumption $f\in\Fspinf$ implies $\sum_{l=J+1}^{\infty}|\varDelta^L_{2^{-k}t,1}f_{k+l-m_0}(x)|<\infty$ for every $k\in\bbbz$ and almost every $t\in[1,2],x\in\bbbr^n$. From this implication, (\ref{eq514}) and (\ref{eq579}), we deduce
\begin{equation}\label{eq585}
\varDelta^L_{2^{-k}t,1}(\phi_{2^{m_0-J-k}}*f)(x)=\varDelta^L_{2^{-k}t,1}f(x)-
\sum_{l=J+1}^{\infty}\varDelta^L_{2^{-k}t,1}f_{k+l-m_0}(x)
\end{equation}
in the sense of $\Sw'(\bbbr^n)$ for every $k\in\bbbz$ and almost every $t\in[1,2],x\in\bbbr^n$ when $0<p\leq 1$, $q=\infty$ and $\sigma_p+\frac{1}{p}<s<\infty$, furthermore estimating (\ref{eq578}) from above by the sum of (\ref{eq579}) and (\ref{eq580}) is justified, moreover (\ref{eq580}) can be estimated from above by the first line of (\ref{eq582}). We have obtained
\begin{align}\label{eq587}
&\|\esssup_{k\in\bbbz}2^{ks}|\iFT_n[\rho_1(2^{m_0-k}\xi)\FT_n f](\cdot)|\|_{L^p(\bbbr^n)} \nonumber\\
&\leq C_1'2^{Jn(\frac{1}{r}-1)}\|\esssup_{t>0}\frac{|\varDelta^L_{t,1}f(\cdot)|}{t^s}\|_{L^p(\bbbr^n)}
      +C_1''2^{J[n(\frac{1}{r}-1)+\frac{1}{r}-s]}\|f\|_{\Fspinf}.
\end{align}
Inequality (\ref{eq587}) is still true if we replace $\rho_1,\varDelta^L_{t,1},C_1',C_1''$ by $\rho_j,\varDelta^L_{t,j},C_j',C_j''$ respectively for $j=2,\cdots,n$. Inserting these inequalities back into (\ref{eq572}) yields
\begin{align}\label{eq588}
\|f\|_{\Fspinf}
&\leq C'2^{Jn(\frac{1}{r}-1)}\sum_{j=1}^n\|\esssup_{t>0}\frac{|\varDelta^L_{t,j}f(\cdot)|}{t^s}\|_{L^p(\bbbr^n)} \nonumber\\
&\quad+C''2^{J[n(\frac{1}{r}-1)+\frac{1}{r}-s]}\|f\|_{\Fspinf}.
\end{align}
Due to (\ref{eq583}), we can pick $J$ to be sufficiently large so that the coefficient $$C''2^{J[n(\frac{1}{r}-1)+\frac{1}{r}-s]}<\frac{1}{2},$$ and shift the second term on the right side of (\ref{eq588}) to its left side to complete the proof of Theorem \ref{theorem6} (iv). Now we conclude the proof of Theorem \ref{theorem6}.
\end{proof}

\section{Besov-Lipschitz space and iterated differences along coordinate axes}\label{proof.of.theorem7}
\begin{proof}[Proof Of Theorem \ref{theorem7}]
To prove Theorem \ref{theorem7} (i), it suffices to prove inequality (\ref{eq600}) for $j=1$, and the other cases for $j=2,\cdots,n$ can be proven in the same way. We begin with estimating the following two terms
\begin{align}\label{eq604}
&(\sum_{k\in\bbbz}2^{k(sq+1)}\int_{2^{-k}}^{2^{1-k}}\|\sum_{l\leq k}|\varDelta^L_{t,1}f_l|\|_{L^p(\bbbr^n)}^q dt)^{\frac{1}{q}} \nonumber\\
&\lesssim(\sum_{k\in\bbbz}2^{ksq}\esssup_{t\in[1,2]}\|\sum_{l\leq k}|\varDelta^L_{2^{-k}t,1}f_l|\|_{L^p(\bbbr^n)}^q)^{\frac{1}{q}},
\end{align}
and
\begin{align}\label{eq605}
&(\sum_{k\in\bbbz}2^{k(sq+1)}\int_{2^{-k}}^{2^{1-k}}\|\sum_{l>k}|\varDelta^L_{t,1}f_l|\|_{L^p(\bbbr^n)}^q dt)^{\frac{1}{q}} \nonumber\\
&\lesssim(\sum_{k\in\bbbz}2^{ksq}\esssup_{t\in[1,2]}\|\sum_{l>k}|\varDelta^L_{2^{-k}t,1}f_l|\|_{L^p(\bbbr^n)}^q)^{\frac{1}{q}}.
\end{align}
We pick $0<\varepsilon<\min\{s,L-s\}$. If $1\leq p\leq\infty$, we use Minkowski's inequality for $\|\cdot\|_{L^p(\bbbr^n)}$-norm and obtain
\begin{align}\label{eq606}
&\|\sum_{l\leq k}|\varDelta^L_{2^{-k}t,1}f_l|\|_{L^p(\bbbr^n)}^q \nonumber\\
&\lesssim(\sum_{l\leq k}2^{l\varepsilon}\cdot 2^{-l\varepsilon}\|\varDelta^L_{2^{-k}t,1}f_l\|_{L^p(\bbbr^n)})^q \nonumber\\
&\lesssim(\sum_{l\leq k}2^{l\varepsilon})^q\esssup_{j\leq k}2^{-jq\varepsilon}\|\varDelta^L_{2^{-k}t,1}f_j\|_{L^p(\bbbr^n)}^q \nonumber\\
&\lesssim 2^{kq\varepsilon}\sum_{l\leq k}2^{-lq\varepsilon}\|\varDelta^L_{2^{-k}t,1}f_l\|_{L^p(\bbbr^n)}^q,
\end{align}
and
\begin{align}\label{eq607}
&\|\sum_{l>k}|\varDelta^L_{2^{-k}t,1}f_l|\|_{L^p(\bbbr^n)}^q \nonumber\\
&\lesssim(\sum_{l>k}2^{-l\varepsilon}\cdot 2^{l\varepsilon}\|\varDelta^L_{2^{-k}t,1}f_l\|_{L^p(\bbbr^n)})^q \nonumber\\
&\lesssim(\sum_{l>k}2^{-l\varepsilon})^q\esssup_{j>k}2^{jq\varepsilon}\|\varDelta^L_{2^{-k}t,1}f_j\|_{L^p(\bbbr^n)}^q \nonumber\\
&\lesssim 2^{-kq\varepsilon}\sum_{l>k}2^{lq\varepsilon}\|\varDelta^L_{2^{-k}t,1}f_l\|_{L^p(\bbbr^n)}^q.
\end{align}
If $0<p<1$, then we have
\begin{align}\label{eq608}
&\|\sum_{l\leq k}|\varDelta^L_{2^{-k}t,1}f_l|\|_{L^p(\bbbr^n)}^q \nonumber\\
&\lesssim(\sum_{l\leq k}2^{lp\varepsilon}\cdot 2^{-lp\varepsilon}\|\varDelta^L_{2^{-k}t,1}f_l\|_{L^p(\bbbr^n)}^p)^{\frac{q}{p}} \nonumber\\
&\lesssim 2^{kq\varepsilon}\sum_{l\leq k}2^{-lq\varepsilon}\|\varDelta^L_{2^{-k}t,1}f_l\|_{L^p(\bbbr^n)}^q,
\end{align}
and
\begin{align}\label{eq609}
&\|\sum_{l>k}|\varDelta^L_{2^{-k}t,1}f_l|\|_{L^p(\bbbr^n)}^q \nonumber\\
&\lesssim(\sum_{l>k}2^{-lp\varepsilon}\cdot 2^{lp\varepsilon}\|\varDelta^L_{2^{-k}t,1}f_l\|_{L^p(\bbbr^n)}^p)^{\frac{q}{p}} \nonumber\\
&\lesssim 2^{-kq\varepsilon}\sum_{l>k}2^{lq\varepsilon}\|\varDelta^L_{2^{-k}t,1}f_l\|_{L^p(\bbbr^n)}^q.
\end{align}
When $t\in[1,2]$ and $l\leq k$, we use estimate (\ref{eq478}) with $0<r<p$, the $1$-dimensional version of Lemma \ref{lemma4} and the mapping property of the $1$-dimensional Hardy-Littlewood maximal function to obtain
\begin{equation}\label{eq610}
\|\varDelta^L_{2^{-k}t,1}f_l\|_{L^p(\bbbr^n)}\lesssim 2^{(l-k)L}\|f_l\|_{L^p(\bbbr^n)}\text{ for }0<p\leq\infty,
\end{equation}
where the constant is independent of $t,l,k$. Inserting (\ref{eq606}), (\ref{eq608}) and (\ref{eq610}) into (\ref{eq604}) yields
\begin{equation}\label{eq611}
(\ref{eq604})\lesssim
(\sum_{l\in\bbbz}\sum_{k\geq l}2^{kq(s+\varepsilon-L)}\cdot 2^{lq(L-\varepsilon)}\|f_l\|_{L^p(\bbbr^n)}^q)^{\frac{1}{q}}
\lesssim\|f\|_{\Bspq}.
\end{equation}
We can also use estimate (\ref{eq480}) and proper change of variable to obtain
\begin{equation}\label{eq612}
\|\varDelta^L_{2^{-k}t,1}f_l\|_{L^p(\bbbr^n)}\lesssim\sum_{m=0}^L\|f_l(\cdot+2^{-k}mte_1)\|_{L^p(\bbbr^n)}
\lesssim\|f_l\|_{L^p(\bbbr^n)},
\end{equation}
where constants are independent of $t,k,l$, and (\ref{eq612}) is true for all $0<p\leq\infty$. Inserting (\ref{eq607}), (\ref{eq609}) and (\ref{eq612}) into (\ref{eq605}) yields
\begin{equation}\label{eq613}
(\ref{eq605})\lesssim
(\sum_{l\in\bbbz}\sum_{k<l}2^{kq(s-\varepsilon)}\cdot 2^{lq\varepsilon}\|f_l\|_{L^p(\bbbr^n)}^q)^{\frac{1}{q}}
\lesssim\|f\|_{\Bspq}.
\end{equation}
Combining (\ref{eq604}), (\ref{eq605}), (\ref{eq611}) and (\ref{eq613}), we have proven
\begin{align}\label{eq614}
&(\int_0^{\infty}t^{-sq}\|\sum_{l\in\bbbz}\!|\varDelta^L_{t,1}f_l|\|_{L^p(\bbbr^n)}^q\frac{dt}{t})^{\frac{1}{q}} \nonumber\\
&\lesssim(\sum_{k\in\bbbz}2^{ksq}
\esssup_{t\in[1,2]}\|\sum_{l\in\bbbz}|\varDelta^L_{2^{-k}t,1}f_l|\|_{L^p(\bbbr^n)}^q)^{\frac{1}{q}}
\lesssim\|f\|_{\Bspq},
\end{align}
if $0<p\leq\infty$, $0<q<\infty$ and $0<s<L$. The assumption $f\in\Bspq$ implies $\sum_{l\in\bbbz}|\varDelta^L_{t,1}f_l(x)|<\infty$ for almost every $t\in(0,\infty)$ and $x\in\bbbr^n$. In conjunction with (\ref{eq513}), we have reached the conclusion that
\begin{equation}\label{eq615}
\varDelta^L_{t,1}f=\sum_{l\in\bbbz}\varDelta^L_{t,1}f_l(x)
\end{equation}
in the sense of $\Sw'(\bbbr^n)/\mathscr{P}(\bbbr^n)$ for almost every $t\in(0,\infty)$ and $x\in\bbbr^n$ when $0<p\leq\infty$, $0<q<\infty$ and $0<s<L$. Hence we also have the estimate
\begin{equation}\label{eq616}
(\int_0^{\infty}t^{-sq}\|\varDelta^L_{t,1}f\|_{L^p(\bbbr^n)}^q\frac{dt}{t})^{\frac{1}{q}}\lesssim
(\int_0^{\infty}t^{-sq}\|\sum_{l\in\bbbz}|\varDelta^L_{t,1}f_l|\|_{L^p(\bbbr^n)}^q\frac{dt}{t})^{\frac{1}{q}}.
\end{equation}
Inequalities (\ref{eq614}) and (\ref{eq616}) conclude the proof of Theorem \ref{theorem7} (i).

Now we prove the first and the second parts of Theorem \ref{theorem7} (ii). We can assume the right side of (\ref{eq601}) is finite and the left side of (\ref{eq601}) is positive, otherwise inequality (\ref{eq601}) will be trivial. In the definition of the Peetre-Fefferman-Stein maximal function, we pick the number $r$ so that $0<r<p$. We also use the positive integer $m_0$ given in (\ref{eq417}) and $\{\rho_j\}_{j=1}^n$ is the partition of unity given in (\ref{eq589}) and (\ref{eq590}). Then we have
\begin{align}\label{eq617}
\|f\|_{\Bspq}
&=(\sum_{k\in\bbbz}2^{(k-m_0)sq}\|\psi_{2^{m_0-k}}*f\|_{L^p(\bbbr^n)}^q)^{\frac{1}{q}} \nonumber\\
&\lesssim\sum_{j=1}^n(\sum_{k\in\bbbz}2^{ksq}
         \|\iFT_n[\FT_n\psi(2^{m_0-k}\xi)\rho_j(2^{m_0-k}\xi)\FT_n f]\|_{L^p(\bbbr^n)}^q)^{\frac{1}{q}}.
\end{align}
By Remark \ref{remark3} and Remark \ref{remark12}, we can obtain the following pointwise estimate
\begin{align}\label{eq618}
&\iFT_n[\FT_n\psi(2^{m_0-k}\xi)\rho_j(2^{m_0-k}\xi)\FT_n f](x) \nonumber\\
&\lesssim\PFSmax_n\{\iFT_n[\FT_n\psi(2^{m_0-k}\xi)\rho_j(2^{m_0-k}\xi)\FT_n f]\}(x) \nonumber\\
&\lesssim\PFSmax_n\{\iFT_n[\rho_j(2^{m_0-k}\xi)\FT_n f]\}(x) \nonumber\\
&\lesssim\HLmax_n(|\iFT_n[\rho_j(2^{m_0-k}\xi)\FT_n f]|^r)(x)^{\frac{1}{r}},
\end{align}
for $j\in\{1,\cdots,n\}$ and $x\in\bbbr^n$. Inserting (\ref{eq618}) into (\ref{eq617}) and invoking the mapping property of Hardy-Littlewood maximal function yield
\begin{equation}\label{eq619}
\|f\|_{\Bspq}\lesssim\sum_{j=1}^n\|\{2^{ks}\iFT_n[\rho_j(2^{m_0-k}\xi)\FT_n f](\cdot)\}_{k\in\bbbz}\|_{l^q(L^p)}.
\end{equation}
We estimate the term with $j=1$ and estimates for other terms with $j=2,\cdots,n$ can be obtained in the same way. If $1<p\leq\infty$, $1\leq q<\infty$ and $s\in\bbbr$, we use (\ref{eq489}), (\ref{eq491}), (\ref{eq492}), the mapping property of Hardy-Littlewood maximal function, Minkowski's inequality for $\|\cdot\|_{L^p(\bbbr^n)}$-norm and H\"{o}lder's inequality for $1\leq q<\infty$ in a sequence and then we have
\begin{align}\label{eq621}
&\|\{2^{ks}\iFT_n[\rho_1(2^{m_0-k}\xi)\FT_n f](\cdot)\}_{k\in\bbbz}\|_{l^q(L^p)} \nonumber\\
&\lesssim\|\{2^{ks}\HLmax_n(\int_{[1,2]}|\varDelta^L_{2^{-k}t,1}f(\cdot)|\frac{dt}{t})(\cdot)\}_{k\in\bbbz}\|_{l^q(L^p)} \nonumber\\
&\lesssim(\sum_{k\in\bbbz}2^{ksq}(\int_{[1,2]}\|\varDelta^L_{2^{-k}t,1}f\|_{L^p(\bbbr^n)}\frac{dt}{t})^q)^{\frac{1}{q}} \nonumber\\
&\lesssim(\sum_{k\in\bbbz}2^{ksq}\int_{[1,2]}\|\varDelta^L_{2^{-k}t,1}f\|_{L^p(\bbbr^n)}^q\frac{dt}{t})^{\frac{1}{q}} \nonumber\\
&\lesssim(\int_0^{\infty}t^{-sq}\|\varDelta^L_{t,1}f\|_{L^p(\bbbr^n)}^q\frac{dt}{t})^{\frac{1}{q}}.
\end{align}
Inequality (\ref{eq621}) is also true if we replace $\rho_1,\varDelta^L_{2^{-k}t,1},\varDelta^L_{t,1}$ by $\rho_j,\varDelta^L_{2^{-k}t,j},\varDelta^L_{t,j}$ respectively for $j=2,\cdots,n$. Inserting these inequalities into (\ref{eq619}) proves inequality (\ref{eq601}) when $1<p\leq\infty$, $1\leq q<\infty$ and $s\in\bbbr$. If $1<p\leq\infty$, $0<q<1$ and $0<s<L$, we can still obtain the third line of (\ref{eq621}). We continue from there and estimate each term in the summation of the third line of (\ref{eq621}) as below
\begin{align}\label{eq622}
&2^{ksq}(\int_{[1,2]}\|\varDelta^L_{2^{-k}t,1}f\|_{L^p(\bbbr^n)}\frac{dt}{t})^q \nonumber\\
&=(\int_{[1,2]}2^{ksq}\|\varDelta^L_{2^{-k}t,1}f\|_{L^p(\bbbr^n)}^q\cdot
2^{ks(1-q)}\|\varDelta^L_{2^{-k}t,1}f\|_{L^p(\bbbr^n)}^{1-q}\frac{dt}{t})^q \nonumber\\
&\lesssim(\int_{[1,2]}2^{ksq}\|\varDelta^L_{2^{-k}t,1}f\|_{L^p(\bbbr^n)}^q\frac{dt}{t})^q\cdot
2^{ks(1-q)q}\esssup_{t\in[1,2]}\|\varDelta^L_{2^{-k}t,1}f\|_{L^p(\bbbr^n)}^{(1-q)q},
\end{align}
and then by using H\"{o}lder's inequality with conjugates $\frac{1}{q}$ and $\frac{1}{1-q}$, we can estimate the third line of (\ref{eq621}) from above by the product of the following two terms,
\begin{equation}\label{eq623}
\sum_{k\in\bbbz}\int_{[1,2]}2^{ksq}\|\varDelta^L_{2^{-k}t,1}f\|_{L^p(\bbbr^n)}^q\frac{dt}{t}\lesssim
\int_0^{\infty}t^{-sq}\|\varDelta^L_{t,1}f\|_{L^p(\bbbr^n)}^q\frac{dt}{t}
\end{equation}
and
\begin{equation}\label{eq624}
(\sum_{k\in\bbbz}2^{ksq}\esssup_{t\in[1,2]}\|\varDelta^L_{2^{-k}t,1}f\|_{L^p(\bbbr^n)}^q)^{\frac{1}{q}-1}.
\end{equation}
Furthermore (\ref{eq624}) can be estimated from above by $\|f\|_{\Bspq}^{1-q}$ if we apply (\ref{eq614}) and the argument for justifying the decomposition afterward. Combining these estimates together, we have shown
\begin{equation}\label{eq625}
\|\{2^{ks}\iFT_n[\rho_1(2^{m_0-k}\xi)\FT_n f](\cdot)\}_{k\in\bbbz}\|_{l^q(L^p)}\lesssim
\int_0^{\infty}t^{-sq}\|\varDelta^L_{t,1}f\|_{L^p(\bbbr^n)}^q\frac{dt}{t}\cdot\|f\|_{\Bspq}^{1-q}.
\end{equation}
Inequality (\ref{eq625}) is still true if we replace $\rho_1,\varDelta^L_{t,1}$ by $\rho_j,\varDelta^L_{t,j}$ respectively for $j=2,\cdots,n$. Inserting these inequalities into (\ref{eq619}) yields
\begin{equation}\label{eq626}
\|f\|_{\Bspq}\lesssim\sum_{j=1}^n
\int_0^{\infty}t^{-sq}\|\varDelta^L_{t,j}f\|_{L^p(\bbbr^n)}^q\frac{dt}{t}\cdot\|f\|_{\Bspq}^{1-q},
\end{equation}
when $1<p\leq\infty$, $0<q<1$ and $0<s<L$. Then inequality (\ref{eq626}) indicates the desired inequality (\ref{eq601}). Next, we come to the proof of the third part of Theorem \ref{theorem7} (ii). We prove (\ref{eq601}) when $0<p\leq 1$, $0<q<\infty$ and $\sigma_p<s<L$, we pick $0<r<p$ and estimate the term $\|\{2^{ks}\iFT_n[\rho_1(2^{m_0-k}\xi)\FT_n f](\cdot)\}_{k\in\bbbz}\|_{l^q(L^p)}$. We still use the sufficiently large positive integer $m_0$ given in (\ref{eq417}), the function $\phi$ satisfying conditions (\ref{eq346}), (\ref{eq422}), (\ref{eq423}), and $J>m_0$ is a large positive integer whose value will be determined later. From (\ref{eq497}), the mapping property of Hardy-Littlewood maximal function and Minkowski's inequality for $\|\cdot\|_{L^{p/r}(\bbbr^n)}$-norm, we can deduce the following
\begin{align}\label{eq627}
&\|\{2^{ks}\iFT_n[\rho_1(2^{m_0-k}\xi)\FT_n f](\cdot)\}_{k\in\bbbz}\|_{l^q(L^p)} \nonumber\\
&\lesssim 2^{Jn(\frac{1}{r}-1)}(\sum_{k\in\bbbz}2^{ksq}
          \|\int_1^2|\varDelta^L_{2^{-k}t,1}(\phi_{2^{m_0-J-k}}*f)(\cdot)|^r\frac{dt}{t}
          \|_{L^{p/r}(\bbbr^n)}^{q/r})^{\frac{1}{q}} \nonumber\\
&\lesssim 2^{Jn(\frac{1}{r}-1)}(\sum_{k\in\bbbz}2^{ksq}
          (\int_1^2\|\varDelta^L_{2^{-k}t,1}(\phi_{2^{m_0-J-k}}*f)\|_{L^p(\bbbr^n)}^r\frac{dt}{t})^{q/r})^{\frac{1}{q}}.
\end{align}
If furthermore $q$ and $r$ satisfy $q\geq r$, then by using H\"{o}lder's inequality we have
\begin{equation}\label{eq628}
(\int_1^2\|\varDelta^L_{2^{-k}t,1}(\phi_{2^{m_0-J-k}}*f)\|_{L^p(\bbbr^n)}^r\frac{dt}{t})^{q/r}\lesssim
\int_1^2\|\varDelta^L_{2^{-k}t,1}(\phi_{2^{m_0-J-k}}*f)\|_{L^p(\bbbr^n)}^q\frac{dt}{t}.
\end{equation}
Hence we can estimate (\ref{eq627}) from above by
\begin{equation}\label{eq629}
2^{Jn(\frac{1}{r}-1)}(\sum_{k\in\bbbz}2^{ksq}
\int_1^2\|\varDelta^L_{2^{-k}t,1}(\phi_{2^{m_0-J-k}}*f)\|_{L^p(\bbbr^n)}^q\frac{dt}{t})^{\frac{1}{q}}.
\end{equation}
From (\ref{eq500}) and (\ref{eq514}), we can infer
\begin{equation}\label{eq630}
\varDelta^L_{2^{-k}t,1}(\phi_{2^{m_0-J-k}}*f)(x)=\varDelta^L_{2^{-k}t,1}f(x)
-\sum_{j=J+1}^{\infty}\varDelta^L_{2^{-k}t,1}f_{k+j-m_0}(x)
\end{equation}
in the sense of $\Sw'(\bbbr^n)$. Assuming the validity of decomposition for now, we can estimate (\ref{eq629}) from above by the sum of the following two terms,
\begin{align}\label{eq631}
&2^{Jn(\frac{1}{r}-1)}(\sum_{k\in\bbbz}2^{ksq}
\int_1^2\|\varDelta^L_{2^{-k}t,1}f\|_{L^p(\bbbr^n)}^q\frac{dt}{t})^{\frac{1}{q}} \nonumber\\
&\sim 2^{Jn(\frac{1}{r}-1)}(\int_0^{\infty}t^{-sq}\|\varDelta^L_{t,1}f\|_{L^p(\bbbr^n)}^q\frac{dt}{t})^{\frac{1}{q}}<\infty,
\end{align}
and
\begin{equation}\label{eq632}
2^{Jn(\frac{1}{r}-1)}(\sum_{k\in\bbbz}2^{ksq}
\int_1^2\|\varDelta^L_{2^{-k}t,1}(\sum_{j=J+1}^{\infty}f_{k+j-m_0})\|_{L^p(\bbbr^n)}^q\frac{dt}{t})^{\frac{1}{q}}.
\end{equation}
To estimate (\ref{eq632}), we use (\ref{eq612}) and the condition $0<p\leq 1$ and begin with the following
\begin{align}\label{eq634}
&2^{Jn(\frac{1}{r}-1)}(\sum_{k\in\bbbz}2^{ksq}
\int_1^2\|\sum_{j=J+1}^{\infty}|\varDelta^L_{2^{-k}t,1}f_{k+j-m_0}|\|_{L^p(\bbbr^n)}^q\frac{dt}{t})^{\frac{1}{q}} \nonumber\\
&\lesssim 2^{Jn(\frac{1}{r}-1)}(\sum_{k\in\bbbz}2^{ksq}(\sum_{j=J+1}^{\infty}
          \|f_{k+j-m_0}\|_{L^p(\bbbr^n)}^p)^{\frac{q}{p}})^{\frac{1}{q}}.
\end{align}
In case $q\leq p$, since $0<\sigma_p<s$, we have
\begin{align}\label{eq635}
(\ref{eq634})
&\lesssim 2^{Jn(\frac{1}{r}-1)}(\sum_{k\in\bbbz}2^{ksq}\sum_{j=J+1}^{\infty}\|f_{k+j-m_0}\|_{L^p(\bbbr^n)}^q)^{\frac{1}{q}} \nonumber\\
&\lesssim 2^{Jn(\frac{1}{r}-1)}(\sum_{j=J+1}^{\infty}2^{(m_0-j)sq}
          \sum_{k\in\bbbz}2^{(k+j-m_0)sq}\|f_{k+j-m_0}\|_{L^p(\bbbr^n)}^q)^{\frac{1}{q}} \nonumber\\
&\lesssim 2^{J[n(\frac{1}{r}-1)-s]}\|f\|_{\Bspq}.
\end{align}
In case $p<q$, we use $0<\varepsilon<\min\{s,L-s\}$ and estimate as follows
\begin{align}\label{eq636}
(\ref{eq634})
&=2^{Jn(\frac{1}{r}-1)}(\sum_{k\in\bbbz}2^{ksq}(\sum_{j=J+1}^{\infty}2^{-jp\varepsilon}\cdot 2^{jp\varepsilon}
   \|f_{k+j-m_0}\|_{L^p(\bbbr^n)}^p)^{\frac{q}{p}})^{\frac{1}{q}} \nonumber\\
&\lesssim 2^{Jn(\frac{1}{r}-1)}(\sum_{k\in\bbbz}2^{ksq}(\sum_{j=J+1}^{\infty}2^{-jp\varepsilon})^{\frac{q}{p}}
          \cdot\esssup_{l>J}2^{lq\varepsilon}\|f_{k+l-m_0}\|_{L^p(\bbbr^n)}^q)^{\frac{1}{q}} \nonumber\\
&\lesssim 2^{J[n(\frac{1}{r}-1)-\varepsilon]}
          (\sum_{k\in\bbbz}2^{ksq}\sum_{l=J+1}^{\infty}2^{lq\varepsilon}\|f_{k+l-m_0}\|_{L^p(\bbbr^n)}^q)^{\frac{1}{q}} \nonumber\\
&\lesssim 2^{J[n(\frac{1}{r}-1)-s]}\|f\|_{\Bspq}.
\end{align}
Combining (\ref{eq634}), (\ref{eq635}), (\ref{eq636}) and the assumption $f\in\Bspq$ altogether, we find $\sum_{j=J+1}^{\infty}|\varDelta^L_{2^{-k}t,1}f_{k+j-m_0}(x)|<\infty$ for every $k\in\bbbz$ and almost every $t\in[1,2],x\in\bbbr^n$. In conjunction with (\ref{eq630}) and (\ref{eq631}), we have proven (\ref{eq630}) is true not only in the sense of $\Sw'(\bbbr^n)$ but also for every $k\in\bbbz$ and almost every $t\in[1,2],x\in\bbbr^n$ when $0<r<p\leq 1$, $r\leq q<\infty$ and $\sigma_p<s<L$. Furthermore estimating (\ref{eq629}) from above by the sum of (\ref{eq631}) and (\ref{eq632}) is justified, moreover (\ref{eq632}) can be estimated from above by (\ref{eq634}) and hence by the term $2^{J[n(\frac{1}{r}-1)-s]}\|f\|_{\Bspq}$. Recall (\ref{eq627}) and (\ref{eq629}), then we have obtained
\begin{align}\label{eq637}
&\|\{2^{ks}\iFT_n[\rho_1(2^{m_0-k}\xi)\FT_n f](\cdot)\}_{k\in\bbbz}\|_{l^q(L^p)} \nonumber\\
&\leq C_1'2^{Jn(\frac{1}{r}-1)}(\int_0^{\infty}\!\!\!\!t^{-sq}\|\varDelta^L_{t,1}f\|_{L^p(\bbbr^n)}^q\frac{dt}{t})^{\frac{1}{q}}
+C_1''2^{J[n(\frac{1}{r}-1)-s]}\|f\|_{\Bspq}.
\end{align}
Inequality (\ref{eq637}) is also true if we replace $\rho_1,\varDelta^L_{t,1},C_1',C_1''$ by $\rho_j,\varDelta^L_{t,j},C_j',C_j''$ respectively for $j=2,\cdots,n$. Inserting these inequalities into (\ref{eq619}) yields
\begin{align}\label{eq638}
\|f\|_{\Bspq}
&\leq C'2^{Jn(\frac{1}{r}-1)}\sum_{j=1}^n
(\int_0^{\infty}t^{-sq}\|\varDelta^L_{t,j}f\|_{L^p(\bbbr^n)}^q\frac{dt}{t})^{\frac{1}{q}} \nonumber\\
&\quad+C''2^{J[n(\frac{1}{r}-1)-s]}\|f\|_{\Bspq}.
\end{align}
The condition $\sigma_p<s<L$ implies $n(\frac{1}{r}-1)-s<0$ if $r$ is sufficiently close to $p$, and hence the coefficient $C''2^{J[n(\frac{1}{r}-1)-s]}$ is less than $\frac{1}{2}$ when $J$ is a sufficiently large positive integer. Shifting the second term on the right side of (\ref{eq638}) to its left side proves the desired inequality (\ref{eq601}) when $0<r<p\leq 1$, $r\leq q<\infty$ and $\sigma_p<s<L$. If $0<q<r<p\leq 1$ and $\sigma_p<s<L$, then we continue from (\ref{eq627}). Applying (\ref{eq630}) and assuming the validity of decomposition for now, then (\ref{eq627}) can be estimated from above by the sum of the following two terms,
\begin{equation}\label{eq639}
2^{Jn(\frac{1}{r}-1)}(\sum_{k\in\bbbz}2^{ksq}
(\int_1^2\|\varDelta^L_{2^{-k}t,1}f\|_{L^p(\bbbr^n)}^r\frac{dt}{t})^{\frac{q}{r}})^{\frac{1}{q}}
\end{equation}
and
\begin{equation}\label{eq640}
2^{Jn(\frac{1}{r}-1)}(\sum_{k\in\bbbz}2^{ksq}
(\int_1^2\|\varDelta^L_{2^{-k}t,1}(\sum_{j=J+1}^{\infty}f_{k+j-m_0})\|_{L^p(\bbbr^n)}^r\frac{dt}{t})^{\frac{q}{r}})^{\frac{1}{q}}.
\end{equation}
To estimate (\ref{eq639}), we first rewrite each term in the summation of (\ref{eq639}) as follows
\begin{align}\label{eq641}
&2^{ksq}(\int_1^2\|\varDelta^L_{2^{-k}t,1}f\|_{L^p(\bbbr^n)}^r\frac{dt}{t})^{\frac{q}{r}} \nonumber\\
&=(\int_1^2 2^{ksq}\|\varDelta^L_{2^{-k}t,1}f\|_{L^p(\bbbr^n)}^q\cdot
   2^{ks(r-q)}\|\varDelta^L_{2^{-k}t,1}f\|_{L^p(\bbbr^n)}^{r-q}\frac{dt}{t})^{\frac{q}{r}} \nonumber\\
&\lesssim(\int_1^2 2^{ksq}\|\varDelta^L_{2^{-k}t,1}f\|_{L^p(\bbbr^n)}^q\frac{dt}{t})^{\frac{q}{r}}\cdot
          (2^{ks}\esssup_{t\in[1,2]}\|\varDelta^L_{2^{-k}t,1}f\|_{L^p(\bbbr^n)})^{(r-q)\frac{q}{r}},
\end{align}
and then we apply H\"{o}lder's inequality with conjugates $\frac{r}{q}$ and $\frac{r}{r-q}$ to obtain
\begin{align}\label{eq642}
(\ref{eq639})
&\lesssim 2^{Jn(\frac{1}{r}-1)}
(\sum_{k\in\bbbz}\int_1^2 2^{ksq}\|\varDelta^L_{2^{-k}t,1}f\|_{L^p(\bbbr^n)}^q\frac{dt}{t})^{\frac{1}{r}} \nonumber\\
&\quad\cdot(\sum_{k\in\bbbz}2^{ksq}\esssup_{t\in[1,2]}\|\varDelta^L_{2^{-k}t,1}f\|_{L^p(\bbbr^n)}^q)^{\frac{1}{q}-\frac{1}{r}} \nonumber\\
&\lesssim 2^{Jn(\frac{1}{r}-1)}
(\int_0^{\infty}t^{-sq}\|\varDelta^L_{t,1}f\|_{L^p(\bbbr^n)}^q\frac{dt}{t})^{\frac{1}{r}}\cdot\|f\|_{\Bspq}^{1-\frac{q}{r}}
<\infty,
\end{align}
where in (\ref{eq642}) we used estimate (\ref{eq614}) and the argument afterward to justify the decomposition. And (\ref{eq614}) requires $0<s<L$. To estimate (\ref{eq640}), we use (\ref{eq612}) and the condition $0<q<r<p\leq 1$ to get
\begin{equation}\label{eq643}
\|\sum_{j=J+1}^{\infty}|\varDelta^L_{2^{-k}t,1}f_{k+j-m_0}|\|_{L^p(\bbbr^n)}^r\lesssim
\sum_{j=J+1}^{\infty}\|f_{k+j-m_0}\|_{L^p(\bbbr^n)}^r,
\end{equation}
and then we can insert (\ref{eq643}) into the following term and obtain
\begin{align}\label{eq644}
&2^{Jn(\frac{1}{r}-1)}(\sum_{k\in\bbbz}2^{ksq}(\int_1^2\|\sum_{j=J+1}^{\infty}|\varDelta^L_{2^{-k}t,1}f_{k+j-m_0}|
\|_{L^p(\bbbr^n)}^r\frac{dt}{t})^{\frac{q}{r}})^{\frac{1}{q}} \nonumber\\
&\lesssim 2^{Jn(\frac{1}{r}-1)}(\sum_{k\in\bbbz}2^{ksq}(\sum_{j=J+1}^{\infty}\|f_{k+j-m_0}\|_{L^p(\bbbr^n)}^r
          )^{\frac{q}{r}})^{\frac{1}{q}} \nonumber\\
&\lesssim 2^{Jn(\frac{1}{r}-1)}(\sum_{j=J+1}^{\infty}2^{(m_0-j)sq}\sum_{k\in\bbbz}2^{(k+j-m_0)sq}
          \|f_{k+j-m_0}\|_{L^p(\bbbr^n)}^q)^{\frac{1}{q}} \nonumber\\
&\lesssim 2^{J[n(\frac{1}{r}-1)-s]}\|f\|_{\Bspq}.
\end{align}
The assumption $f\in\Bspq$ tells us $\sum_{j=J+1}^{\infty}|\varDelta^L_{2^{-k}t,1}f_{k+j-m_0}(x)|<\infty$ for every $k\in\bbbz$ and almost every $t\in[1,2],x\in\bbbr^n$. In conjunction with (\ref{eq642}) and (\ref{eq630}), we have proven (\ref{eq630}) is true not only in the sense of $\Sw'(\bbbr^n)$ but also for every $k\in\bbbz$ and almost every $t\in[1,2],x\in\bbbr^n$ when $0<q<r<p\leq 1$ and $\sigma_p<s<L$. Furthermore estimating (\ref{eq627}) from above by the sum of (\ref{eq639}) and (\ref{eq640}) is justified, moreover (\ref{eq640}) can be estimated from above by the first line of (\ref{eq644}) and hence by the last line of (\ref{eq644}). Combining (\ref{eq627}), (\ref{eq639}), (\ref{eq640}), (\ref{eq642}), and (\ref{eq644}) altogether, we have obtained
\begin{align}\label{eq646}
&\|\{2^{ks}\iFT_n[\rho_1(2^{m_0-k}\xi)\FT_n f](\cdot)\}_{k\in\bbbz}\|_{l^q(L^p)} \nonumber\\
&\leq C_1'2^{Jn(\frac{1}{r}-1)}(\int_0^{\infty}t^{-sq}\|\varDelta^L_{t,1}f\|_{L^p(\bbbr^n)}^q\frac{dt}{t})^{\frac{1}{r}}
\cdot\|f\|_{\Bspq}^{1-\frac{q}{r}} \nonumber\\
&\quad+C_1''2^{J[n(\frac{1}{r}-1)-s]}\|f\|_{\Bspq}.
\end{align}
Inequality (\ref{eq646}) is also true if we replace $\rho_1,\varDelta^L_{t,1},C_1',C_1''$ by $\rho_j,\varDelta^L_{t,j},C_j',C_j''$ respectively for $j=2,\cdots,n$. Inserting these inequalities into (\ref{eq619}) yields
\begin{align}\label{eq647}
\|f\|_{\Bspq}
&\leq C'2^{Jn(\frac{1}{r}-1)}\sum_{j=1}^n
(\int_0^{\infty}\!\!\!\!t^{-sq}\|\varDelta^L_{t,j}f\|_{L^p(\bbbr^n)}^q\frac{dt}{t})^{\frac{1}{r}}
\cdot\|f\|_{\Bspq}^{1-\frac{q}{r}} \nonumber\\
&\quad+C''2^{J[n(\frac{1}{r}-1)-s]}\|f\|_{\Bspq}.
\end{align}
The condition $\sigma_p<s<L$ implies $n(\frac{1}{r}-1)-s<0$ if $r$ is sufficiently close to $p$, and hence the coefficient $C''2^{J[n(\frac{1}{r}-1)-s]}$ is less than $\frac{1}{2}$ when $J$ is a sufficiently large positive integer. And then we can shift the second term on the right side of (\ref{eq647}) to its left side, divide both sides of the resulting inequality by $\|f\|_{\Bspq}^{1-\frac{q}{r}}$ and raise the power to $\frac{r}{q}$, finally we can reach the desired inequality (\ref{eq601}) when $0<q<r<p\leq 1$ and $\sigma_p<s<L$. We have finished the proof for the third part of Theorem \ref{theorem7} (ii).

Now we prove Theorem \ref{theorem7} (iii). We only need to prove inequality (\ref{eq602}) for $j=1$, and the other cases for $j=2,\cdots,n$ can be proved in the same way. We pick $0<\varepsilon<\min\{s,L-s\}$ and begin with estimating the following term
\begin{equation}\label{eq648}
\esssup_{k\in\bbbz}\esssup_{2^{-k}\leq t<2^{1-k}}\!\!\!\!2^{ks}\|\sum_{l\leq k}|\varDelta^L_{t,1}f_l|\|_{L^p(\bbbr^n)}+
\esssup_{k\in\bbbz}\esssup_{2^{-k}\leq t<2^{1-k}}\!\!\!\!2^{ks}\|\sum_{l>k}|\varDelta^L_{t,1}f_l|\|_{L^p(\bbbr^n)}.
\end{equation}
If $1\leq p\leq\infty$, by using Minkowski's inequality and (\ref{eq610}), we can estimate the first term of (\ref{eq648}) from above by
\begin{align}\label{eq649}
&\esssup_{k\in\bbbz}2^{ks}\sum_{l\leq k}2^{l\varepsilon}\cdot 2^{-l\varepsilon}
  \esssup_{2^{-k}\leq t<2^{1-k}}\!\!\!\!\|\varDelta^L_{t,1}f_l\|_{L^p(\bbbr^n)} \nonumber\\
&\lesssim\esssup_{k\in\bbbz}2^{k(s+\varepsilon-L)}\esssup_{l\leq k}2^{l(L-\varepsilon)}\|f_l\|_{L^p(\bbbr^n)}\nonumber\\
&\lesssim\esssup_{l\in\bbbz}\esssup_{k\geq l}2^{k(s+\varepsilon-L)}\cdot 2^{l(L-\varepsilon)}\|f_l\|_{L^p(\bbbr^n)}
         \lesssim\|f\|_{\Bspinf},
\end{align}
and by using Minkowski's inequality and (\ref{eq612}), we can estimate the second term of (\ref{eq648}) from above by
\begin{align}\label{eq650}
&\esssup_{k\in\bbbz}2^{ks}\sum_{l>k}2^{-l\varepsilon}\cdot 2^{l\varepsilon}
  \esssup_{2^{-k}\leq t<2^{1-k}}\!\!\!\!\|\varDelta^L_{t,1}f_l\|_{L^p(\bbbr^n)} \nonumber\\
&\lesssim\esssup_{k\in\bbbz}2^{k(s-\varepsilon)}\esssup_{l>k}2^{l\varepsilon}\|f_l\|_{L^p(\bbbr^n)}\nonumber\\
&\lesssim\esssup_{l\in\bbbz}\esssup_{k<l}2^{k(s-\varepsilon)}\cdot 2^{l\varepsilon}\|f_l\|_{L^p(\bbbr^n)}
         \lesssim\|f\|_{\Bspinf}.
\end{align}
If $0<p<1$, by using (\ref{eq610}) for $l\leq k$, we can estimate the first term of (\ref{eq648}) from above by
\begin{equation}\label{eq651}
\esssup_{k\in\bbbz}2^{ks}\esssup_{2^{-k}\leq t<2^{1-k}}(\sum_{l\leq k}2^{lp\varepsilon}\cdot 2^{-lp\varepsilon}
\|\varDelta^L_{t,1}f_l\|_{L^p(\bbbr^n)}^p)^{\frac{1}{p}}\lesssim\|f\|_{\Bspinf},
\end{equation}
and by using (\ref{eq612}) for $l>k$, we can estimate the second term of (\ref{eq648}) from above by
\begin{equation}\label{eq652}
\esssup_{k\in\bbbz}2^{ks}\esssup_{2^{-k}\leq t<2^{1-k}}(\sum_{l>k}2^{-lp\varepsilon}\cdot 2^{lp\varepsilon}
\|\varDelta^L_{t,1}f_l\|_{L^p(\bbbr^n)}^p)^{\frac{1}{p}}\lesssim\|f\|_{\Bspinf}.
\end{equation}
From (\ref{eq648}), (\ref{eq649}), (\ref{eq650}), (\ref{eq651}) and (\ref{eq652}), we have proven
\begin{equation}\label{eq653}
\esssup_{t>0}t^{-s}\|\sum_{l\in\bbbz}|\varDelta^L_{t,1}f_l|\|_{L^p(\bbbr^n)}\lesssim(\ref{eq648})\lesssim\|f\|_{\Bspinf},
\end{equation}
when $0<p\leq\infty$, $q=\infty$ and $0<s<L$. The assumption $f\in\Bspinf$ implies $\sum_{l\in\bbbz}|\varDelta^L_{t,1}f_l(x)|<\infty$ for almost every $t>0,x\in\bbbr^n$. In conjunction with (\ref{eq513}), we have shown
\begin{equation}\label{eq654}
\varDelta^L_{t,1}f=\sum_{l\in\bbbz}\varDelta^L_{t,1}f_l(x)
\end{equation}
in the sense of $\Sw'(\bbbr^n)/\mathscr{P}(\bbbr^n)$ for almost every $t>0,x\in\bbbr^n$ when $0<p\leq\infty$, $q=\infty$ and $0<s<L$. Therefore we have obtained the inequality
\begin{equation}\label{eq655}
\esssup_{t>0}t^{-s}\|\varDelta^L_{t,1}f\|_{L^p(\bbbr^n)}\lesssim
\esssup_{t>0}t^{-s}\|\sum_{l\in\bbbz}|\varDelta^L_{t,1}f_l|\|_{L^p(\bbbr^n)}.
\end{equation}
Inequalities (\ref{eq653}) and (\ref{eq655}) complete the proof of Theorem \ref{theorem7} (iii).

Finally, we come to the proof of Theorem \ref{theorem7} (iv). We assume the right side of (\ref{eq603}) is finite, otherwise inequality (\ref{eq603}) is trivial. In the definition of the Peetre-Fefferman-Stein maximal function, we pick the number $r$ so that $0<r<p$. We also use the positive integer $m_0$ given in (\ref{eq417}) and $\{\rho_j\}_{j=1}^n$ is the partition of unity given in (\ref{eq589}) and (\ref{eq590}). Then we have
\begin{align}\label{eq656}
\|f\|_{\Bspinf}
&\lesssim\sum_{j=1}^n\esssup_{k\in\bbbz}2^{ks}
          \|\iFT_n[\FT_n\psi(2^{m_0-k}\xi)\rho_j(2^{m_0-k}\xi)\FT_n f]\|_{L^p(\bbbr^n)} \nonumber\\
&\lesssim\sum_{j=1}^n\esssup_{k\in\bbbz}2^{ks}\|\iFT_n[\rho_j(2^{m_0-k}\xi)\FT_n f]\|_{L^p(\bbbr^n)},
\end{align}
where (\ref{eq656}) is a consequence by Remark \ref{remark3}, Remark \ref{remark12} and the mapping property of Hardy-Littlewood maximal function. It suffices to estimate the term for $j=1$ in (\ref{eq656}), and estimates for the other terms in (\ref{eq656}) when $j=2,\cdots,n$ can be obtained in the same way. If $1<p\leq\infty$, $q=\infty$ and $s\in\bbbr$, we can infer from the calculation method displayed in (\ref{eq657}) that
\begin{align}\label{eq658}
&\|\iFT_n[\rho_1(2^{m_0-k}\xi)\FT_n f]\|_{L^p(\bbbr^n)} \nonumber\\
&=\esssup_{t\in[1,2]}\|\iFT_n[\rho_1(2^{m_0-k}\xi)\FT_n f]\|_{L^p(\bbbr^n)}\nonumber\\
&=\esssup_{t\in[1,2]}\|\iFT_n[\frac{\rho_1(2^{m_0-k}\xi)}{(e^{2\pi i\cdot 2^{-k}t\xi_1}-1)^L}]*\varDelta^L_{2^{-k}t,1}f
   \|_{L^p(\bbbr^n)}.
\end{align}
We use (\ref{eq491}), Minkowski's inequality for $\|\cdot\|_{L^p(\bbbr^n)}$-norm, the mapping property of Hardy-Littlewood maximal function and the following estimate
\begin{equation}\label{eq659}
\|\mvint_{A_{k-l}}|\varDelta^L_{2^{-k}t,1}f(\cdot-y)|dy\|_{L^p(\bbbr^n)}\lesssim
\|\HLmax_n(|\varDelta^L_{2^{-k}t,1}f|)(\cdot)\|_{L^p(\bbbr^n)}\lesssim
\|\varDelta^L_{2^{-k}t,1}f\|_{L^p(\bbbr^n)},
\end{equation}
where $A_{k-l}$ denotes the annulus $\{y\in\bbbr^n:2^{l-k}\leq|y|<2^{1+l-k}\}$, and then we can estimate (\ref{eq658}) from above by
\begin{equation}\label{eq660}
\esssup_{t\in[1,2]}(\sum_{l\leq 0}2^{ln}+\sum_{l>0}2^{l(n-N)})\|\varDelta^L_{2^{-k}t,1}f\|_{L^p(\bbbr^n)}\lesssim
\esssup_{t\in[1,2]}\|\varDelta^L_{2^{-k}t,1}f\|_{L^p(\bbbr^n)},
\end{equation}
if we pick $N>n$. Combining (\ref{eq658}) and (\ref{eq660}), we have proven
\begin{equation}\label{eq661}
\esssup_{k\in\bbbz}2^{ks}\|\iFT_n[\rho_1(2^{m_0-k}\xi)\FT_n f]\|_{L^p(\bbbr^n)}\lesssim
%\esssup_{k\in\bbbz}\esssup_{t\in[1,2]}2^{ks}\|\varDelta^L_{2^{-k}t,1}f\|_{L^p(\bbbr^n)}\lesssim
\esssup_{t>0}t^{-s}\|\varDelta^L_{t,1}f\|_{L^p(\bbbr^n)}
\end{equation}
for any $s\in\bbbr$. And inequality (\ref{eq661}) is true if we replace $\rho_1,\varDelta^L_{t,1}$ by $\rho_j,\varDelta^L_{t,j}$ respectively for $j=2,\cdots,n$. Inserting these inequalities into (\ref{eq656}) yields the desired inequality (\ref{eq603}). If $0<p\leq 1$, $q=\infty$ and $\sigma_p<s<\infty$, then we use the function $\phi$ satisfying conditions (\ref{eq346}), (\ref{eq422}), (\ref{eq423}), and $J>m_0$ is a large positive integer whose value will be determined later. Since $0<r<p\leq 1$, we use the following inequality
\begin{equation}\label{eq662}
\mvint_{A_{k-l}}|\varDelta^L_{2^{-k}t,1}(\phi_{2^{m_0-J-k}}*f)(x-y)|^rdy\lesssim
\HLmax_n(|\varDelta^L_{2^{-k}t,1}(\phi_{2^{m_0-J-k}}*f)|^r)(x),
\end{equation}
and deduce from (\ref{eq494}) and (\ref{eq495}) the estimate below
\begin{align}\label{eq663}
&|\iFT_n[\frac{\rho_1(2^{m_0-k}\xi)}{(e^{2\pi i\cdot 2^{-k}t\xi_1}-1)^L}]*
  [\varDelta^L_{2^{-k}t,1}(\phi_{2^{m_0-J-k}}*f)(\cdot)](x)| \nonumber\\
&\lesssim 2^{Jn(\frac{1}{r}-1)}\HLmax_n(|\varDelta^L_{2^{-k}t,1}(\phi_{2^{m_0-J-k}}*f)|^r)(x)^{\frac{1}{r}},
\end{align}
where we let $N'>\frac{n}{r}$ in (\ref{eq495}). From (\ref{eq417}) we know that when $\rho_1(2^{m_0-k}\xi)\neq 0$ and $t\in[1,2]$, $|(e^{2\pi i\cdot 2^{-k}t\xi_1}-1)^L|>0$ and we can infer from the calculation method displayed in the second line of (\ref{eq577}) that
\begin{align}\label{eq664}
&\esssup_{k\in\bbbz}2^{ks}\|\iFT_n[\rho_1(2^{m_0-k}\xi)\FT_n f]\|_{L^p(\bbbr^n)} \nonumber\\
&=\esssup_{k\in\bbbz}2^{ks}\esssup_{t\in[1,2]}\|\iFT_n[\frac{\rho_1(2^{m_0-k}\xi)}{(e^{2\pi i\cdot 2^{-k}t\xi_1}-1)^L}] \nonumber\\
&\quad*[\varDelta^L_{2^{-k}t,1}(\phi_{2^{m_0-J-k}}*f)]\|_{L^p(\bbbr^n)} \nonumber\\
&\lesssim\esssup_{k\in\bbbz}2^{ks+Jn(\frac{1}{r}-1)}\esssup_{t\in[1,2]}
          \|\HLmax_n(|\varDelta^L_{2^{-k}t,1}(\phi_{2^{m_0-J-k}}*f)|^r)\|_{L^{\frac{p}{r}}(\bbbr^n)}^{\frac{1}{r}} \nonumber\\
&\lesssim\esssup_{k\in\bbbz}2^{ks+Jn(\frac{1}{r}-1)}\esssup_{t\in[1,2]}
          \|\varDelta^L_{2^{-k}t,1}(\phi_{2^{m_0-J-k}}*f)\|_{L^p(\bbbr^n)}.
\end{align}
Recall (\ref{eq630}) and assume the validity of decomposition for now, then we can estimate (\ref{eq664}) from above by the sum of the following two terms,
\begin{align}\label{eq665}
&2^{Jn(\frac{1}{r}-1)}\esssup_{k\in\bbbz}\esssup_{t\in[1,2]}2^{ks}\|\varDelta^L_{2^{-k}t,1}f\|_{L^p(\bbbr^n)} \nonumber\\
&\sim 2^{Jn(\frac{1}{r}-1)}\esssup_{t>0}t^{-s}\|\varDelta^L_{t,1}f\|_{L^p(\bbbr^n)}<\infty,
\end{align}
and
\begin{equation}\label{eq666}
2^{Jn(\frac{1}{r}-1)}\esssup_{k\in\bbbz}\esssup_{t\in[1,2]}2^{ks}\|\varDelta^L_{2^{-k}t,1}
(\sum_{j=J+1}^{\infty}f_{k+j-m_0})\|_{L^p(\bbbr^n)}.
\end{equation}
To estimate (\ref{eq666}), we pick $0<\varepsilon<s$ and use (\ref{eq612}) to obtain
\begin{align}\label{eq667}
&\|\sum_{j=J+1}^{\infty}|\varDelta^L_{2^{-k}t,1}f_{k+j-m_0}|\|_{L^p(\bbbr^n)} \nonumber\\
&\lesssim(\sum_{j=J+1}^{\infty}\|\varDelta^L_{2^{-k}t,1}f_{k+j-m_0}\|_{L^p(\bbbr^n)}^p)^{\frac{1}{p}} \nonumber\\
&\lesssim(\sum_{j=J+1}^{\infty}2^{-jp\varepsilon}\cdot 2^{jp\varepsilon}\|f_{k+j-m_0}\|_{L^p(\bbbr^n)}^p)^{\frac{1}{p}} \nonumber\\
&\lesssim 2^{-J\varepsilon}\esssup_{j>J}2^{j\varepsilon}\|f_{k+j-m_0}\|_{L^p(\bbbr^n)},
\end{align}
where the constants are independent of $t$. And then we can have the following estimate
\begin{align}\label{eq669}
&2^{Jn(\frac{1}{r}-1)}\esssup_{k\in\bbbz}\esssup_{t\in[1,2]}2^{ks}
  \|\sum_{j=J+1}^{\infty}|\varDelta^L_{2^{-k}t,1}f_{k+j-m_0}|\|_{L^p(\bbbr^n)} \nonumber\\
&\lesssim 2^{J[n(\frac{1}{r}-1)-\varepsilon]}\esssup_{j>J}2^{j(\varepsilon-s)+m_0 s}\esssup_{k\in\bbbz}2^{(k+j-m_0)s}
          \|f_{k+j-m_0}\|_{L^p(\bbbr^n)} \nonumber\\
&\lesssim 2^{J[n(\frac{1}{r}-1)-s]}\|f\|_{\Bspinf}.
\end{align}
The assumption $f\in\Bspinf$ implies $\sum_{j=J+1}^{\infty}|\varDelta^L_{2^{-k}t,1}f_{k+j-m_0}(x)|<\infty$ for every $k\in\bbbz$ and almost every $t\in[1,2],x\in\bbbr^n$. In conjunction with (\ref{eq514}) and (\ref{eq665}), we have shown (\ref{eq630}) is true not only in the sense of $\Sw'(\bbbr^n)$ but also for every $k\in\bbbz$ and almost every $t\in[1,2],x\in\bbbr^n$ when $0<p\leq 1$, $q=\infty$ and $\sigma_p<s<\infty$. Furthermore estimating (\ref{eq664}) from above by the sum of (\ref{eq665}) and (\ref{eq666}) is justified, moreover (\ref{eq666}) can be estimated from above by the first line of (\ref{eq669}) and hence by the last line of (\ref{eq669}). We have obtained the inequality
\begin{align}\label{eq670}
&\esssup_{k\in\bbbz}2^{ks}\|\iFT_n[\rho_1(2^{m_0-k}\xi)\FT_n f]\|_{L^p(\bbbr^n)} \nonumber\\
&\leq C_1'2^{Jn(\frac{1}{r}-1)}\esssup_{t>0}t^{-s}\|\varDelta^L_{t,1}f\|_{L^p(\bbbr^n)}+
C_1''2^{J[n(\frac{1}{r}-1)-s]}\|f\|_{\Bspinf}.
\end{align}
Inequality (\ref{eq670}) is also true if we replace $\rho_1,\varDelta^L_{t,1},C_1',C_1''$ by $\rho_j,\varDelta^L_{t,j},C_j',C_j''$ respectively for $j=2,\cdots,n$. Inserting these inequalities into (\ref{eq656}) yields
\begin{align}\label{eq671}
\|f\|_{\Bspinf}&\leq C'2^{Jn(\frac{1}{r}-1)}\sum_{j=1}^n\esssup_{t>0}t^{-s}\|\varDelta^L_{t,j}f\|_{L^p(\bbbr^n)} \nonumber\\
&\quad+C''2^{J[n(\frac{1}{r}-1)-s]}\|f\|_{\Bspinf}.
\end{align}
The assumption $\sigma_p<s<\infty$ allows $n(\frac{1}{r}-1)-s<0$ when $r$ is sufficiently close to $p$. Thus when $J$ is a sufficiently large positive integer, the coefficient $C''2^{J[n(\frac{1}{r}-1)-s]}$ is less than $\frac{1}{2}$, and we can shift the second term on the right side of (\ref{eq671}) to its left side and then the desired inequality (\ref{eq603}) is proved. Now the proof of Theorem \ref{theorem7} is complete.
\end{proof}

\section{Besov-Lipschitz space and iterated differences}\label{proof.of.theorem8}
\begin{proof}[Proof Of Theorem \ref{theorem8}]
We first prove Theorem \ref{theorem8} (i). We continue using the notation $A_k=\{h\in\bbbr^n:2^{-k}\leq|h|<2^{1-k}\}$ for $k\in\bbbz$ and thus $\bbbr^n\setminus\{0\}=\bigcup_{k\in\bbbz}A_k$. We also pick the number $r$ in the definition of the Peetre-Fefferman-Stein maximal function so that $0<r<p$. We begin with estimating $$(\int_{\bbbr^n}|h|^{-sq}\|
\sum_{j\in\bbbz}|\varDelta^L_h f_j|\|_{L^p(\bbbr^n)}^q\frac{dh}{|h|^n})^{\frac{1}{q}}$$ from above by the following
\begin{align}\label{eq704}
&(\sum_{k\in\bbbz}2^{k(sq+n)}\int_{A_k}\|\sum_{j\leq k}|\varDelta^L_h f_j|\|_{L^p(\bbbr^n)}^q dh)^{\frac{1}{q}} \nonumber\\
& \quad+(\sum_{k\in\bbbz}2^{k(sq+n)}\int_{A_k}\|\sum_{j>k}|\varDelta^L_h f_j|\|_{L^p(\bbbr^n)}^q dh)^{\frac{1}{q}}.
\end{align}
We can pick $0<\varepsilon<\min\{s,L-s\}$ and use the same calculation method as in (\ref{eq606}), (\ref{eq607}), (\ref{eq608}), (\ref{eq609}) to obtain
\begin{align}
\|\sum_{j\leq k}|\varDelta^L_h f_j|\|_{L^p(\bbbr^n)}^q&\lesssim
2^{kq\varepsilon}\sum_{j\leq k}2^{-jq\varepsilon}\|\varDelta^L_h f_j\|_{L^p(\bbbr^n)}^q,\label{eq705}\\
\|\sum_{j>k}|\varDelta^L_h f_j|\|_{L^p(\bbbr^n)}^q&\lesssim
2^{-kq\varepsilon}\sum_{j>k}2^{jq\varepsilon}\|\varDelta^L_h f_j\|_{L^p(\bbbr^n)}^q.\label{eq706}
\end{align}
And (\ref{eq705}), (\ref{eq706}) are true for $0<p\leq\infty$, $0<q<\infty$. To estimate the first term in (\ref{eq704}), we use (\ref{eq270}), Remark \ref{remark12} and the mapping property of Hardy-Littlewood maximal function for $\|\cdot\|_{L^{\frac{p}{r}}(\bbbr^n)}$-norm to obtain
\begin{equation}\label{eq707}
\|\varDelta^L_h f_j\|_{L^p(\bbbr^n)}\lesssim 2^{(j-k)L}(1+2^{j-k})^{\frac{n}{r}}\|\PFSmax_n f_j\|_{L^p(\bbbr^n)}
\lesssim 2^{(j-k)L}\|f_j\|_{L^p(\bbbr^n)},
\end{equation}
for $j\leq k$, $|h|\lesssim 2^{-k}$ and constants are independent of $h\in\bbbr^n$ and $j,k\in\bbbz$. We put (\ref{eq705}) and (\ref{eq707}) into the first term of (\ref{eq704}) then we have
\begin{align}\label{eq710}
&(\sum_{k\in\bbbz}2^{k(sq+n)}\int_{A_k}\|\sum_{j\leq k}|\varDelta^L_h f_j|\|_{L^p(\bbbr^n)}^q dh)^{\frac{1}{q}} \nonumber\\
&\lesssim(\sum_{k\in\bbbz}\sum_{j\leq k}2^{kq(s+\varepsilon-L)}\cdot
          2^{jq(L-\varepsilon)}\|f_j\|_{L^p(\bbbr^n)}^q)^{\frac{1}{q}} \nonumber\\
&\lesssim(\sum_{j\in\bbbz}\sum_{k\geq j}2^{kq(s+\varepsilon-L)}\cdot
          2^{jq(L-\varepsilon)}\|f_j\|_{L^p(\bbbr^n)}^q)^{\frac{1}{q}}\lesssim\|f\|_{\Bspq}.
\end{align}
To estimate the second term of (\ref{eq704}), we can use (\ref{eq260}) and proper change of variable to obtain
\begin{equation}\label{eq708}
\|\varDelta^L_h f_j\|_{L^p(\bbbr^n)}\lesssim\|f_j\|_{L^p(\bbbr^n)}\qquad\text{ for all }j\in\bbbz,
\end{equation}
and the constant is independent of $j\in\bbbz$ and $h\in\bbbr^n$. We put (\ref{eq706}) and (\ref{eq708}) into the second term of (\ref{eq704}) then we have
\begin{align}\label{eq709}
&(\sum_{k\in\bbbz}2^{k(sq+n)}\int_{A_k}\|\sum_{j>k}|\varDelta^L_h f_j|\|_{L^p(\bbbr^n)}^q dh)^{\frac{1}{q}} \nonumber\\
&\lesssim(\sum_{k\in\bbbz}\sum_{j>k}2^{kq(s-\varepsilon)}\cdot 2^{jq\varepsilon}\|f_j\|_{L^p(\bbbr^n)}^q)^{\frac{1}{q}} \nonumber\\
&\lesssim(\sum_{j\in\bbbz}\sum_{k<j}2^{kq(s-\varepsilon)}\cdot 2^{jq\varepsilon}\|f_j\|_{L^p(\bbbr^n)}^q)^{\frac{1}{q}}
         \lesssim\|f\|_{\Bspq}.
\end{align}
Combining (\ref{eq704}), (\ref{eq710}) and (\ref{eq709}), we have proven
\begin{equation}\label{eq711}
(\int_{\bbbr^n}|h|^{-sq}\|\sum_{j\in\bbbz}|\varDelta^L_h f_j|\|_{L^p(\bbbr^n)}^q\frac{dh}{|h|^n})^{\frac{1}{q}}
\lesssim\|f\|_{\Bspq}.
\end{equation}
The assumption $f\in\Bspq$ implies $\sum_{j\in\bbbz}|\varDelta^L_h f_j(x)|<\infty$ for almost every $h\in\bbbr^n$ and $x\in\bbbr^n$. In conjunction with (\ref{eq513}), we have shown
\begin{equation}\label{eq712}
\varDelta^L_h f=\sum_{j\in\bbbz}\varDelta^L_h f_j(x)
\end{equation}
in the sense of $\Sw'(\bbbr^n)/\mathscr{P}(\bbbr^n)$ for almost every $h\in\bbbr^n$ and $x\in\bbbr^n$ when $0<p\leq\infty$, $0<q<\infty$ and $0<s<L$. Therefore we obtain
\begin{equation}\label{eq713}
(\int_{\bbbr^n}|h|^{-sq}\|\varDelta^L_h f\|_{L^p(\bbbr^n)}^q\frac{dh}{|h|^n})^{\frac{1}{q}}\lesssim
(\int_{\bbbr^n}|h|^{-sq}\|\sum_{j\in\bbbz}|\varDelta^L_h f_j|\|_{L^p(\bbbr^n)}^q\frac{dh}{|h|^n})^{\frac{1}{q}}.
\end{equation}
Inequalities (\ref{eq711}) and (\ref{eq713}) conclude the proof of Theorem \ref{theorem8} (i).

To prove Theorem \ref{theorem8} (ii), we assume the right side of inequality (\ref{eq699}) is finite otherwise the inequality is trivial. We recall that $spt.\FT_n\psi\subseteq A'=\{\xi\in\bbbr^n:\frac{1}{2}\leq|\xi|<2\}$ and use the positive integer $m_0$ satisfying (\ref{eq417}). We also continue using the spherical caps $\{C_l\}_{l=1}^M$ constructed right after (\ref{eq417}), the corresponding sets $\{P_l\}_{l=1}^M$ given in (\ref{eq418}), (\ref{eq419}), (\ref{eq420}), and the associated smooth partition of unity $\{\rho_l\}_{l=1}^M$ satisfying (\ref{eq421}). We have the apparent estimate
\begin{equation}\label{eq714}
\|f\|_{\Bspq}=\|\{2^{k-m_0}f_{k-m_0}\}_{k\in\bbbz}\|_{l^q(L^p)}\lesssim
(\sum_{k\in\bbbz}2^{ksq}\|\psi_{2^{m_0-k}}*f\|_{L^p(\bbbr^n)}^q)^{\frac{1}{q}}.
\end{equation}
When $1<p\leq\infty$, we can obtain from (\ref{eq421}) the following
\begin{equation}\label{eq715}
\psi_{2^{m_0-k}}*f(x)=\sum_{l=1}^M\iFT_n[\FT_n\psi(2^{m_0-k}\xi)\rho_l(2^{m_0-k}\xi)\FT_n f](x).
\end{equation}
For each $l\in\{1,\cdots,M\}$, $\theta\in C_l$ and $1\leq\tau\leq 2$, we can infer from (\ref{eq424}), (\ref{eq425}) and (\ref{eq426}) the following
\begin{align}\label{eq716}
&\iFT_n[\FT_n\psi(2^{m_0-k}\xi)\rho_l(2^{m_0-k}\xi)\FT_n f](x) \nonumber\\
&=\iFT_n[\frac{\FT_n\psi(2^{m_0-k}\xi)\rho_l(2^{m_0-k}\xi)}{(e^{2\pi i\cdot 2^{-k}\tau\theta\cdot\xi}-1)^L}\cdot
   (e^{2\pi i\cdot 2^{-k}\tau\theta\cdot\xi}-1)^L\FT_n f](x) \nonumber\\
&=\iFT_n\lambda_{l,\tau\theta}(2^{-k}\xi)(\cdot)*(\varDelta^L_{2^{-k}\tau\theta}f)(x).
\end{align}
Due to Lemma \ref{lemma2} and hence (\ref{eq429}), we have
\begin{align}\label{eq717}
&|\iFT_n\lambda_{l,\tau\theta}(2^{-k}\xi)(\cdot)*(\varDelta^L_{2^{-k}\tau\theta}f)(x)| \nonumber\\
&\lesssim\int_{2^k|y|<1}2^{kn}|\iFT_n\lambda_{l,\tau\theta}(2^{k}y)|\cdot|(\varDelta^L_{2^{-k}\tau\theta}f)(x-y)|dy \nonumber\\
&\quad+\sum_{l=0}^{\infty}
      \int_{A_{k-l}}2^{kn}|\iFT_n\lambda_{l,\tau\theta}(2^{k}y)|\cdot|(\varDelta^L_{2^{-k}\tau\theta}f)(x-y)|dy \nonumber\\
&\lesssim\mvint_{2^k|y|<1}|(\varDelta^L_{2^{-k}\tau\theta}f)(x-y)|dy+\sum_{l=0}^{\infty}2^{l(n-N)}
         \mvint_{A_{k-l}}|(\varDelta^L_{2^{-k}\tau\theta}f)(x-y)|dy \nonumber\\
&\lesssim(1+\sum_{l=0}^{\infty}2^{l(n-N)})\HLmax_n(|\varDelta^L_{2^{-k}\tau\theta}f|)(x)\lesssim
         \HLmax_n(|\varDelta^L_{2^{-k}\tau\theta}f|)(x),
\end{align}
if in the last step above we pick $N>n$. Since $1<p\leq\infty$, we invoke the mapping property of Hardy-Littlewood maximal function and obtain for $0<q<\infty$,
\begin{align}\label{eq718}
&\|\psi_{2^{m_0-k}}*f\|_{L^p(\bbbr^n)}^q \nonumber\\
&\lesssim\sum_{l=1}^M\mvint_{[1,2]}\mvint_{C_l}\|\iFT_n[\FT_n\psi(2^{m_0-k}\xi)\rho_l(2^{m_0-k}\xi)\FT_n f]\|_{L^p(\bbbr^n)}^q
          d\Haus^{n-1}(\theta)\frac{d\tau}{\tau} \nonumber\\
&\lesssim\sum_{l=1}^M\int_1^2\int_{C_l}\|\iFT_n\lambda_{l,\tau\theta}(2^{-k}\xi)(\cdot)*(\varDelta^L_{2^{-k}\tau\theta}f)
          \|_{L^p(\bbbr^n)}^q d\Haus^{n-1}(\theta)\frac{d\tau}{\tau} \nonumber\\
&\lesssim\int_1^2\int_{\unitsph}\|\varDelta^L_{2^{-k}\tau\theta}f\|_{L^p(\bbbr^n)}^q d\Haus^{n-1}(\theta)\frac{d\tau}{\tau} \nonumber\\
&\lesssim\int_{A_k}\|\varDelta^L_h f\|_{L^p(\bbbr^n)}^q\frac{dh}{|h|^n},
\end{align}
for every $k\in\bbbz$. Inserting (\ref{eq718}) into (\ref{eq714}) proves (\ref{eq699}) when $1<p\leq\infty$, $0<q<\infty$ and $s\in\bbbr$. When $0<p<1$, $0<q<\infty$ and $\sigma_p<s<\infty$, we use the function $\phi$ satisfying conditions (\ref{eq346}), (\ref{eq422}), (\ref{eq423}), and $J>m_0$ is a large positive integer whose value will be determined later. Then we have
\begin{equation}\label{eq719}
\psi_{2^{m_0-k}}*f(x)=\sum_{l=1}^M\iFT_n[\FT_n\psi(2^{m_0-k}\xi)\rho_l(2^{m_0-k}\xi)
\FT_n\phi(2^{m_0-J-k}\xi)\FT_n f](x).
\end{equation}
Furthermore for each $l\in\{1,\cdots,M\}$, $\theta\in C_l$ and $1\leq\tau\leq 2$, we can infer from (\ref{eq424}), (\ref{eq425}) and (\ref{eq426}) the following
\begin{align}\label{eq720}
&\iFT_n[\FT_n\psi(2^{m_0-k}\xi)\rho_l(2^{m_0-k}\xi)\FT_n\phi(2^{m_0-J-k}\xi)\FT_n f](x) \nonumber\\
&=\iFT_n[\frac{\FT_n\psi(2^{m_0-k}\xi)\rho_l(2^{m_0-k}\xi)}{(e^{2\pi i\cdot 2^{-k}\tau\theta\cdot\xi}-1)^L}\cdot
  (e^{2\pi i\cdot 2^{-k}\tau\theta\cdot\xi}-1)^L \nonumber\\
&\quad\cdot\FT_n\phi(2^{m_0-J-k}\xi)\FT_n f](x) \nonumber\\
&=\iFT_n\lambda_{l,\tau\theta}(2^{-k}\xi)(\cdot)*(\varDelta^L_{2^{-k}\tau\theta}(\phi_{2^{m_0-J-k}}*f))(x).
\end{align}
The Fourier transform of the Schwartz function
$$y\mapsto 2^{kn}\iFT_n\lambda_{l,\tau\theta}(2^{k}y)\cdot\varDelta^L_{2^{-k}\tau\theta}(\phi_{2^{m_0-J-k}}*f)(x-y)$$
is compactly supported in a ball centered at origin of radius about $2^{k+J-m_0}$ in $\bbbr^n$, therefore by invoking Remark \ref{remark5} or the more general Lemma \ref{lemma7} of Plancherel-Polya-Nikol'skij inequality, we can estimate (\ref{eq720}) from above by
\begin{align}\label{eq721}
&\|2^{kn}\iFT_n\lambda_{l,\tau\theta}(2^{k}\cdot)\cdot
\varDelta^L_{2^{-k}\tau\theta}(\phi_{2^{m_0-J-k}}*f)(x-\cdot)\|_{L^1(\bbbr^n)} \nonumber\\
&\lesssim 2^{(k+J-m_0)n(\frac{1}{r}-1)}\|2^{kn}\iFT_n\lambda_{l,\tau\theta}(2^{k}\cdot)\cdot
\varDelta^L_{2^{-k}\tau\theta}(\phi_{2^{m_0-J-k}}*f)(x\!-\!\cdot)\|_{L^r(\bbbr^n)},
\end{align}
for $0<r<p<1$. We insert (\ref{eq430}) with $n-Nr<0$ into (\ref{eq721}) and use the following inequality
\begin{equation}\label{eq722}
\mvint_{A_{k-m}}|\varDelta^L_{2^{-k}\tau\theta}(\phi_{2^{m_0-k-J}}*f)(x-y)|^r dy\lesssim
\HLmax_n(|\varDelta^L_{2^{-k}\tau\theta}(\phi_{2^{m_0-k-J}}*f)|^r)(x),
\end{equation}
and then combine the result with (\ref{eq720}) to obtain
\begin{align}\label{eq723}
&|\iFT_n[\FT_n\psi(2^{m_0-k}\xi)\rho_l(2^{m_0-k}\xi)\FT_n\phi(2^{m_0-J-k}\xi)\FT_n f](x)| \nonumber\\
&\lesssim 2^{Jn(\frac{1}{r}-1)}\HLmax_n(|\varDelta^L_{2^{-k}\tau\theta}(\phi_{2^{m_0-k-J}}*f)|^r)(x)^{\frac{1}{r}},
\end{align}
for every $l\in\{1,\cdots,M\}$, $\theta\in C_l$ and $\tau\in[1,2]$. Then we use (\ref{eq719}), (\ref{eq723}), the calculation method displayed in (\ref{eq718}) and the mapping property of Hardy-Littlewood maximal function for $\|\cdot\|_{L^{\frac{p}{r}}(\bbbr^n)}$-norm to obtain
\begin{align}\label{eq724}
&\|\psi_{2^{m_0-k}}*f\|_{L^p(\bbbr^n)}^q \nonumber\\
&\lesssim 2^{Jnq(\frac{1}{r}-1)}\sum_{l=1}^M\int_1^2\!\!\int_{C_l}\!\!
          \|\HLmax_n(|\varDelta^L_{2^{-k}\tau\theta}(\phi_{2^{m_0-k-J}}*f)|^r)^{\frac{1}{r}}\|_{L^p(\bbbr^n)}^q
          d\Haus^{n-1}(\theta)\frac{d\tau}{\tau} \nonumber\\
&\lesssim 2^{Jnq(\frac{1}{r}-1)}\int_1^2\int_{\unitsph}
          \|\varDelta^L_{2^{-k}\tau\theta}(\phi_{2^{m_0-k-J}}*f)\|_{L^p(\bbbr^n)}^q
          d\Haus^{n-1}(\theta)\frac{d\tau}{\tau} \nonumber\\
&\lesssim 2^{Jnq(\frac{1}{r}-1)}\int_{A_k}\|\varDelta^L_h(\phi_{2^{m_0-k-J}}*f)\|_{L^p(\bbbr^n)}^q\frac{dh}{|h|^n},
\end{align}
and (\ref{eq724}) is true for $0<q<\infty$. We insert (\ref{eq724}) into (\ref{eq714}), recall (\ref{eq423}), (\ref{eq500}) and (\ref{eq548}), and we also assume the validity of decomposition for now, then $\|f\|_{\Bspq}$ can be estimated from above by the sum of the following two terms,
\begin{align}\label{eq725}
&2^{Jn(\frac{1}{r}-1)}(\sum_{k\in\bbbz}2^{ksq}\int_{A_k}\|\varDelta^L_h f\|_{L^p(\bbbr^n)}^q\frac{dh}{|h|^n})^{\frac{1}{q}} \nonumber\\
&\sim 2^{Jn(\frac{1}{r}-1)}(\int_{\bbbr^n}|h|^{-sq}\|\varDelta^L_h f\|_{L^p(\bbbr^n)}^q\frac{dh}{|h|^n})^{\frac{1}{q}}<\infty
\end{align}
and
\begin{equation}\label{eq726}
2^{Jn(\frac{1}{r}-1)}(\sum_{k\in\bbbz}2^{ksq}\int_{A_k}\|\varDelta^L_h(\sum_{j=J+1}^{\infty}f_{k+j-m_0})
\|_{L^p(\bbbr^n)}^q\frac{dh}{|h|^n})^{\frac{1}{q}}.
\end{equation}
To estimate (\ref{eq726}), we begin with
\begin{align}\label{eq727}
&2^{Jn(\frac{1}{r}-1)}(\sum_{k\in\bbbz}2^{ksq}\int_{A_k}\|\sum_{j=J+1}^{\infty}|\varDelta^L_h f_{k+j-m_0}|
\|_{L^p(\bbbr^n)}^q\frac{dh}{|h|^n})^{\frac{1}{q}} \nonumber\\
&\lesssim 2^{Jn(\frac{1}{r}-1)}(\sum_{k\in\bbbz}2^{k(sq+n)}\int_{A_k}
(\sum_{j=J+1}^{\infty}\|\varDelta^L_h f_{k+j-m_0}\|_{L^p(\bbbr^n)}^p)^{\frac{q}{p}}dh)^{\frac{1}{q}},
\end{align}
since $0<p<1$. When $0<q\leq p$, we use (\ref{eq708}) and the following inequality
\begin{align}\label{eq728}
&(\sum_{j=J+1}^{\infty}\|\varDelta^L_h f_{k+j-m_0}\|_{L^p(\bbbr^n)}^p)^{\frac{q}{p}} \nonumber\\
&\lesssim\sum_{j=J+1}^{\infty}\|\varDelta^L_h f_{k+j-m_0}\|_{L^p(\bbbr^n)}^q
\lesssim\sum_{j=J+1}^{\infty}\|f_{k+j-m_0}\|_{L^p(\bbbr^n)}^q,
\end{align}
where the constants are independent of $j,k\in\bbbz$ and $h\in\bbbr^n$. When $p<q<\infty$, we pick $0<\varepsilon<s$, use (\ref{eq708}) and the following inequality
\begin{align}\label{eq729}
&(\sum_{j=J+1}^{\infty}2^{-jp\varepsilon}\cdot 2^{jp\varepsilon}\|\varDelta^L_h f_{k+j-m_0}\|_{L^p(\bbbr^n)}^p)^{\frac{q}{p}} \nonumber\\
&\lesssim(\sum_{l=J+1}^{\infty}2^{-lp\varepsilon})^{\frac{q}{p}}\esssup_{j>J}
          2^{jq\varepsilon}\|\varDelta^L_h f_{k+j-m_0}\|_{L^p(\bbbr^n)}^q \nonumber\\
&\lesssim 2^{-Jq\varepsilon}\sum_{j=J+1}^{\infty}2^{jq\varepsilon}\|f_{k+j-m_0}\|_{L^p(\bbbr^n)}^q,
\end{align}
where the constants are independent of $j,k\in\bbbz$ and $h\in\bbbr^n$. Insert (\ref{eq728}) and (\ref{eq729}) into (\ref{eq727}) and exchange the order of summation, and then we can estimate (\ref{eq727}) from above by $2^{J[n(\frac{1}{r}-1)-s]}\|f\|_{\Bspq}$. The assumption $f\in\Bspq$ implies $\sum_{j=J+1}^{\infty}|\varDelta^L_h f_{k+j-m_0}(x)|<\infty$ for every $k\in\bbbz$ and almost every $h\in A_k,x\in\bbbr^n$. In conjunction with (\ref{eq514}), (\ref{eq500}) and (\ref{eq725}), we have shown
\begin{equation}\label{eq730}
\varDelta^L_{h}(\phi_{2^{m_0-k-J}}*f)(x)=\varDelta^L_{h}f(x)-\sum_{j=J+1}^{\infty}\varDelta^L_{h}f_{k+j-m_0}(x)
\end{equation}
is true not only in the sense of $\Sw'(\bbbr^n)$ but also for every $k\in\bbbz$ and almost every $h\in A_k,x\in\bbbr^n$ when $0<p<1$, $0<q<\infty$ and $\sigma_p<s<\infty$. Furthermore estimating $\|f\|_{\Bspq}$ from above by the sum of (\ref{eq725}) and (\ref{eq726}) is justified, moreover (\ref{eq726}) can be estimated from above by (\ref{eq727}) and hence by $2^{J[n(\frac{1}{r}-1)-s]}\|f\|_{\Bspq}$. We have obtained the inequality
\begin{equation}\label{eq731}
\|f\|_{\Bspq}\!\leq\!
C'2^{Jn(\frac{1}{r}-1)}(\int_{\bbbr^n}\!\!\!\!|h|^{-sq}\|\varDelta^L_h f\|_{L^p(\bbbr^n)}^q\frac{dh}{|h|^n})^{\frac{1}{q}}\!+\!
C''2^{J[n(\frac{1}{r}-1)-s]}\|f\|_{\Bspq}.
\end{equation}
The condition $\sigma_p<s<\infty$ implies $n(\frac{1}{r}-1)-s<0$ when $r$ is sufficiently close to $p$. Hence when the positive integer $J$ is sufficiently large, the coefficient $C''2^{J[n(\frac{1}{r}-1)-s]}$ is less than $\frac{1}{2}$ and then we can shift the second term on the right side of (\ref{eq731}) to its left side and prove the inequality (\ref{eq699}) when $0<p<1$, $0<q<\infty$ and $\sigma_p<s<\infty$. The proof for Theorem \ref{theorem8} (ii) is now concluded.

Next, we prove Theorem \ref{theorem8} (iii). We pick $0<\varepsilon<\min\{s,L-s\}$ and begin with estimating $\esssup_{h\in\bbbr^n}|h|^{-s}\|\sum_{j\in\bbbz}|\varDelta^L_h f_j|\|_{L^p(\bbbr^n)}$ from above by
\begin{equation}\label{eq732}
\esssup_{k\in\bbbz}\esssup_{h\in A_k}2^{ks}\|\sum_{j\leq k}|\varDelta^L_h f_j|\|_{L^p(\bbbr^n)}+
\esssup_{k\in\bbbz}\esssup_{h\in A_k}2^{ks}\|\sum_{j>k}|\varDelta^L_h f_j|\|_{L^p(\bbbr^n)}.
\end{equation}
If $1\leq p\leq\infty$, by using Minkowski's inequality, (\ref{eq270}) and the calculation method displayed in (\ref{eq649}), we can estimate the first term in (\ref{eq732}) from above by $\|f\|_{\Bspinf}$. And by using Minkowski's inequality, (\ref{eq708}) and the calculation method displayed in (\ref{eq650}), we can estimate the second term in (\ref{eq732}) from above by $\|f\|_{\Bspinf}$. If $0<p<1$, by applying (\ref{eq270}) and the calculation method given in (\ref{eq651}) to the first term of (\ref{eq732}), and by applying (\ref{eq708}) and the calculation method given in (\ref{eq652}) to the second term of (\ref{eq732}), we can still estimate (\ref{eq732}) from above by $\|f\|_{\Bspinf}$. Thus we can obtain the inequality
\begin{equation}\label{eq733}
\esssup_{h\in\bbbr^n}|h|^{-s}\|\sum_{j\in\bbbz}|\varDelta^L_h f_j|\|_{L^p(\bbbr^n)}\lesssim\|f\|_{\Bspinf}.
\end{equation}
The assumption $f\in\Bspinf$ shows $\sum_{j\in\bbbz}|\varDelta^L_h f_j(x)|<\infty$ for almost every $h\in\bbbr^n,x\in\bbbr^n$. In conjunction with (\ref{eq513}), we have shown (\ref{eq712}) is true not only in the sense of $\Sw'(\bbbr^n)/\mathscr{P}(\bbbr^n)$ but also for almost every $h\in\bbbr^n,x\in\bbbr^n$ when $0<p\leq\infty$, $q=\infty$ and $0<s<L$. Therefore we have
\begin{equation}\label{eq734}
\esssup_{h\in\bbbr^n}|h|^{-s}\|\varDelta^L_h f\|_{L^p(\bbbr^n)}\lesssim
\esssup_{h\in\bbbr^n}|h|^{-s}\|\sum_{j\in\bbbz}|\varDelta^L_h f_j|\|_{L^p(\bbbr^n)}.
\end{equation}
Inequalities (\ref{eq733}) and (\ref{eq734}) conclude the proof of Theorem \ref{theorem8} (iii).

Now we prove Theorem \ref{theorem8} (iv). We assume the right side of (\ref{eq701}) is finite, otherwise the inequality is trivial. We still use the positive integer $m_0$ satisfying (\ref{eq417}), the spherical caps $\{C_l\}_{l=1}^M$ constructed right after (\ref{eq417}), the corresponding sets $\{P_l\}_{l=1}^M$ given in (\ref{eq418}), (\ref{eq419}), (\ref{eq420}), and the associated smooth partition of unity $\{\rho_l\}_{l=1}^M$ satisfying (\ref{eq421}). When $1<p\leq\infty$, we use (\ref{eq715}), (\ref{eq716}) and (\ref{eq717}) to obtain
\begin{align}\label{eq735}
&\|\psi_{2^{m_0-k}}*f\|_{L^p(\bbbr^n)} \nonumber\\
&\lesssim\sum_{l=1}^M\esssup_{\tau\in[1,2]}\esssup_{\theta\in C_l}
          \|\iFT_n[\FT_n\psi(2^{m_0-k}\xi)\rho_l(2^{m_0-k}\xi)\FT_n f]\|_{L^p(\bbbr^n)} \nonumber\\
&\lesssim\sum_{l=1}^M\esssup_{\tau\in[1,2]}\esssup_{\theta\in C_l}
          \|\iFT_n\lambda_{l,\tau\theta}(2^{-k}\xi)(\cdot)*(\varDelta^L_{2^{-k}\tau\theta}f)\|_{L^p(\bbbr^n)} \nonumber\\
&\lesssim\sum_{l=1}^M\esssup_{\tau\in[1,2]}\esssup_{\theta\in C_l}
          \|\HLmax_n(|\varDelta^L_{2^{-k}\tau\theta}f|)\|_{L^p(\bbbr^n)} \nonumber\\
&\lesssim\esssup_{\tau\in[1,2]}\esssup_{\theta\in\unitsph}\|\varDelta^L_{2^{-k}\tau\theta}f\|_{L^p(\bbbr^n)}
          =\esssup_{h\in A_k}\|\varDelta^L_h f\|_{L^p(\bbbr^n)}.
\end{align}
Therefore we have the inequality
\begin{equation}\label{eq736}
\|f\|_{\Bspinf}\lesssim\esssup_{k\in\bbbz}2^{ks}\|\psi_{2^{m_0-k}}*f\|_{L^p(\bbbr^n)}\lesssim
\esssup_{h\in\bbbr^n}|h|^{-s}\|\varDelta^L_h f\|_{L^p(\bbbr^n)},
\end{equation}
and inequality (\ref{eq736}) is true for $1<p\leq\infty$, $q=\infty$ and $s\in\bbbr$. When $0<p<1$, we also use the function $\phi$ satisfying conditions (\ref{eq346}), (\ref{eq422}), (\ref{eq423}), and $J>m_0$ is a large positive integer whose value will be determined later. Then we can use (\ref{eq719}) and (\ref{eq723}) with $0<r<p<1$ to obtain
\begin{align}\label{eq737}
&\|\psi_{2^{m_0-k}}*f\|_{L^p(\bbbr^n)} \nonumber\\
&\lesssim\sum_{l=1}^M\esssup_{\tau\in[1,2]}\esssup_{\theta\in C_l}
          \|\iFT_n[\FT_n\psi(2^{m_0-k}\xi) \nonumber\\
&\quad\cdot\rho_l(2^{m_0-k}\xi)\FT_n\phi(2^{m_0-J-k}\xi)\FT_n f]\|_{L^p(\bbbr^n)} \nonumber\\
&\lesssim 2^{Jn(\frac{1}{r}-1)}\sum_{l=1}^M\esssup_{\tau\in[1,2]}\esssup_{\theta\in C_l}
          \|\HLmax_n(|\varDelta^L_{2^{-k}\tau\theta}(\phi_{2^{m_0-k-J}}*f)|^r)^{\frac{1}{r}}\|_{L^p(\bbbr^n)} \nonumber\\
&\lesssim 2^{Jn(\frac{1}{r}-1)}\esssup_{\tau\in[1,2]}\esssup_{\theta\in\unitsph}
          \|\varDelta^L_{2^{-k}\tau\theta}(\phi_{2^{m_0-k-J}}*f)\|_{L^p(\bbbr^n)} \nonumber\\
&=2^{Jn(\frac{1}{r}-1)}\esssup_{h\in A_k}\|\varDelta^L_h(\phi_{2^{m_0-k-J}}*f)\|_{L^p(\bbbr^n)}.
\end{align}
Insert (\ref{eq737}) into $\|f\|_{\Bspinf}$, recall (\ref{eq500}) and assume the validity of decomposition for now, then we can estimate $\|f\|_{\Bspinf}$ from above by the sum of the following two terms,
\begin{equation}\label{eq738}
2^{Jn(\frac{1}{r}-1)}\esssup_{h\in\bbbr^n}|h|^{-s}\|\varDelta^L_h f\|_{L^p(\bbbr^n)}<\infty
\end{equation}
and
\begin{equation}\label{eq739}
2^{Jn(\frac{1}{r}-1)}\esssup_{k\in\bbbz}\esssup_{h\in A_k}2^{ks}
\|\varDelta^L_h(\sum_{j=J+1}^{\infty}f_{k+j-m_0})\|_{L^p(\bbbr^n)}.
\end{equation}
To estimate (\ref{eq739}), we use (\ref{eq708}) and begin with the following
\begin{align}\label{eq740}
&2^{Jn(\frac{1}{r}-1)}\esssup_{k\in\bbbz}\esssup_{h\in A_k}2^{ks}
\|\sum_{j=J+1}^{\infty}|\varDelta^L_h f_{k+j-m_0}|\|_{L^p(\bbbr^n)} \nonumber\\
&\lesssim 2^{Jn(\frac{1}{r}-1)}\esssup_{k\in\bbbz}\esssup_{h\in A_k}2^{ks}
          (\sum_{j=J+1}^{\infty}\|\varDelta^L_h f_{k+j-m_0}\|_{L^p(\bbbr^n)}^p)^{\frac{1}{p}} \nonumber\\
&\lesssim 2^{Jn(\frac{1}{r}-1)}\esssup_{k\in\bbbz}
          (\sum_{j=J+1}^{\infty}2^{(m_0-j)sp}\cdot 2^{(k+j-m_0)sp}\|f_{k+j-m_0}\|_{L^p(\bbbr^n)}^p)^{\frac{1}{p}} \nonumber\\
&\lesssim 2^{J[n(\frac{1}{r}-1)-s]}\|f\|_{\Bspinf}.
\end{align}
The assumption $f\in\Bspinf$ implies $\sum_{j=J+1}^{\infty}|\varDelta^L_h f_{k+j-m_0}(x)|<\infty$ for every $k\in\bbbz$ and almost every $h\in A_k,x\in\bbbr^n$. In conjunction with (\ref{eq514}), (\ref{eq500}) and (\ref{eq738}), we have shown (\ref{eq730}) is true not only in the sense of $\Sw'(\bbbr^n)$ but also for every $k\in\bbbz$ and almost every $h\in A_k,x\in\bbbr^n$ when $0<p<1$, $q=\infty$ and $\sigma_p<s<\infty$. Furthermore estimating $\|f\|_{\Bspinf}$ from above by the sum of (\ref{eq738}) and (\ref{eq739}) is justified, moreover (\ref{eq739}) can be estimated from above by the first line of (\ref{eq740}) and hence by the last line of (\ref{eq740}). We have obtained the inequality
\begin{align}\label{eq742}
\|f\|_{\Bspinf}&\leq C'2^{Jn(\frac{1}{r}-1)}\esssup_{h\in\bbbr^n}|h|^{-s}\|\varDelta^L_h f\|_{L^p(\bbbr^n)} \nonumber\\
&\quad+C''2^{J[n(\frac{1}{r}-1)-s]}\|f\|_{\Bspinf}.
\end{align}
The condition $\sigma_p<s<\infty$ indicates $n(\frac{1}{r}-1)-s<0$ when $r$ is sufficiently close to $p$. Thus the coefficient $C''2^{J[n(\frac{1}{r}-1)-s]}$ is less than $\frac{1}{2}$ when the positive integer $J$ is sufficiently large, and then we can shift the second term on the right side of (\ref{eq742}) to its left side and prove the desired inequality (\ref{eq701}). The proof of Theorem \ref{theorem8} (iv) is now complete.

Finally, we come to the proof of Theorem \ref{theorem8} (v). By using a different method, some optimal conditions can be obtained in the case $p=1$. We assume the right sides of (\ref{eq702}) and (\ref{eq703}) are finite. If $1\leq q\leq\infty$ and $-n<s<\infty$, we use the radial Schwartz function $g$ satisfying (\ref{eq297}), (\ref{eq298}), (\ref{eq299}), and the kernel $G_k(\cdot)=\sum^L_{j=1}d_j g_{2^{-k}j}(\cdot)$ satisfying (\ref{eq300}), (\ref{eq301}), (\ref{eq302}) with $0<r<p=1$. Then from the inequality (\ref{eq553}), we deduce
\begin{align}\label{eq743}
2^{ks}\|G_k*f\|_{L^1(\bbbr^n)}
&\lesssim\sum_{l\geq 0}2^{-l(n+s)}\cdot 2^{(k+l)s}\int_{A_0}\|\varDelta^L_{2^{-k-l}h}f\|_{L^1(\bbbr^n)}dh \nonumber\\
&\quad+\sum_{l<0}2^{l(N'-n-s)}\cdot 2^{(k+l)s}\int_{A_0}\|\varDelta^L_{2^{-k-l}h}f\|_{L^1(\bbbr^n)}dh.
\end{align}
When $1\leq q<\infty$, we use Minkowski's inequality for $\|\cdot\|_{l^q}$-norm, H\"{o}lder's inequality and compute as in (\ref{eq306}) and (\ref{eq307}) with $N'>n+s>0$ then we can obtain
\begin{align}\label{eq744}
&\|\{2^{ks}G_k*f\}_{k\in\bbbz}\|_{l^q(L^1)} \nonumber\\
&\lesssim(\sum_{k\in\bbbz}2^{ksq}\int_{A_0}\|\varDelta^L_{2^{-k}h}f\|_{L^1(\bbbr^n)}^q dh)^{\frac{1}{q}} \nonumber\\
&\lesssim(\int_{\bbbr^n}|h|^{-sq}\|\varDelta^L_{h}f\|_{L^1(\bbbr^n)}^q\frac{dh}{|h|^n})^{\frac{1}{q}}.
\end{align}
When $q=\infty$, we use Minkowski's inequality for $\|\cdot\|_{l^{\infty}}$-norm and the inequality (\ref{eq743}) with $N'>n+s>0$ to obtain
\begin{align}\label{eq745}
&\|\{2^{ks}G_k*f\}_{k\in\bbbz}\|_{l^{\infty}(L^1)} \nonumber\\
&\lesssim\esssup_{k\in\bbbz}2^{ks}\int_{A_0}\|\varDelta^L_{2^{-k}h}f\|_{L^1(\bbbr^n)}dh \nonumber\\
&\lesssim\esssup_{k\in\bbbz}2^{ks}\esssup_{h\in A_0}\|\varDelta^L_{2^{-k}h}f\|_{L^1(\bbbr^n)} \nonumber\\
&\lesssim\esssup_{h\in\bbbr^n}|h|^{-s}\|\varDelta^L_{h}f\|_{L^1(\bbbr^n)}.
\end{align}
Indicated by the first line of (\ref{eq553}), (\ref{eq743}), (\ref{eq744}), and (\ref{eq745}), we know that the integral on the left end of (\ref{eq299}) is absolutely convergent and hence well-defined for every $k\in\bbbz$ and almost every $x\in\bbbr^n$. By (\ref{eq301}), (\ref{eq302}), Remark \ref{remark12} and the mapping property of Hardy-Littlewood maximal function, we have $\|f\|_{\dot{B}^s_{1,q}(\bbbr^n)}\lesssim\|\{2^{ks}G_k*f\}_{k\in\bbbz}\|_{l^q(L^1)}$ for $1\leq q\leq\infty$ and any $s\in\bbbr$. In conjunction with (\ref{eq744}) and (\ref{eq745}), we have proven the inequality (\ref{eq703}) and the inequality (\ref{eq702}) when $p=1$, $1\leq q<\infty$ and $-n<s<\infty$. To prove (\ref{eq702}) is true for $p=1$, $0<q<1$ and $0<s<\infty$, we notice that by picking $0<r<p=1$, the method given for the proof of the second part of Theorem \ref{theorem8} (ii) still applies and $\sigma_1=0$. The proof of Theorem \ref{theorem8} is now complete.
\end{proof}

%The acknowledgments section should not be numbered.
\section*{Acknowledgments}
The author would like to thank the anonymous reviewer for the review of this article.

\bibliographystyle{plain}
\bibliography{bibliography}
{\textit{E-mail address:} liw87@pitt.edu, lifeng.wang.1987@gmail.com}
\end{document}